\documentclass[a4paper]{amsart}                

\usepackage[latin1]{inputenc}                  
\usepackage[T1]{fontenc}                       
\usepackage[spanish,english]{babel}            
\usepackage{amsmath,amssymb,amsthm}            
\usepackage{latexsym}                          
\usepackage{delarray}                          
\usepackage{bbm}                               
\usepackage{hyperref}                          
\usepackage[pdftex,usenames,dvipsnames]{color} 
\usepackage{bbding,trfsigns}                   
\usepackage{wasysym}                           
\usepackage{datetime}                          

\usepackage{tikz}                   

\DeclareMathAlphabet{\mathpzc}{OT1}{pzc}{m}{it} 

\newtheorem{Th}{Theorem}[section]              
\newtheorem{Cor}[Th]{Corollary}
\newtheorem{Rem}[Th]{Remark}

\newtheorem{Prop}[Th]{Proposition}
\newtheorem{Lem}[Th]{Lemma}

\DeclareMathOperator{\supp}{supp}

\title[Anisotropic Hardy-Lorentz spaces with variable exponents]
      {Anisotropic Hardy-Lorentz spaces with variable exponents}

\author[V. Almeida]{V. Almeida}
\author[J.J. Betancor]{J.J. Betancor}
\author[L. Rodr\'{\i}guez-Mesa]{L. Rodr\'{\i}guez-Mesa}

\address{\newline
        V\'{\i}ctor Almeida, Jorge J. Betancor, Lourdes Rodr\'{\i}guez-Mesa \newline
        Departamento de An\'alisis Matem\'atico,
        Universidad de la Laguna, \newline
        Campus de Anchieta, Avda. Astrof\'{\i}sico Francisco S\'anchez, s/n, \newline
        38271, La Laguna (Sta. Cruz de Tenerife), Spain}
\email{valmeida@ull.es, jbetanco@ull.es, lrguez@ull.es}



\thanks{The authors are partially supported by MTM2013-44357-P}

\begin{document}

  \footnotetext{Date: \today.}

  \maketitle                

\begin{abstract}
    In this paper we introduce Hardy-Lorentz spaces with variable exponents associated to dilations in ${\Bbb R}^n$. We establish maximal characterizations and atomic decompositions for our variable exponent anisotropic Hardy-Lorentz spaces.
\end{abstract}

\section{Introduction} \label{sec:intro}
The celebrated Fefferman and Stein's paper \cite{FS} has been crucial in the development of the real variable theory of Hardy spaces. In \cite{FS} the tempered distributions in the Hardy spaces $H^p({\Bbb R}^n)$ were characterized as those ones such that certain maximal functions are in $L^p({\Bbb R}^n)$. Coifman \cite{Co} and Latter \cite{Lat} obtained atomic decompositions of the elements of the Hardy spaces $H^p({\Bbb R}^n)$. Here, $0<p<\infty$ and $H^p({\Bbb R}^n)=L^p({\Bbb R}^n)$ provided that $1<p<\infty$.

Many authors have investigated Hardy spaces in several settings. Some generalizations substitute the underlying domain ${\Bbb R}^n$ by other ones (see, for instance, \cite{BD}, \cite{CKS}, \cite{CoWe}, \cite{MaSe}, \cite{Stri} and \cite{Tol}). Also, Hardy spaces associated with operators have been defined (see \cite{DL}, \cite{DY1}, \cite{HLMMY}, \cite{HM1} and \cite{Ya}, amongst others). If $X$ is a function space, the Hardy space $H({\Bbb R}^n,X)$ on ${\Bbb R}^n$ modeled on $X$ consists of all those tempered distributions $f$ on ${\Bbb R}^n$ such that the maximal function ${\mathcal M}(f)$ of $f$ is in $X$. The definition of the maximal operator ${\mathcal M}$ will be shown below. The classical Hardy space $H^p({\Bbb R}^n)$ is the Hardy space on ${\Bbb R}^n$ modeled on $L^p({\Bbb R}^n)$. If $\nu$ is a weight on ${\Bbb R}^n$ and  $L^p({\Bbb R}^n,\nu)$ denotes the weighted Lebesgue space, the Hardy space $H({\Bbb R}^n,L^p({\Bbb R}^n,\nu))$ was investigated in \cite{GC}. The Hardy space $H({\Bbb R}^n,L^{p,q}({\Bbb R}^n))$ where $L^{p,q}({\Bbb R}^n)$ represents the Lorentz space has been studied in \cite{AT1}, \cite{FRS}, \cite{FSo}, \cite{GH} and \cite{He}. The Hardy space $H({\Bbb R}^n,{\Lambda}^p(\phi))$ on ${\Bbb R}^n$ modeled on generalized Lorentz space ${\Lambda}^p(\phi)$ was studied by Almeida and Caetano \cite{AC}. The variable exponent Hardy space $H^{p(\cdot)}({\Bbb R}^n)$, investigated by \cite{CrW}, \cite{NS}, \cite{Sa} and \cite{ZYL}, is the space $H({\Bbb R}^n, L^{p(\cdot)}({\Bbb R}^n))$ on ${\Bbb R}^n$ modeled on the variable exponent Lebesgue space $L^{p(\cdot)}({\Bbb R}^n)$.

By $S({\Bbb R}^n)$, as usual, we denote the Schwartz function class on ${\Bbb R}^n$ and by $S'({\Bbb R}^n)$ its dual space. If $\varphi\in S({\Bbb R}^n)$, the radial maximal function ${\mathcal M}=M_{\varphi}$ used to characterized Hardy spaces is defined by
$${\mathcal M}(f)=\sup_{t>0}|f*\varphi_t|,\;\;\;f\in S'({\Bbb R}^n),$$
where $\varphi_t(x)=t^{-n}\varphi(x/t)$, $x\in{\Bbb R}^n$ and $t>0$. Bownik \cite{Bow1} studied anisotropic Hardy spaces on ${\Bbb R}^n$ associated with dilations in ${\Bbb R}^n$. If $A$ is an expansive dilation matrix  in ${\Bbb R}^n$, that is, a $n\times n$ real matrix such that $\displaystyle\min_{\lambda\in\sigma(A)}|\lambda|>1$ where $\sigma(A)$ represents the set of eigenvalues of $A$, for every $k\in\Bbb Z$, we define
$${\varphi}_{A,k}(x)=|{\rm det}\;A|^{-k}\varphi(A^{-k}x),\;\;x\in{\Bbb R}^n,$$
and the maximal function ${\mathcal M}_A=M_{A,\varphi}$ associated with $A$ is given by
$${\mathcal M}_A(f)=\sup_{k\in\Bbb Z}|f*\varphi_{A,k}|,\;\;\;f\in S'({\Bbb R}^n).$$
Bownik \cite{Bow1} characterizes anisotropic Hardy spaces by maximal functions like ${\mathcal M}_A$. Recently,  Liu, Yang, and Yuan \cite{LYY} have extended Bownik's results by studying anisotropic Hardy spaces on $\mathbb{R}^n$ modeled on Lorentz spaces $L^{p,q}(\mathbb{R}^n)$.

Ephremidze, Kokilashvili and Samko \cite{EKS} introduced variable exponent Lorentz Spaces ${\mathcal L}^{p(\cdot),q(\cdot)}({\Bbb R}^n)$. In this paper  we define anisotropic Hardy spaces on ${\Bbb R}^n$ associated with a dilation $A$ modelled on ${\mathcal L}^{p(\cdot),q(\cdot)}({\Bbb R}^n)$. These Hardy spaces are represented by $H^{p(\cdot),q(\cdot)}({\Bbb R}^n,A)$ and they are called variable exponent anisotropic Hardy-Lorentz spaces on ${\Bbb R}^n$. We characterize the tempered distributions in $H^{p(\cdot),q(\cdot)}({\Bbb R}^n,A)$ by using anisotropic maximal function ${\mathcal M}_A$. Also, we obtain atomic decompositions for the elements of $H^{p(\cdot),q(\cdot)}({\Bbb R}^n,A)$. Our results extend those ones in \cite{LYY} to variable exponent setting.

Before establishing the results of this paper we recall the definitions and properties about anisotropy and  variable exponent Lebesgue and Lorentz spaces we will need.

An exhaustive and systematic study about variable exponent Lebesgue spaces $L^{p(\cdot)}(\Omega)$ where $\Omega\subset{\Bbb R}^n$ can be found in the monograph \cite{CF} and in \cite{DHHR}. Here $p:\Omega\rightarrow(0,\infty)$ is a measurable function. We assume that $0<p_-(\Omega)\leq p_+(\Omega)<\infty$, where $\displaystyle p_-(\Omega)={\rm essinf}_{x\in\Omega}p(x)$ and $\displaystyle p_+(\Omega)={\rm essup}_{x\in\Omega}p(x)$. The space $L^{p(\cdot)}(\Omega)$ is the collection of all measurable functions $f$ such that, for some $\lambda>0$, $\rho(f/\lambda)<\infty$, where
$$\rho(f)=\rho _{p(\cdot )}(f)=\int_\Omega{|f(x)|^{p(x)}dx}.$$
We define $\|\cdot\|_{p(\cdot)}$ as follows
$$\|f\|_{p(\cdot)}
    =\inf\left\{\lambda>0\ : \ \int_{\Omega}{\left({{|f(x)|}\over{\lambda}}\right)^{p(x)}dx}\leq 1\right\},\;\;\;f\in L^{p(\cdot)}(\Omega).$$

If $p_-(\Omega)\geq 1$, then $\|\cdot\|_{p(\cdot)}$ is a norm and $(L^{p(\cdot)}(\Omega), \|\cdot\|_{p(\cdot)})$ is a Banach space. However,  if $p_-(\Omega)< 1$, then $\|\cdot\|_{p(\cdot)}$ is a quasinorm and $(L^{p(\cdot)}(\Omega), \|\cdot\|_{p(\cdot)})$ is a quasi Banach space.

A crucial problem concerning to variable exponent Lebesgue spaces is to describe the exponents $p$ for which the Hardy-Littlewood maximal function is bounded in $L^{p(\cdot)}({\Bbb R}^n)$ (see, \cite{CDF}, \cite{Die1}, \cite{Ne}, and \cite{sam}, amongst others). As it is shown in \cite{CFMP}, \cite{CrW2} and \cite{GK}, the boundedness of the Hardy-Littlewood maximal function, together with extensions of Rubio de Francia's extrapolation theorem, lead to the boundedness of a wide class of operators and vector-valued inequalities on $L^{p(\cdot)}({\Bbb R}^n)$ and the weighted $L^{p(\cdot)}(\nu)$. These ideas also work in the variable exponent Lorentz spaces, introduced by Ephremidze, Kokilashvili and Samko \cite{EKS} and they will play a fundamental role in the proof of some of our main results.

The Lorentz spaces were introduced in \cite{Lo1} and \cite{Lo2} as a generalization of classical Lebesgue spaces. The theory of Lorentz spaces can be encountered in \cite{BS} and \cite{CRS}. Assume that $f$ is a measurable function. We define the distribution function $\mu_f:[0,\infty)\rightarrow[0,\infty]$ associated with $f$ by
$$\mu_f(s)=|\{x\in{\Bbb R}^n:\;|f(x)|>s\}|,\;\;\;s\in[0,\infty).$$
Here, $|E|$ denotes the Lebesgue measure of $E$, for every Lebesgue measurable set $E$. The non-increasing equimeasurable rearrangement $f^*:[0,\infty)\rightarrow[0,\infty]$ of $f$ is defined by
$$f^*(t)=\inf\{s\geq 0:\;\mu_f(s)\leq t\},\;\;\;t\in[0,\infty).$$
If $0<p,q<\infty$, the measurable function $f$ is in the Lorentz space $L^{p,q}({\Bbb R}^n)$ provided that
$$\|f\|_{L^{p,q}({\Bbb R}^n)}=\left(\int_0^{\infty}t^{\frac{q}{p}-1}(f^*(t))^qdt\right)^{1/q}<\infty.$$
$L^{p,q}({\Bbb R}^n)$ is complete and it is normable, that is, there exists a norm equivalent with the quasinorm $\|\cdot\|_{L^{p,q}({\Bbb R}^n)}$ (see \cite[p. 66]{CRS}) for $1<p<\infty$ and $1\leq q<\infty$.

Variable exponent Lorentz spaces have been defined in two different ways: one of them by Ephremidze, Kokilashvili and Samko \cite{EKS} and the other one  by Kempka and Vyb\'{\i}ral \cite{KV}.

In this paper we consider the space defined in \cite{EKS}. This election is motivated by the following fact. We need to use a vectorial inequality for the anisotropic Hardy-Littlewood maximal function (see Proposition \ref{Prop1.3}). In order to prove this property we use an extrapolation argument requering us to know the associated K\"othe dual space of the Lorentz space. The dual space of the variable exponent Lorentz space in \cite{EKS} is known (\cite[Lemma 2.7]{EKS}). However, characterizations of the dual space of the variable exponent Lorentz space in \cite{KV} have not been established (see Remark \ref{Rem1.1}).

For every $a\geq 0$ we denote by ${\mathfrak P}_a$ the set of measurable functions $p:(0,\infty)\rightarrow(0,\infty)$ such that $a<p_-((0,\infty))\leq p_+((0,\infty))<\infty$. By ${\Bbb P}$ we represent the class of bounded measurable functions $p:(0,\infty)\rightarrow(0,\infty)$ such that there exist the limits
$$p(0)=:\lim_{t\rightarrow 0^+}p(t)\;\;\mbox{and}\;\;p(\infty)=:\lim_{t\rightarrow +\infty}p(t)$$
and the following log-H\"older continuity conditions are satisfied:

$\begin{array}{ll}
 & |p(t)-p(0)|\leq\displaystyle{C\over{|\ln\,t|}},\;\;\mbox{for}\;\;0<t\leq 1/2 ,\\
 \noindent\mbox{and} & \\
 & |p(t)-p(\infty)|\leq\displaystyle{C\over{\ln(e+t)}},\;\;\mbox{for}\;\;t\in (0,\infty).
\end{array}$

\noindent We also denote ${\Bbb P}_a={\Bbb P}\cap{\mathfrak P}_a$, for every $a\geq 0$.

Let $p,q\in {\mathfrak P}_0$. We represent by $(p(\cdot),q(\cdot))$-Lorentz space ${\mathcal L}^{p(\cdot),q(\cdot)}({\Bbb R}^n)$ the space of all those measurable functions $f$ on ${\Bbb R}^n$ such that $t^{{1\over{p(t)}}-{1\over{q(t)}}}f^*(t)\in L^{q(\cdot)}(0,\infty)$. We define
$$\|f\|_{{\mathcal L}^{p(\cdot),q(\cdot)}({\Bbb R}^n)}=\|t^{{1\over{p(t)}}-{1\over{q(t)}}}f^*(t)\|_{L^{q(\cdot)}(0,\infty)},\;\;\;f\in {\mathcal L}^{p(\cdot),q(\cdot)}({\Bbb R}^n).$$
We also consider the average $f^{**}$ of $f^*$ given by
$$f^{**}(t)={1\over t}\int_0^t{f^*(s)ds},\;\;\;t\in(0,\infty),$$
and
$$\|f\|_{{\mathcal L}^{p(\cdot),q(\cdot)}({\Bbb R}^n)}^{(1)}=\|t^{{1\over{p(t)}}-{1\over{q(t)}}}f^{**}(t)\|_{L^{q(\cdot)}(0,\infty)},\;\;\;f\in {\mathcal L}^{p(\cdot),q(\cdot)}({\Bbb R}^n).$$
$\|\cdot\|_{{\mathcal L}^{p(\cdot),q(\cdot)}({\Bbb R}^n)}^{(1)}$ satisfies the triangular inequality provided that $q_-((0,\infty))\geq 1$. It is clear that $\|f\|_{{\mathcal L}^{p(\cdot),q(\cdot)}({\Bbb R}^n)}\leq \|f\|_{{\mathcal L}^{p(\cdot),q(\cdot)}({\Bbb R}^n)}^{(1)}$. According to \cite[Theorem 2.4]{EKS} if $p\in{\Bbb P}_0$, $q\in{\Bbb P}_1$, $p(0)>1$, and $p(\infty)>1$, there exists $C>0$ for which
$$\|f\|_{{\mathcal L}^{p(\cdot),q(\cdot)}({\Bbb R}^n)}^{(1)}\leq C\|f\|_{{\mathcal L}^{p(\cdot),q(\cdot)}({\Bbb R}^n)},\;\;\;f\in{\mathcal L}^{p(\cdot),q(\cdot)}({\Bbb R}^n).$$
If $p,q\in{\Bbb P}_1$, ${\mathcal L}^{p(\cdot),q(\cdot)}({\Bbb R}^n)$ is a Banach function space (in the sense of \cite{BS}) and the dual space $({\mathcal L}^{p(\cdot),q(\cdot)}({\Bbb R}^n))'$ coincides with ${\mathcal L}^{p'(\cdot),q'(\cdot)}({\Bbb R}^n)$ \cite[Lemma 2.7 and Theorem 2.8]{EKS}. Here as usual if $r:(0,\infty)\rightarrow(1,\infty)$, $r'={r\over{r-1}}$. The behaviour of the anisotropic Hardy-Littlewood maximal function on ${\mathcal L}^{p(\cdot),q(\cdot)}({\Bbb R}^n)$ will be very useful in the sequel. According to \cite[Theorem 3.12]{EKS}, the classical Hardy-Littlewood maximal operator is bounded from ${\mathcal L}^{p(\cdot),q(\cdot)}({\Bbb R}^n)$ into itself provided that $p,q\in{\Bbb P}_1$.

The main definitions and properties about the anisotropic setting we will use in this paper can be found in \cite{Bow1}.

Suppose that $A$ is an expansive dilation matrix  in ${\Bbb R}^n$. We say that a measurable function $\rho:{\Bbb R}^n\rightarrow[0,\infty)$ is a homogeneous quasinorm associated with $A$ when the following properties hold:
\begin{enumerate}
\item[(a)] $\rho(x)=0$ if, and only if, $x=0$;
\item[(b)] $\rho(Ax)=|{\rm det}\;A|\rho(x),\;\;x\in{\Bbb R}^n$;
\item[(c)] $\rho(x+y)\leq H(\rho(x)+\rho(y)),\;\;x,y\in{\Bbb R}^n$, for certain $H\geq 1$.
\end{enumerate}
If $P$ is a nondegenerate $n\times n$ matrix, the set $\Delta$ defined by
$$\Delta=\left\{x\in{\Bbb R}^n:\;|Px|<1\right\}$$
is called the ellipsoid generated by $P$. According to \cite[Lemma 2.2, p. 5]{Bow1} there exists an ellipsoid $\Delta$ with Lebesgue measure $1$ and such that, for certain $r_0>1$,
$$\Delta\subseteq r_0\Delta\subseteq A\Delta.$$
>From now on, the ellipsoid $\Delta$ satisfying the above properties is fixed. For every $k\in\Bbb Z$, we define $B_k=A^k\Delta$, as the equivalent of the Euclidean balls in our anisotropic context,  and denote by $\omega$ the smallest integer such that $2B_0\subset B_{\omega}$. We have that, for every $k\in\Bbb Z$, $|B_k|=b^k$ , where $b=|{\rm det}\;A|$, and $B_k\subset r_0B_k\subset B_{k+1}$.

The step quasinorm $\rho_A$ on ${\Bbb R}^n$ is defined by
$$\rho_A(x)=\left\{\begin{array}{ll}
b^k, & x\in B_{k+1}\backslash B_k,\;\;k\in\Bbb Z, \\
0, & x=0.
\end{array}\right.
$$
Thus, $\rho_A$ is a homogeneous quasinorm associated with $A$.

By \cite[Lemma 2.4, p. 6]{Bow1} if $\rho$ is any quasinorm associated with $A$, then $\rho_A$ and $\rho$ are equivalent, that is, for a certain $C>0$,
$$\rho(x)/C\leq\rho_A(x)\leq C\rho(x),\;\;\;x\in{\Bbb R}^n.$$
The triplet $({\Bbb R}^n,\rho_A,|\cdot|)$, where $|\cdot|$ denotes the Lebesgue measure in ${\Bbb R}^n$, is a space of homogeneous type in the sense of Coifman and Weiss \cite{CW}.

We now define maximal functions in our anisotropic setting. Suppose that $\varphi\in S({\Bbb R}^n)$ and $f\in S'({\Bbb R}^n)$. The radial maximal function $M_{\varphi}^0(f)$ of $f$ with respect to $\varphi$ is defined by
$$M_{\varphi}^0(f)(x)=\sup_{k\in\Bbb Z}|(f*{\varphi}_k)(x)|,$$
where ${\varphi}_k(x)=b^{-k}\varphi(A^{-k}x),\;\;k\in\Bbb Z$ and $x\in{\Bbb R}^n$. Since the matrix $A$ is fixed we do not refer it in the notation of maximal functions.

The nontangential maximal function $M_{\varphi}(f)$ with respect to $\varphi$ is given by
$$M_{\varphi}(f)(x)=\sup_{k\in\Bbb Z\;,\;y\in x+B_k}|(f*{\varphi}_k)(y)|,\;\;\;x\in{\Bbb R}^n.$$

If $\alpha=(\alpha_1,...,\alpha_n)\in{\Bbb N}^n$, we write $|\alpha|=\alpha_1+...+\alpha_n$. Let $N\in \Bbb N$. We consider the set
$$S_N=\{\varphi\in S({\Bbb R}^n):\;\sup_{x\in{\Bbb R}^n}(1+\rho_A(x))^N|D^{\alpha}\varphi(x)|\leq 1,\,\;\alpha\in{\Bbb N}^n\;\mbox{and}\;|\alpha|\leq N\}.$$
Here $D^{\alpha}=\displaystyle{{\partial^{|\alpha |}}\over{\partial x_1^{\alpha_1}...\partial x_n^{\alpha_n}}}$, when $\alpha=(\alpha_1,...,\alpha_n)\in{\Bbb N}^n$.

The radial grandmaximal function $M_N^0(f)$ of $f$ of order $N$ is defined by
$$M_N^0(f)=\sup_{\varphi\in S_N}M_{\varphi}^0(f).$$

The nontangential grandmaximal function $M_N(f)$ of $f$ of order $N$ is given by
$$M_N(f)=\sup_{\varphi\in S_N}M_{\varphi}(f).$$

We now define variable exponent anisotropic Hardy-Lorentz spaces. Let $N\in\Bbb N$ and $p,q\in {\mathfrak P}_0$. The $(p(\cdot),q(\cdot))$-anisotropic Hardy-Lorentz space $H^{p(\cdot),q(\cdot)}_{N}({\Bbb R}^n,A)$ associated with $A$ is the set of all those $f\in S'(\Bbb{R}^n)$ such that $M_N(f)\in{\mathcal L}^{p(\cdot),q(\cdot)}({\Bbb R}^n)$. On $H^{p(\cdot),q(\cdot)}_{N}({\Bbb R}^n,A)$ we consider the quasinorm $\|\cdot\|_{H^{p(\cdot),q(\cdot)}_{N}({\Bbb R}^n,A)}$ defined by
$$\|f\|_{H^{p(\cdot),q(\cdot)}_{N}({\Bbb R}^n,A)}=\|M_N(f)\|_{{\mathcal L}^{p(\cdot),q(\cdot)}({\Bbb R}^n)},\;\;\;f\in H^{p(\cdot),q(\cdot)}_{N}({\Bbb R}^n,A) .$$

Our first result establishes that the space $H^{p(\cdot),q(\cdot)}_N({\Bbb R}^n,A)$ actually does not depend on $N$ provided that $N$ is large enough. Furthermore, we prove that $H^{p(\cdot),q(\cdot)}_N({\Bbb R}^n,A)$ can be characterized also by using the maximal functions $M_{\varphi}^0$, $M_{\varphi}$ and $M_N^0$.

\begin{Th}\label{Th1.1}
Let $f\in S'({\Bbb R}^n)$ and $\varphi\in S({\Bbb R}^n)$ such that $\int{\varphi}\neq 0$. Assume that $p,q\in{\Bbb P}_0$. Then, the following assertions are equivalent.
\begin{enumerate}
\item[(i)] There exists $N_0\in \mathbb{N}$ such that, for every $N\geq N_0$, $f\in H^{p(\cdot),q(\cdot)}_N({\Bbb R}^n,A)$.
\item[(ii)] $M_{\varphi}(f)\in {\mathcal L}^{p(\cdot),q(\cdot)}({\Bbb R}^n)$.
\item[(iii)] $M_{\varphi}^0(f)\in {\mathcal L}^{p(\cdot),q(\cdot)}({\Bbb R}^n)$.
\end{enumerate}
Moreover, for every $g\in S'({\Bbb R}^n)$ the quantities $\|M_N(g)\|_{{\mathcal L}^{p(\cdot),q(\cdot)}({\Bbb R}^n)}$, $N\geq N_0$, $\|M_{\varphi}^0(g)\|_{{\mathcal L}^{p(\cdot),q(\cdot)}({\Bbb R}^n)}$ and $\|M_{\varphi}(g)\|_{{\mathcal L}^{p(\cdot),q(\cdot)}({\Bbb R}^n)}$ are equivalent.
\end{Th}

According to Theorem \ref{Th1.1} we write $H^{p(\cdot),q(\cdot)}({\Bbb R}^n,A)$ to denote $H^{p(\cdot),q(\cdot)}_N({\Bbb R}^n,A)$, for every $N\geq N_0$.

In order to prove this theorem we follow the ideas developed by Bownik \cite[\S 7]{Bow1} (see also \cite[\S 4]{LYY}) but we need to make some modifications due to that decreasing rearrangement and variable exponents appear.

Let $1<r\leq \infty$, $s\in \mathbb{N}$ and  $p,q\in {\mathfrak P}_0$. We say that a measurable function $a$ on ${\Bbb R}^n$ is a $(p(\cdot),q(\cdot),r,s)$-atom associated with $x_0\in{\Bbb R}^n$ and $k\in\Bbb Z$ when $a$ satisfies

\begin{enumerate}
\item[(a)] $\mbox{supp}\; a\subseteq x_0+B_k$.
\item[(b)] $\|a\|_r\leq b^{k/r}\|\chi_{x_0+B_k}\|^{-1}_{{\mathcal L}^{p(\cdot),q(\cdot)}({\Bbb R}^n)}$. (Note that $(\chi_{x_0+B_k})^*=\chi _{(0,b^k)}$).
\item[(c)] $\displaystyle\int_{{\Bbb R}^n}a(x)x^{\alpha}dx=0$, for every $\alpha\in{\Bbb N}^n$ such that $|\alpha|\leq s$.
\end{enumerate}
Here, if $\alpha=(\alpha_1,...,\alpha_n)\in\Bbb N$ and $x=(x_1,...,x_n)\in{\Bbb R}^n$, $x^{\alpha}=x_1^{\alpha_1}\cdot\cdot\cdot x_n^{\alpha_n}$.

\begin{Rem}
>From now on, any time we write $a $ is a $(p(\cdot),q(\cdot),r,s)$-atom associated with $x_0\in{\Bbb R}^n$ and $k\in\Bbb Z$, it is understood that (a), (b) and (c) hold.
\end{Rem}
In the next result we characterize the distributions in $H^{p(\cdot),q(\cdot)}({\Bbb R}^n,A)$ by atomic decompositions.

\begin{Th}\label{Th1.2}
Let $p,q\in{\Bbb P}_0$.
\begin{enumerate}
\item[(i)] There exist $s_0\in \mathbb{N}$ and $C>0$ such that if, for every $j\in\Bbb N$, $\lambda_j\geq 0$ and $a_j$ is a $(p(\cdot),q(\cdot),\infty,s_0)$-atom associated with $x_j\in{\Bbb R}^n$ and $\ell_j\in\Bbb Z$, satisfying that \newline $\displaystyle\sum_{j\in\Bbb N}\lambda_j\|\chi_{x_j+B_{\ell_j}}\|^{-1}_{{\mathcal L}^{p(\cdot),q(\cdot)}({\Bbb R}^n)}\chi_{x_j+B_{\ell_j}}\in {\mathcal L}^{p(\cdot),q(\cdot)}({\Bbb R}^n)$, then $f=\displaystyle\sum_{j\in\Bbb N}\lambda_ja_j\in H^{p(\cdot),q(\cdot)}({\Bbb R}^n,A)$ and
    $$\|f\|_{H^{p(\cdot),q(\cdot)}({\Bbb R}^n,A)}\leq C\left\|\sum_{j\in\Bbb N}\lambda_j\|\chi_{x_j+B_{\ell_j}}\|^{-1}_{{\mathcal L}^{p(\cdot),q(\cdot)}({\Bbb R}^n)}\chi_{x_j+B_{\ell_j}}\right\|_{{\mathcal L}^{p(\cdot),q(\cdot)}({\Bbb R}^n)}.$$
    If also $p(0)<q(0)$, then there exists $r_0>1$ such that for every $r_0<r<\infty$ the above assertion is true when $(p(\cdot),q(\cdot),\infty,s_0)$-atoms are replaced by $(p(\cdot),q(\cdot),r,s_0)$-atoms.
\item[(ii)] There exists $s_0\in \mathbb{N}$ such that for every $s\in \mathbb{N}$, $s\ge s_0$, and $1<r\le \infty$, we can find $C>0$ such that, for every $f\in H^{p(\cdot),q(\cdot)}({\Bbb R}^n,A)$, there exist, for each $j\in\Bbb N$, $\lambda_j>0$ and a $(p(\cdot),q(\cdot),r,s)$-atom $a_j$ associated with $x_j\in{\Bbb R}^n$ and $\ell_j\in\Bbb Z$, satisfying that $\displaystyle\sum_{j\in\Bbb N}\lambda_j\|\chi_{x_j+B_{\ell_j}}\|^{-1}_{{\mathcal L}^{p(\cdot),q(\cdot)}({\Bbb R}^n)}\chi_{x_j+B_{\ell_j}}\in {\mathcal L}^{p(\cdot),q(\cdot)}({\Bbb R}^n)$, $f=\displaystyle\sum_{j\in\Bbb N}\lambda_ja_j$ in $S'({\Bbb R}^n)$ \noindent and
    $$\left\|\sum_{j\in\Bbb N}\lambda_j\|\chi_{x_j+B_{\ell_j}}\|^{-1}_{{\mathcal L}^{p(\cdot),q(\cdot)}({\Bbb R}^n)}\chi_{x_j+B_{\ell_j}}\right\|_{{\mathcal L}^{p(\cdot),q(\cdot)}({\Bbb R}^n)}\leq C\|f\|_{H^{p(\cdot),q(\cdot)}({\Bbb R}^n,A)}.$$

\end{enumerate}
\end{Th}

Let $1<r\leq \infty$, $s\in \mathbb{N}$ and  $p,q\in {\mathfrak P}_0$. We define the anisotropic variable exponent atomic Hardy-Lorentz space $H^{p(\cdot),q(\cdot),r,s}({\Bbb R}^n,A)$ as follows. A distribution $f\in S'({\Bbb R}^n)$ is in $H^{p(\cdot),q(\cdot),r,s}({\Bbb R}^n,A)$ when,  for every $j\in\Bbb N$ there exist $\lambda_j>0$ and a $(p(\cdot),q(\cdot),r,s)$-atom  $a_j$ associated with $x_j\in{\Bbb R}^n$ and $\ell_j\in\Bbb Z$ such that $f=\displaystyle\sum_{j\in\Bbb N}\lambda_ja_j$, where the series converges in $S'({\Bbb R}^n)$, and
$$\sum_{j\in\Bbb N}\lambda_j\|\chi_{x_j+B_{\ell_j}}\|^{-1}_{{\mathcal L}^{p(\cdot),q(\cdot)}({\Bbb R}^n)}\chi_{x_j+B_{\ell_j}}\in {\mathcal L}^{p(\cdot),q(\cdot)}({\Bbb R}^n).$$

For every $f\in H^{p(\cdot),q(\cdot),r,s}({\Bbb R}^n,A)$ we define
$$\|f\|_{H^{p(\cdot),q(\cdot),r,s}({\Bbb R}^n,A)}=\inf\left\|\sum_{j\in\Bbb N}\lambda_j\|\chi_{x_j+B_{\ell_j}}\|^{-1}_{{\mathcal L}^{p(\cdot),q(\cdot)}({\Bbb R}^n)}\chi_{x_j+B_{\ell_j}}\right\|_{{\mathcal L}^{p(\cdot),q(\cdot)}({\Bbb R}^n)},$$
where the infimum is taken over all the sequences $(\lambda_j)_{j\in \mathbb{N}}\subset(0,\infty)$ and $(a_j)_{j\in \mathbb{N}}$ of $(p(\cdot),q(\cdot),r,s)$-atoms satisfying that $f=\displaystyle\sum_{j\in\Bbb N}\lambda_ja_j$ in  $ S'({\Bbb R}^n)$ and
$$\displaystyle\sum_{j\in\Bbb N}\lambda_j\|\chi_{x_j+B_{\ell_j}}\|^{-1}_{{\mathcal L}^{p(\cdot),q(\cdot)}({\Bbb R}^n)}\chi_{x_j+B_{\ell_j}}\in {\mathcal L}^{p(\cdot),q(\cdot)}({\Bbb R}^n),$$
being $a_j$ associated with $x_j\in{\Bbb R}^n$ and $\ell_j\in\Bbb Z$, for every $j\in\Bbb N$.

In Theorem \ref{Th1.2} we state some conditions in order that the inclusions $H^{p(\cdot),q(\cdot)}({\Bbb R}^n,A)\subset H^{p(\cdot),q(\cdot),r,s}({\Bbb R}^n,A)$ and  $H^{p(\cdot),q(\cdot),r,s}({\Bbb R}^n,A)\subset H^{p(\cdot),q(\cdot)}({\Bbb R}^n,A)$ hold continuously.

In our proof of Theorem \ref{Th1.2}  a vector valued inequality, involving the Hardy-Littlewood maximal function in our anisotropic setting, plays an important role. The mentioned maximal function is defined by
$$M_{HL}(f)(x)=\sup_{k\in\Bbb Z,y\in x+B_k}{1\over{b^k}}\int_{y+B_k}|f(z)|dz,\;\;\;x\in{\Bbb R}^n.$$

After proving a version of \cite[Theorem 3.12]{EKS} for $M_{HL}$, by using an extension of Rubio de Francia extrapolation Theorem (see \cite{CFMP}, \cite{CrW2}, and \cite{GK}), we can establish the following result.

\begin{Prop}\label{Prop1.3}
Assume that $p,q\in{\Bbb P}_1$. For every $r\in(1,\infty)$ there exists $C>0$ such that
$$\left\|\left(\sum_{j\in\Bbb N}M_{HL}(f_j)^r\right)^{1/r}\right\|_{{\mathcal L}^{p(\cdot),q(\cdot)}({\Bbb R}^n)}\leq C\left\|\left(\sum_{j\in\Bbb N}|f_j|^r\right)^{1/r}\right\|_{{\mathcal L}^{p(\cdot),q(\cdot)}({\Bbb R}^n)},$$
for each sequence $(f_j)_{j\in\Bbb N}$ of functions in $L^1_{loc}({\Bbb R}^n)$.
\end{Prop}

\begin{Rem}\label{Rem1.1}
We do not know if the last vectorial inequality holds when the Lorentz space ${\mathcal L}^{p(\cdot),q(\cdot)}({\Bbb R}^n)$ is replaced by the variable exponent Lorentz space $L_{p(\cdot),q(\cdot)}({\Bbb R}^n)$ introduced by Kempka and Vyb\'{\i}ral \cite{KV}. In order to apply extrapolation technique it is necessary to know the associated K\"othe dual space (see \cite[p. 25]{LT}) $(L_{p(\cdot),q(\cdot)}({\Bbb R}^n))^*$ of  $L_{p(\cdot),q(\cdot)}({\Bbb R}^n)$, but its characterization is, as far we know, an open question.
\end{Rem}

Also in order to prove Theorem \ref{Th1.2} we need to establish that $H^{p(\cdot),q(\cdot)}({\Bbb R}^n,A)\cap L^1_{loc}({\Bbb R}^n)$ is a dense subspace of $H^{p(\cdot),q(\cdot)}({\Bbb R}^n,A)$. At this point a careful study of Calder\'on-Zygmund decomposition of the distributions in $H^{p(\cdot),q(\cdot)}({\Bbb R}^n,A)$ must be made.

To establish boundedness of operators on Hardy spaces, atomic characterizations (as in Theorem \ref{Th1.2}) play an important role. Meyer \cite{Me} (see also \cite[p. 513]{MTW}) gave a function $f\in H^1({\Bbb R}^n)$ whose norm is not achieved by finite atomic decomposition. More recently, Bownik \cite{Bow2}  adapted that example to get, for every $0<p\leq 1$, an atom in $H^p({\Bbb R}^n)$ with the same property. Also in \cite[Theorem 2]{Bow2} it was proved that there exists a linear functional ${\frak l}$ defined on the space $H^{1,\infty}_{fin}({\Bbb R}^n)$, consisting in finite linear combinations of $(1,\infty)$-atoms, such that, for a certain $C>0$, $|{\frak l}(a)|\leq C$, for every $(1,\infty)$-atom $a$, and ${\frak l}$ can not be extended to a bounded functional on the whole $H^1({\Bbb R}^n)$.

Bownik's results have motivated some investigations about operators on Hardy spaces via atomic decompositions. Meda, Sj{\"o}gren and Vallarino \cite{MSV} proved that if $1<q<\infty$ and $T$ is a linear operator defined on $H^{1,q}_{fin}({\Bbb R}^n)$, the space of finite linear combinations of $(1,q)$-atoms, into a quasi Banach space $Y$ such that $\sup\{\|Ta\|_Y:\;a\;\mbox{is a}\;(1,q)\mbox{-atom}\}<\infty$, then $T$ can be extended to $ H^1({\Bbb R}^n)$ as a bounded operator from $ H^1({\Bbb R}^n)$ into $Y$. Also, it is proved that the same is true when $(1,q)$-atoms are replaced by continuous $(1,\infty)$-atoms, in contrast with the Bownik's result. Yang and Zhou \cite{YZ} established the result when $0<p\leq 1$ and $(p,2)$-atoms are considered. Ricci and Verdera \cite{RV} proved that, for $0<p<1$, when $H^{p,\infty}_{fin}({\Bbb R}^n)$ is endowed with the natural topology, the dual spaces of $H^{p,\infty}_{fin}({\Bbb R}^n)$ and $H^p({\Bbb R}^n)$ coincide.

Also, this type of results have been recently established for Hardy spaces in more general settings (see, for instance, \cite{BLYZ}, \cite{CrW}, \cite{LYY} and \cite{ZSY}).

In order to study boundedness of some singular integrals on our anisotropic Hardy-Lorentz spaces with variable exponents we consider finite atomic Hardy-Lorentz spaces in our settings.

Let $1<r< \infty$, $s\in \mathbb{N}$ and  $p,q\in {\mathfrak P}_0$. The space  $H^{p(\cdot),q(\cdot),r,s}_{fin}({\Bbb R}^n,A)$ consists of all those $f\in H^{p(\cdot),q(\cdot)}({\Bbb R}^n,A)$ such that there exist $k\in \Bbb N$ and, for every $j\in \Bbb N$, $1\leq j\leq k$, $\lambda_j>0$ and a $(p(\cdot),q(\cdot),r,s)$-atom $a_j$ such that $f=\displaystyle\sum_{j=1}^k\lambda_ja_j$. For every $f\in H^{p(\cdot),q(\cdot),r,s}_{fin}({\Bbb R}^n,A)$, we define
$$\|f\|_{H^{p(\cdot),q(\cdot),r,s}_{fin}({\Bbb R}^n,A)}=\inf\left\|\sum_{j=1}^k\lambda_j\|\chi_{x_j+B_{\ell_j}}\|^{-1}_{{\mathcal L}^{p(\cdot),q(\cdot)}({\Bbb R}^n)}\chi_{x_j+B_{\ell_j}}\right\|_{{\mathcal L}^{p(\cdot),q(\cdot)}({\Bbb R}^n)},$$
where the infimum is taken over all the finite sequences $(\lambda_j)_{j=1}^{k}\subset(0,\infty)$ and $(a_j)_{j=1}^{k}$ of $(p(\cdot),q(\cdot),r,s)$-atoms such that $f=\displaystyle\sum_{j=1}^k\lambda_ja_j$
and  being, for every $j\in\Bbb N$, $j\leq k$,  $a_j$ associated with $x_j\in{\Bbb R}^n$ and $\ell_j\in\Bbb Z$.

The space $H^{p(\cdot),q(\cdot),\infty,s}_{fin,con}({\Bbb R}^n,A)$ and the quasinorm $\|\cdot\|_{H^{p(\cdot),q(\cdot),\infty,s}_{fin,con}({\Bbb R}^n,A)}$ are defined in a similar way by considering continuous $(p(\cdot),q(\cdot),\infty,s)$-atoms.

In Theorem \ref{Th1.3} we establish some conditions that imply that  $H^{p(\cdot),q(\cdot),r,s}_{fin}({\Bbb R}^n,A)$ is dense in $H^{p(\cdot),q(\cdot)}({\Bbb R}^n,A)$.

\begin{Th}\label{Th1.3}
Let $p,q\in {\Bbb P}_0$.
\begin{enumerate}
\item[(i)] Assume that $p(0)<q(0)$. Then, there exist $s_0\in\Bbb N$ and $r_0\in(1,\infty)$ such that, for every  $s\in\Bbb N$, $s\geq s_0$, and $r\in(r_0,\infty)$, $\|\cdot\|_{H^{p(\cdot),q(\cdot),r,s}_{fin}({\Bbb R}^n,A)}$ and $\|\cdot\|_{H^{p(\cdot),q(\cdot)}({\Bbb R}^n,A)}$ are equivalent quasinorms in $H^{p(\cdot),q(\cdot),r,s}_{fin}({\Bbb R}^n,A)$.
\item[(ii)]  There exists $s_0\in\Bbb N$ such that, for every $s\geq s_0$, $\|\cdot\|_{H^{p(\cdot),q(\cdot),\infty,s}_{fin,con}({\Bbb R}^n,A)}$ and $\|\cdot\|_{H^{p(\cdot),q(\cdot)}({\Bbb R}^n,A)}$ are equivalent quasinorms in $H^{p(\cdot),q(\cdot),\infty,s}_{fin,con}({\Bbb R}^n,A)$.
\end{enumerate}
\end{Th}

As application of Theorem \ref{Th1.3} we prove that convolutional type Calder\'on-Zygmund singular integrals are bounded in $H^{p(\cdot),q(\cdot)}({\Bbb R}^n,A)$. A precise definition of the singular integral that we consider can be found in Section 6.

\begin{Th}\label{Th1.4}
Let $p,q\in {\Bbb P}_0$. Assume that $p(0)<q(0)$. If $T$ is a convolutional type Calder\'on-Zygmund singular integral of order $m\in\Bbb N$, $m\geq s_0$ where $s_0$ is as in Theorem \ref{Th1.3}, {\rm (i)}, then
\begin{enumerate}
\item[(i)] $T$ is bounded from $H^{p(\cdot),q(\cdot)}({\Bbb R}^n,A)$ into  ${\mathcal L}^{p(\cdot),q(\cdot)}({\Bbb R}^n)$.
\item[(ii)]  $T$ is bounded from $H^{p(\cdot),q(\cdot)}({\Bbb R}^n,A)$ into  itself.
\end{enumerate}
\end{Th}

Our results, as far as we know, are new even in the isotropic case, that is, for the Hardy-Lorentz $H^{p(\cdot),q(\cdot)}({\Bbb R}^n)$ of variable exponents, extending results in \cite{AT1}.

The paper is organized as follows. A proof of Theorem \ref{Th1.1} is presented in Section 2 where we prove the main properties of variable exponent anisotropic Hardy-Lorentz spaces. Next, in Section 3, Calder\'on-Zygmund decompositions in our setting are investigated. The proof of Theorem \ref{Th1.2}, which is presented distinguishing the cases $r=\infty$ and $r<\infty$, is included in Section 4. Finite atomic decompositions are considered in Section 5 where Theorem \ref{Th1.3} is proved. In Section 6 we define the singular integral that we consider and we prove Theorem \ref{Th1.4} after showing some auxiliary results.

Throughout this paper $C$ always denotes a positive constant that can change its value from a line to another one.

\quad \\
\textbf{Acknowledgements}. We would like to offer thanks to the referee for his suggestions which have improved quite a lot this paper.

\section{Maximal characterizations (proof of Theorem \ref{Th1.1})} \label{sec:2}
>From now on, for simplicity, we write $\|\cdot\|_{p(\cdot),q(\cdot)}$ and $\|\cdot\|_{q(\cdot)}$ instead of $\|\cdot\|_{{\mathcal L}^{p(\cdot),q(\cdot)}({\Bbb R}^n)}$ and $\|\cdot\|_{L^{q(\cdot)}(0,\infty)}$, respectively.

First of all we establish very useful boundedness results for the anisotropic maximal function $M_{HL}$ on variable exponent Lorentz spaces.
\begin{Prop}\label{propM_{HL}}
 Assume that $p,q\in{\Bbb P}_1(0,\infty)$. Then, the maximal function $M_{HL}$ is bounded from ${\mathcal L}^{p(\cdot),q(\cdot)}({\Bbb R}^n)$ into itself.
\end{Prop}
\begin{proof}
This property can be proved as \cite[Theorem 3.12]{EKS}. Indeed, it is clear that $\|M_{HL}f\|_{L^{\infty}({\Bbb R}^n)}\leq \|f\|_{L^{\infty}({\Bbb R}^n)}$, $f\in{L^{\infty}({\Bbb R}^n)}$. On the other hand, according to \cite[p. 14]{Bow1} $M_{HL}$ is bounded from ${L^1({\Bbb R}^n)}$ into ${L^{1,\infty}({\Bbb R}^n)}$. Then, by proceeding as in the proof of \cite[Theorem 3.8, p. 122]{BS} we deduce that, for some $C>0$, $(M_{HL}f)^*\leq Cf^{**}$.
Since $p,q\in \mathbb{P}_1(0,\infty )$, by taking $\alpha =1/p-1/q$ and $\nu =0$ in \cite[Theorem 2.2]{EKS} we can write
\begin{align*}
\|M_{HL}(f)\|_{p(\cdot ), q(\cdot )}&\leq C\|t^{1/p(\cdot )-1/q(\cdot )}f^{**}\|_{q(\cdot )}\leq C\|t^{1/p(\cdot )-1/q(\cdot )}f^{*}\|_{q(\cdot )}\\
&=C\|f\|_{p(\cdot ), q(\cdot )}.
\end{align*}
Thus, the proof of this proposition is finished.
\end{proof}

The following vectorial boundedness result for $M_{HL}$ appears as Proposition 1.3 in the introduction.

\begin{Prop}\label{Prop vectorial}
Assume that $p,q\in{\Bbb P}_1$. For every $r\in(1,\infty)$ there exists $C>0$ such that
\begin{equation}{\label{F1}}\left\|\left(\sum_{j\in\Bbb N}(M_{HL}(f_j))^r\right)^{1/r}\right\|_{p(\cdot),q(\cdot)}\leq C\left\|\left(\sum_{j\in\Bbb N}|f_j|^r\right)^{1/r}\right\|_{p(\cdot),q(\cdot)},
\end{equation}
for each sequence $(f_j)_{j\in\Bbb N}$ of functions in $L^1_{loc}({\Bbb R}^n)$.
\end{Prop}
\begin{proof}
According to \cite[Proposition 2.6, (ii)]{BLYZ} the family of anisotropic balls $\{x+B_k\}_{x\in \mathbb{R}^n,\,k\in \mathbb{Z}}$ constitutes a Muckenhoupt basis in $\mathbb{R}^n$. For every $r>0$, we define the $r$-power of the space ${\mathcal L}^{p(\cdot),q(\cdot)}({\Bbb R}^n)$, $({\mathcal L}^{p(\cdot),q(\cdot)}({\Bbb R}^n))^r$, as follows
$$
({\mathcal L}^{p(\cdot),q(\cdot)}({\Bbb R}^n))^r=\{f\,\,\,\mbox{measurable in $\mathbb{R}^n$}\,:|f|^r\in {\mathcal L}^{p(\cdot),q(\cdot)}({\Bbb R}^n)\},
$$
(see \cite[p. 67]{CMP}). By using \cite[Lemma 2.3]{CrW} we deduce that, for every $r>0$, $({\mathcal L}^{p(\cdot),q(\cdot)}({\Bbb R}^n))^r={\mathcal L}^{rp(\cdot),rq(\cdot)}({\Bbb R}^n)$.

We choose $\beta\in (0,1)$ such that $\beta p,\beta q\in \mathbb{P}_1$. According to \cite[Lemma 2.7]{EKS},  $({\mathcal L}^{\beta p(\cdot),\beta q(\cdot)}({\Bbb R}^n))^*=({\mathcal L}^{\beta p(\cdot),\beta q(\cdot)}({\Bbb R}^n))'={\mathcal L}^{(\beta p(\cdot))',(\beta q(\cdot))'}({\Bbb R}^n)$, where the first space represents the associate dual space of ${\mathcal L}^{\beta p(\cdot),\beta q(\cdot)}({\Bbb R}^n)$ in the K\"othe sense (see \cite[p. 25]{LT}). Since $\beta p,\beta q\in \mathbb{P}_1$, Proposition \ref{propM_{HL}} implies that $M_{HL}$ is bounded from ${\mathcal L}^{(\beta p(\cdot))',(\beta q(\cdot))'}({\Bbb R}^n)$ into itself. According to \cite[Corollary 4.8 and Remark 4.9]{CMP} we conclude that (\ref{F1}) holds for every $r\in (1,\infty)$.
\end{proof}

As in \cite[p. 44]{Bow1} we consider the following maximal functions that will be useful in the sequel. If $K\in\Bbb Z$ and $N,L\in\Bbb N$ we define, for every $f\in S'({\Bbb R}^n)$,

$$M_{\varphi}^{0,K,L}(f)(x)=\sup_{k\in\Bbb Z,k\leq K}|(f*{\varphi}_k)(x)|\max(1,\rho(A^{-K}x))^{-L}(1+b^{-k-K})^{-L},\;\;x\in{\Bbb R}^n,$$

$$M_{\varphi}^{K,L}(f)(x)=\sup_{k\in\Bbb Z,k\leq K}\sup_{y\in x+B_k}|(f*{\varphi}_k)(y)|\max(1,\rho(A^{-K}y))^{-L}(1+b^{-k-K})^{-L},\;\;x\in{\Bbb R}^n,$$

$$T_{\varphi}^{N,K,L}(f)(x)=\sup_{k\in\Bbb Z,k\leq K}\sup_{y\in{\Bbb R^n}}{{|(f*{\varphi}_k)(y)|}\over{\max(1,\rho(A^{-k}(x-y)))^N}}{{(1+b^{-k-K})^{-L}}\over{\max(1,\rho(A^{-K}y))^L}},\;\;x\in{\Bbb R}^n,$$

$$M_N^{0,K,L}(f)=\sup_{\varphi\in S_N}M_{\varphi}^{0,K,L}(f),$$
and
$$M_N^{K,L}(f)=\sup_{\varphi\in S_N}M_{\varphi}^{K,L}(f).$$

Now we are going to establish some properties we will need later.

\begin{Lem}\label{lem2.1}
Let $K\in\Bbb Z$, $N,L\in\Bbb N$, $r>0$ and $\varphi\in S({\Bbb R}^n)$. Then, there exists a constant $C>0$ which does not depend neither on $K,L,N,r$ nor $\varphi$ such that, for every $f\in S'({\Bbb R}^n)$
$$\left(T_{\varphi}^{N,K,L}(f)(x)\right)^r\leq C M_{HL}\left(\left(M_{\varphi}^{K,L}(f)\right)^r\right)(x),\;\;\;x\in{\Bbb R}^n.$$
\end{Lem}
\begin{proof}
Our proof is inspired on the ideas presented in \cite[p. 10]{SLur}.

Let $f\in S'(\mathbb{R}^n)$, $k\in\Bbb Z$, $k\leq K$ and $x\in{\Bbb R}^n$. Since
$$\left(|(f*{\varphi}_k)(y)|\max(1,\rho(A^{-K}y))^{-L}(1+b^{-k-K})^{-L}\right)^r\leq\left(M_{\varphi}^{K,L}(f)\right)^r(z),\;\;\;y\in z+B_{k},$$
we can write
$$
\left(|(f*{\varphi}_k)(y)|  \max(1,\rho(A^{-K}y))^{-L}(1+b^{-k-K})^{-L}\right)^r
         \leq {1\over{|y+B_k|}}\int_{y+B_k}\left(M_{\varphi}^{K,L}(f)(z)\right)^rdz,\;\;\;y\in{\Bbb R}^n.
        $$

Suppose that $z\in y+B_k$ and $y\in{\Bbb R}^n$. According to \cite[p. 8]{Bow1} we have that $\rho(z-x)\leq b^{\omega}(\rho(z-y)+\rho(y-x))\leq b^{\omega+k}(1+b^{-k}\rho(y-x))$, where $\omega$ is the smallest integer so that $2B_0\subset B_{\omega}$. We choose $s\in\Bbb Z$ such that $b^{\omega+k}(1+b^{-k}\rho(y-x))\in[b^s,b^{s+1})$. Then, we get

\begin{align*}
       &  \left(|(f*{\varphi}_k)(y)|  \max(1,\rho(A^{-K}y))^{-L}(1+b^{-k-K})^{-L}\right)^r \\
        & \leq b^{\omega}(1+b^{-k}\rho(y-x)) {1\over{b^{\omega+k}(1+b^{-k}\rho(y-x))}}\int_{y+B_k}\left(M_{\varphi}^{K,L}(f)(z)\right)^rdz \\
        & \leq b^{\omega}(1+b^{-k}\rho(y-x)) {1\over{b^s}}\int_{x+B_{s+1}}\left(M_{\varphi}^{K,L}(f)(z)\right)^rdz \\
        & \leq 2 b^{\omega+1}(1+b^{-k}\rho(y-x))^{Nr}M_{HL}\left(\left(M_{\varphi}^{K,L}(f)\right)^r\right)(x),\;\;\;y\in x+B_k.
        \end{align*}
Hence, we obtain
$$\left(T_{\varphi}^{N,K,L}(f)(x)\right)^r\leq C M_{HL}\left(\left(M_{\varphi}^{K,L}(f)\right)^r\right)(x),\;\;\;x\in{\Bbb R}^n.$$
\end{proof}

According to \cite[p. 14]{Bow1}, for every $1<p<\infty$, the Hardy-Littlewood maximal function $M_{HL}$ is bounded from $L^p({\Bbb R}^n)$ into itself. So from Lemma \ref{lem2.1} we deduce that, for every $1<p<\infty$, there exists $C>0$ such that
$$\left\|T_{\varphi}^{N,K,L}(f)\right\|_{L^p({\Bbb R}^n)}\leq C \left\|M_{\varphi}^{K,L}(f)\right\|_{L^p({\Bbb R}^n)},\;\;\;f\in S'({\Bbb R}^n).$$
This property was proved in \cite[Lemma 7.4]{Bow1} by using other procedure.

\begin{Lem}\label{lem2.2}
Let $K\in\Bbb Z$, $N,L\in\Bbb N$ and $\varphi\in S({\Bbb R}^n)$. Assume that $p,q\in{\Bbb P}_0$. Then,
$$\|T_{\varphi}^{N,K,L}(f)\|_{p(\cdot),q(\cdot)}\leq C\|M_{\varphi}^{K,L}(f)\|_{p(\cdot),q(\cdot)},\;\;\;f\in S'({\Bbb R}^n),$$
where $C>0$ does not depend on $(N,K,L,\varphi)$.
\end{Lem}
\begin{proof}
We choose $r>0$ such that $rp,rq\in{\Bbb P}_1$. Let $f\in S'({\Bbb R}^n)$. According to \cite[Lemma 2.3]{CrW} and a well known property of the nondecreasing equimeasurable rearrangement we get
\begin{align*}
         \|T_{\varphi}^{N,K,L}(f)\|_{p(\cdot),q(\cdot)}& =\|t^{{1\over{rp(t)}}-{1\over{rq(t)}}}\left([T_{\varphi}^{N,K,L}(f)]^*(t)\right)^{1/r}\|_{rq(\cdot)}^r \\
        & = \|t^{{1\over{rp(t)}}-{1\over{rq(t)}}}\left[\left(T_{\varphi}^{N,K,L}(f)\right)^{1/r}\right]^*(t)\|_{rq(\cdot)}^r  \\
        & =\|\left(T_{\varphi}^{N,K,L}(f)\right)^{1/r}\|_{rp(\cdot),rq(\cdot)}^r.
        \end{align*}
>From Lemma \ref{lem2.1} and Proposition \ref{propM_{HL}} it follows that
$$\|T_{\varphi}^{N,K,L}(f)\|_{p(\cdot),q(\cdot)}\leq C\|\left(M_{\varphi}^{K,L}(f)\right)^{1/r}\|_{rp(\cdot),rq(\cdot)}^r=C\|M_{\varphi}^{K,L}(f)\|_{p(\cdot),q(\cdot)}$$
\end{proof}

Next two results were established in \cite[p. 45-47]{Bow1} as Lemmas 7.5 and 7.6, respectively.
\begin{Lem}\label{lem2.4}
For every $N,L\in \mathbb{N}$, there exists $M_0\in \mathbb{N}$ satisfying the following property: if $\varphi\in S({\Bbb R}^n)$ is such that $\int\varphi(x)dx\neq 0$, then there exists $C>0$ such that, for every $f\in S'({\Bbb R}^n)$ and $K\in\Bbb N$,
$$M_{M_0}^{0,K,L}(f)(x)\leq CT_{\varphi}^{N,K,L}(f)(x),\;\;\;x\in{\Bbb R}^n.$$
\end{Lem}

\begin{Lem}\label{lem2.5}
Let $\varphi\in S({\Bbb R}^n)$. Then, for every $M,K\in\Bbb N$ and  $f\in S'({\Bbb R}^n)$ there exist $L\in\Bbb N$ and $C>0$ such that
$$M_{\varphi}^{K,L}(f)(x)\leq C\max(1,\rho_A(x))^{-M},\;\;\;x\in{\Bbb R}^n.$$
Actually, $L$ does not depend on $K\in \mathbb{N}$.
\end{Lem}

\begin{Lem}\label{lem2.3}
Let $p,q\in{\Bbb P}_0$. There exists $\alpha_0>0$ such that the function $g_{\alpha}$ defined by
$$g_{\alpha}(x)=\left(\max(1,\rho_A(x))\right)^{-\alpha},\;\;\;x\in{\Bbb R}^n,$$
is in ${\mathcal L}^{p(\cdot),q(\cdot)}({\Bbb R}^n)$, for every $\alpha \geq\alpha_0$.
\end{Lem}
\begin{proof}
Let $\alpha>0$. According to \cite[Lemma 3.2]{Bow1} we have that
$$g_{\alpha}(x)\le h_\alpha(x)=C\left\{\begin{array}{ll}
1,\;\;\;|x|\leq 1, \\
 \\
 |x|^{-\alpha \ln\;b/\ln\;{\lambda}_+},\;\;\;|x|>1,
 \end{array}\right.$$
 for certain $C>0$. Here ${\lambda}_+$ is greater than $\max\{|\lambda|:\;\lambda\;\mbox{is an eingenvalue of}\;A\}$ (for instance we can take $\lambda_+=2\max\{|\lambda|:\;\lambda\;\mbox{is an eingenvalue of}\;A\}$). Note that $g_\alpha^*\le h_\alpha^*$.

To simplify we denote $v_n=|B(0,1)|$. We have that
$$
\mu_{h_\alpha}(s)=\left\{\begin{array}{ll}
                        0,\,\,\, s\ge C,\\
                        \\
                        v_n(C/s)^{n\ln(\lambda_+)/(\alpha \ln(b))}, \,\,\,s\in (0,C).
                        \end{array}
                        \right.
                        $$
                        Then,
$$
h_\alpha^*(t)=C\left\{\begin{array}{ll}
                        1,\,\,\,t\in (0,v_n),\\
                        \\
                        (v_n/t)^{\alpha \ln(b)/(n\ln(\lambda_+)},\,\,\,t\ge v_n.
                        \end{array}
                        \right.
                        $$

Since $q(0)>0$ and $p(0)>0$ we have that $\displaystyle\int_0^{v_n}t^{q(t)/p(t)-1}|g_{\alpha}^*(t)|^{q(t)}dt<\infty$. Also, there exists $\alpha_0>0$ such that $\displaystyle\int_{v_n}^{\infty}t^{q(t)/p(t)-1}|g_{\alpha_0}^*(t)|^{q(t)}dt<\infty$ because $p,q\in{\Bbb P}_0$.

Hence, $g_{\alpha}\in {\mathcal L}^{p(\cdot),q(\cdot)}({\Bbb R}^n)$, for every $\alpha \geq\alpha_0$.

\end{proof}

\begin{Lem}\label{medida}
Let $p,q\in \mathfrak{P}_0$ and let $D$ be a subset of $\mathbb{R}^n$. Then, $\chi_D\in {\mathcal L}^{p(\cdot),q(\cdot)}({\Bbb R}^n)$ if, and only if, $|D|<\infty$.
\end{Lem}

\begin{proof} We have that $(\chi_D)^*=\chi_{[0,|D|)}$. Since $p,q\in \mathfrak{P}_0$, for every $\lambda>0$
$$
\int_0^\infty \left(\frac{(\chi_D)^*(t)t^{\frac{1}{p(t)}-\frac{1}{q(t)}}}{\lambda}\right)^{q(t)}dt=\int_0^{|D|}\frac{t^{-1+q(t)/p(t)}}{\lambda^{q(t)}}dt<\infty,
$$
if, and only if, $|D|<\infty$.
\end{proof}

\vspace{3mm}

{\bf  Proof of Theorem \ref{Th1.1}. }

We recall that we are taking $f\in S'(\mathbb{R}^n)$ and $\varphi\in S(\mathbb{R}^n)$ such that $\int\varphi(x)dx\neq 0$. It is clear that, for every $N\in \mathbb{N}$,
$$\|M_{\varphi}^0(f)\|_{p(\cdot),q(\cdot)}\leq\|M_{\varphi}(f)\|_{p(\cdot),q(\cdot)}\leq C\|f\|_{H_N^{p(\cdot),q(\cdot)}({\Bbb R}^n,A)}.$$
Hence, $(i)\Rightarrow (ii)\Rightarrow (iii)$.

Now, we are going to complete the proof. Let $M_0$ be the value in Lemma \ref{lem2.4} for $N=L=0$. Then, for a certain $C>0$,
\begin{equation}\label{eq2.1}
\|M_M^0(g)\|_{p(\cdot),q(\cdot)}\leq C\|M_{\varphi}(g)\|_{p(\cdot),q(\cdot)},\; g\in S'(\mathbb{R}^n)\,\,\,{\rm and}\,\,\,M\geq M_0.
\end{equation}
Indeed, by Lemma \ref{lem2.4}, there exists $C>0$ such that
$$M_M^{0,K,0}(g)(x)\leq CT_{\varphi}^{0,K,0}(g)(x),\;\;\;x\in{\Bbb R}^n,\,\,\,g\in S'(\mathbb{R}^n),\,K\in \mathbb{N},\,\,{\rm and}\,\,M\geq M_0.$$
Then, Lemma \ref{lem2.2} leads to
$$\|M_M^{0, K,0}(g)\|_{p(\cdot),q(\cdot)}\leq C\|M_{\varphi}^{K,0}(g)\|_{p(\cdot),q(\cdot)},\;\;\;g\in S'(\mathbb{R}^n),\;K\in\Bbb N\;\mbox{and}\;M\geq M_0.$$
By using monotone convergence theorem in ${\mathcal L}^{p(\cdot),q(\cdot)}({\Bbb R}^n)$ (see \cite[Definition 2.5 v)]{EKS}) jointly with \cite[Lemma 2.3]{CrW} and by letting $K\rightarrow\infty$ we conclude that (\ref{eq2.1}) holds.

\vspace{3mm}

Our next objective is  to see that, for a certain $C>0$,
\begin{equation}\label{F01}
\|M_{\varphi}(f)\|_{p(\cdot),q(\cdot)}\leq C\|M_{\varphi}^0(f)\|_{p(\cdot),q(\cdot)}.
\end{equation}
Note that by combining (\ref{eq2.1}), (\ref{F01}) and \cite[Proposition 3.10]{Bow1} we conclude that $(iii)\Rightarrow (ii)\Rightarrow (i)$.

In order to show (\ref{F01}) we firstly note that there exists $L_0\in\Bbb N$ such that $M_{\varphi}^{K,L_0}(f)\in \mathcal{L}^{p(\cdot),q(\cdot)}({\Bbb R}^n)$, for every $K\in\Bbb N$. Indeed, we denote by $\alpha_0$ the constant appearing in Lemma \ref{lem2.3}. According to Lemma \ref{lem2.5} we can find $L_0\in\Bbb N$ such that, for every $K\in\Bbb N$, there exists $C>0$ for which

\begin{equation}\label{eq2.2}
M_{\varphi}^{K,L_0}(f)(x)\leq C\max(1,\rho(x))^{-\alpha_0},\;\;\;x\in{\Bbb R}^n.
\end{equation}
Then, Lemma \ref{lem2.3} leads to $M_{\varphi}^{K,L_0}(f)\in {\mathcal L}^{p(\cdot),q(\cdot)}({\Bbb R}^n)$, for each $K\in\Bbb N$.

>From Lemmas \ref{lem2.2} and \ref{lem2.4} we infer that there exist $M_0\in\Bbb N$ and $C_0>0$ such that
\begin{equation}\label{eq2.3}
\|M_{M_0}^{0,K,L_0}(f)\|_{p(\cdot),q(\cdot)}\leq C_0\|M_{\varphi}^{K,L_0}(f)\|_{p(\cdot),q(\cdot)},
\end{equation}
for every $K\in\Bbb N$.

Fix $K_0\in\Bbb N$. We define the set $\Omega_0$ by
$$\Omega_0=\left\{x\in{\Bbb R}^n:\;M_{M_0}^{0,K_0,L_0}(f)(x)\leq C_2M_{\varphi}^{K_0,L_0}(f)(x)\right\},$$
where $C_2>0$ will be specified later.

By using (\ref{eq2.3}), \cite[Lemma 2.3]{CrW} and \cite[Theorem 2.4]{EKS} and choosing $r> 1$ such that $rp,rq\in{\Bbb P}_1$, we get

\begin{align*}
& \|M_{\varphi}^{K_0,L_0}(f)\|_{p(\cdot),q(\cdot)}=\|t^{{1\over{p(t)}}-{1\over{q(t)}}}\left(M_{\varphi}^{K_0,L_0}(f)\right)^*(t)\|_{q(\cdot)} \\
& =\|t^{{1\over{rp(t)}}-{1\over{rq(t)}}}\left([M_{\varphi}^{K_0,L_0}(f)]^*(t)\right)^{1/r}\|_{rq(\cdot)}^r\\
& =\|t^{{1\over{rp(t)}}-{1\over{rq(t)}}}\left([M_{\varphi}^{K_0,L_0}(f)]^{1/r}\right)^*(t)\|_{rq(\cdot)}^r=\|[M_{\varphi}^{K_0,L_0}(f)]^{1/r}\|_{rp(\cdot),rq(\cdot)}^r \\
& \leq \left(\|[M_{\varphi}^{K_0,L_0}(f)]^{1/r}\|_{rp(\cdot),rq(\cdot)}^{(1)}\right)^r  \\
& \leq A_1\left\{\left(\|[M_{\varphi}^{K_0,L_0}(f)\chi_{\Omega_0}]^{1/r}\|_{rp(\cdot),rq(\cdot)}^{(1)}\right)^r+\left(\|[M_{\varphi}^{K_0,L_0}(f)\chi_{\Omega_0^c}]^{1/r}\|_{rp(\cdot),rq(\cdot)}^{(1)}\right)^r \right\} \\
& \leq A_1\left\{\left(\|[M_{\varphi}^{K_0,L_0}(f)\chi_{\Omega_0}]^{1/r}\|_{rp(\cdot),rq(\cdot)}^{(1)}\right)^r+{1\over{C_2}}\left(\|[M_{M_0}^{0,K_0,L_0}(f)]^{1/r}\|_{rp(\cdot),rq(\cdot)}^{(1)}\right)^r \right\} \\
& \leq A_2\left\{\left(\|[M_{\varphi}^{K_0,L_0}(f)\chi_{\Omega_0}]^{1/r}\|_{rp(\cdot),rq(\cdot)}\right)^r+{1\over{C_2}}\left(\|[M_{M_0}^{0,K_0,L_0}(f)]^{1/r}\|_{rp(\cdot),rq(\cdot)}\right)^r \right\} \\
& \leq A_2\left(\|M_{\varphi}^{K_0,L_0}(f)\chi_{\Omega_0}\|_{p(\cdot),q(\cdot)}+{1\over{C_2}}\|M_{M_0}^{0,K_0,L_0}(f)\|_{p(\cdot),q(\cdot)}\right) \\
 & \leq A_2\left(\|M_{\varphi}^{K_0,L_0}(f)\chi_{\Omega_0}\|_{p(\cdot),q(\cdot)}+{{C_0}\over{C_2}}\|M_{\varphi}^{K_0,L_0}(f)\|_{p(\cdot),q(\cdot)}\right),
 \end{align*}
where $A_1,A_2>0$ depend only on $p,q$ and $r$. Hence, by taking $C_2\geq 2C_0A_2$ we obtain
$$\|M_{\varphi}^{K_0,L_0}(f)\|_{p(\cdot),q(\cdot)}\leq 2A_2\|M_{\varphi}^{K_0,L_0}(f)\chi_{\Omega_0}\|_{p(\cdot),q(\cdot)},$$
because $M_{\varphi}^{K_0,L_0}(f)\in {\mathcal L}^{p(\cdot),q(\cdot)}({\Bbb R}^n)$.

According to \cite[(7.16)]{Bow1} we have that
\begin{equation}\label{eq2.4}
M_{\varphi}^{K_0,L_0}(f)(x)\leq C\left[M_{HL}\left(M_{\varphi}^{0,K_0,L_0}(f)^{1/r}\right)(x)\right]^{r},\;\;\;x\in\Omega_0.
\end{equation}
The constant C>0 does not depend on $K_0$ but it depends on $L_0$.

>From (\ref{eq2.4}), Proposition \ref{propM_{HL}} and \cite[Lemma 2.3]{CrW} we obtain

\begin{align*}
         \|M_{\varphi}^{K_0,L_0}(f)\chi_{\Omega_0}\|_{p(\cdot),q(\cdot)}& \leq C \left\|\left(M_{HL}(M_{\varphi}^{0,K_0,L_0}(f)^{1/r})\right)^{r}\right\|_{p(\cdot),q(\cdot)} \\
        & = C \|M_{HL}(M_{\varphi}^{0,K_0,L_0}(f)^{1/r})\|_{rp(\cdot),rq(\cdot)}^{r}\\
        &\leq C \|M_{\varphi}^{0,K_0,L_0}(f)^{1/r}\|_{rp(\cdot),rq(\cdot)}^{r} \\
         & = C \|M_{\varphi}^{0,K_0,L_0}(f)\|_{p(\cdot),q(\cdot)}.
        \end{align*}
We conclude that
$$\|M_{\varphi}^{K_0,L_0}(f)\|_{p(\cdot),q(\cdot)}\leq C \|M_{\varphi}^{0,K_0,L_0}(f)\|_{p(\cdot),q(\cdot)}.$$
Note that again this constant $C>0$ does not depend on $K_0$ and it depends on $L_0$.

We have that $M_{\varphi}^{K,L_0}(f)(x)\;\uparrow\; M_{\varphi}(f)(x)$, as $K\rightarrow\infty$, for every $x\in{\Bbb R}^n$, and $M_{\varphi}^{0,K,L_0}(f)(x)\;\uparrow\; M_{\varphi}^0(f)(x)$, as $K\rightarrow\infty$, for every $x\in{\Bbb R}^n$. Hence, monotone convergence theorem in ${\mathcal L}^{p(\cdot),q(\cdot)}({\Bbb R}^n)$-setting (\cite[Theorem 2.8 and Definition 2.5, v)]{EKS}, jointly with \cite[Lemma 2.3]{CrW}), leads to
$$\|M_{\varphi}(f)\|_{p(\cdot),q(\cdot)}\leq C\|M_{\varphi}^0(f)\|_{p(\cdot),q(\cdot)}.$$
 Observe that the last inequality says that $M_{\varphi}(f)\in {\mathcal L}^{p(\cdot),q(\cdot)}({\Bbb R}^n)$, but the constant $C>0$ depends on $f$ because $L_0$ depends also on $f$.

On the other hand, since $M_{\varphi}(f)\in {\mathcal L}^{p(\cdot),q(\cdot)}({\Bbb R}^n)$, $M_{\varphi}^{K,0}(f)\in {\mathcal L}^{p(\cdot),q(\cdot)}({\Bbb R}^n)$, for every $K\in\Bbb N$. Hence, we can take $L_0=0$ at the beginning of the proof of this part. By proceeding as above we concluded that
$$\|M_{\varphi}(f)\|_{p(\cdot),q(\cdot)}\leq C\|M_{\varphi}^0(f)\|_{p(\cdot),q(\cdot)},$$
where $C>0$ does not depend on $f$.

Thus the proof of the theorem is finished.
$\Box$

The last part of this section is dedicated to establish some properties of the space $H^{p(\cdot),q(\cdot)}({\Bbb R}^n,A)$.

\begin{Prop}\label{prop2.2}
Let $p,q\in{\Bbb P}_0$. Then, $H^{p(\cdot),q(\cdot)}({\Bbb R}^n,A)$ is continuously contained in $S'({\Bbb R}^n)$.
\end{Prop}
\begin{proof}
Let $f\in S'({\Bbb R}^n)$ and $\varphi\in S({\Bbb R}^n)$. We define $\lambda_0=|<f,\varphi >|$. We can write
$$\lambda_0=|(f*\varphi)(0)|\leq \sup_{z\in x+B_0}|(f*\varphi)(z)|\leq M_{\varphi}(f)(x),\;\;\;x\in B_0.$$
Then,
$$|\{x\in{\Bbb R}^n:\;M_{\varphi}(f)(x)>\lambda_0/2\}|\geq 1,$$
and
$$(M_{\varphi}(f))^*(t)\geq\lambda_0/2,\;\;\;t\in (0,1).$$
Hence, we get
$$\|M_{\varphi}(f)\|_{p(\cdot),q(\cdot)}\geq \|t^{{1\over{p(\cdot)}}-{1\over{q(\cdot)}}}(M_{\varphi}(f))^*(t)\chi_{(1/2,1)}(t)\|_{q(\cdot)}\geq {{\lambda_0}\over 2}\|t^{{1\over{p(\cdot)}}-{1\over{q(\cdot)}}}\chi_{(1/2,1)}(t)\|_{q(\cdot)}.$$
Since $\|t^{{1\over{p(\cdot)}}-{1\over{q(\cdot)}}}\chi_{(1/2,1)}(t)\|_{q(\cdot)}>0$ we conclude the desired result.
\end{proof}

\begin{Prop}\label{prop2.3}
Let $p,q\in{\Bbb P}_0$. If $f\in H^{p(\cdot),q(\cdot)}({\Bbb R}^n,A)$, then $f$ is a bounded distribution in $S'({\Bbb R}^n)$.
\end{Prop}
\begin{proof}
Let $f\in H^{p(\cdot),q(\cdot)}({\Bbb R}^n,A)$ and $\varphi\in S({\Bbb R}^n)$. For every $x\in {\Bbb R}^n$, we have that
$$|(f*\varphi)(x)|\leq \sup_{z\in y+B_0}|(f*\varphi)(z)|\leq M_{\varphi}(f)(y),\;\;\;y\in x+B_0.$$
By proceeding as in the proof of Proposition \ref{prop2.2}, we deduce that, for a certain $C>0$,
$$|(f*\varphi)(x)|\leq C \|f\|_{H^{p(\cdot),q(\cdot)}({\Bbb R}^n,A)},\;\;\;x\in{\Bbb R}^n.$$
Thus, we prove that $f$ is a bounded distribution in $S'({\Bbb R}^n)$.
\end{proof}

\begin{Prop}\label{prop2.4}
Assume that $p,q\in{\Bbb P}_0$. Then, $H^{p(\cdot),q(\cdot)}({\Bbb R}^n,A)$ is complete.
\end{Prop}
\begin{proof}
We choose $r\in (0,1]$ such that ${{p(\cdot)}\over r},{{q(\cdot)}\over r}\in {\Bbb P}_1$. In order to see that $ H^{p(\cdot),q(\cdot)}({\Bbb R}^n,A)$ is complete it is sufficient to prove that if $(f_k)_{k\in\Bbb N}$ is a sequence in  $H^{p(\cdot),q(\cdot)}({\Bbb R}^n,A)$ such that $\displaystyle\sum_{k\in\Bbb N}\|f_k\|_{H^{p(\cdot),q(\cdot)}({\Bbb R}^n,A)}^r<\infty$, then the series $\displaystyle\sum_{k\in\Bbb N}f_k$ converges in $H^{p(\cdot),q(\cdot)}({\Bbb R}^n,A)$ (see, for instance \cite[Theorem 1.6, p. 5]{BS}). Assume that $(f_k)_{k\in\Bbb N}$ is a sequence in  $H^{p(\cdot),q(\cdot)}({\Bbb R}^n,A)$ such that $\displaystyle\sum_{k\in\Bbb N}\|f_k\|_{H^{p(\cdot),q(\cdot)}({\Bbb R}^n,A)}^r<\infty$. We define, for every $j\in\Bbb N$, $\displaystyle F_j=\sum_{k=0}^{j}f_k$. According to
\cite[Lemma 2.3]{CrW} and \cite[Theorem 2.4]{EKS}, if $j,\ell\in\Bbb N$, $j<\ell$, we get
\begin{align*}
 \|F_{\ell}-F_j\|_{H^{p(\cdot),q(\cdot)}({\Bbb R}^n,A)}^r&=\Big\|\sum_{k=j +1}^{\ell}f_k\Big\|_{H^{p(\cdot),q(\cdot)}({\Bbb R}^n,A)}^r \leq \Big \|\sum_{k=j +1}^{\ell}M_N(f_k)\Big\|_{p(\cdot),q(\cdot)}^r \\
& \hspace{-3cm}=\Big\|\Big(\sum_{k=j +1}^{\ell}M_N(f_k)\Big)^r\Big\|_{p(\cdot)/r,q(\cdot)/r} \leq \Big\|\sum_{k=j +1}^{\ell}(M_N(f_k))^r\Big\|_{p(\cdot)/r,q(\cdot)/r}  \leq \Big\|\sum_{k=j +1}^{\ell}(M_N(f_k))^r\Big\|_{p(\cdot)/r,q(\cdot)/r}^{(1)}\\
& \hspace{-3cm}\leq\sum_{k=j +1}^{\ell} \|(M_N(f_k))^r\|_{p(\cdot)/r,q(\cdot)/r}^{(1)}  \leq C \sum_{k=j +1}^{\ell} \|(M_N(f_k))^r\|_{p(\cdot)/r,q(\cdot)/r} =C\sum_{k=j +1}^{\ell} \|f_k\|_{H^{p(\cdot),q(\cdot)}({\Bbb R}^n,A)}^r.
\end{align*}
Hence, $(F_j)_{j\in\Bbb N}$ is a Cauchy sequence in $H^{p(\cdot),q(\cdot)}({\Bbb R}^n,A)$. By Proposition \ref{prop2.2} $(F_j)_{j\in\Bbb N}$ is a Cauchy sequence in $S'({\Bbb R}^n)$. Then, there exists $F\in S'({\Bbb R}^n)$ such that $F_j\rightarrow F$, as $j\rightarrow\infty$, in $S'({\Bbb R}^n)$. We have that
$$M_N(F)\leq\lim_{j\rightarrow\infty}\sum_{k=0}^jM_N(f_k).$$
According to \cite[Theorem 2.8 and Definition 2.5 v)]{EKS} by proceeding as above we obtain

\begin{align*}
\|M_N(F)\|_{p(\cdot),q(\cdot)}^r& \leq\left\|\lim_{j\rightarrow\infty}\sum_{k=0}^{j}M_N(f_k)\right\|_{p(\cdot),q(\cdot)}^r\\
& = \lim_{j\rightarrow\infty} \left\|\sum_{k=0}^{j}M_N(f_k)\right\|_{p(\cdot),q(\cdot)}^r \\
& \leq  \sum_{k\in\Bbb N} \|(M_N(f_k))^r\|_{p(\cdot)/r,q(\cdot)/r}\\
& = C\sum_{k\in\Bbb N} \|f_k\|_{H^{p(\cdot),q(\cdot)}({\Bbb R}^n,A)}^r.
\end{align*}
Then, $F\in H^{p(\cdot),q(\cdot)}({\Bbb R}^n,A)$. Also, we have that
$$\|F-\sum_{k=0}^{j}f_k\|_{H^{p(\cdot),q(\cdot)}({\Bbb R}^n,A)}^r\leq C\sum_{k=j+1}^{\infty} \|f_k\|_{H^{p(\cdot),q(\cdot)}({\Bbb R}^n,A)}^r,\;\;\;j\in\Bbb N.$$
Hence, $\displaystyle F=\sum_{k\in\Bbb N}f_k$ in the sense of convergence in $H^{p(\cdot),q(\cdot)}({\Bbb R}^n,A)$.
\end{proof}

\section{A Calder\'on-Zygmund decomposition} \label{sec:3}

In this section we study a Calder\'on-Zygmund decomposition for our anisotropic setting (associated with the matrix dilation $A$) for a distribution $f\in S'({\Bbb R}^n)$ satisfying that $|\{x\in{\Bbb R}^n:\;M_Nf(x)>\lambda\}|<\infty$, where $N\in\Bbb N$, $N\geq 2$ and $\lambda>0$. We will use the ideas and results established in \cite[Section 5, Chapter I]{Bow1}. Also we prove new properties involving variable exponent Hardy-Lorentz norms that will be useful in the sequel.

Let $\lambda>0$, $N\in\Bbb N$, $N\geq 2$ and $f\in S'({\Bbb R}^n)$ such that $|\Omega_\lambda|<\infty$ where
$$\Omega_\lambda=\{x\in{\Bbb R}^n:\;M_N(f)(x)>\lambda\}.$$
According to the Whitney Lemma (\cite[Lemma 2.7]{Bow1}) there exist sequences $(x_j)_{j\in\Bbb N}\subset\Omega_\lambda$ and $(\ell_j)_{j\in\Bbb N}\subset\Bbb Z$ satisfying that
\begin{equation}\label{eq3.1}
\Omega_\lambda=\bigcup_{j\in\Bbb N}(x_j+B_{\ell_j});
\end{equation}

\begin{equation}\label{eq3.2}
(x_i+B_{\ell_i-\omega})\bigcap(x_j+B_{\ell_j-\omega})=\emptyset,\;\;i,j\in\Bbb N,\;i\neq j;
\end{equation}

\begin{equation}\label{eq3.3}
(x_j+B_{\ell_j+4\omega})\bigcap{\Omega_\lambda}^c=\emptyset,\;\;(x_j+B_{\ell_j+4\omega+1})\bigcap{\Omega_\lambda}^c\neq\emptyset,\;\;j\in\Bbb N;
\end{equation}

\begin{equation}\label{eq3.4}
\mbox{if}\; i,j\in\Bbb N\;\mbox{and}\; (x_i+B_{\ell_i+2\omega})\bigcap(x_j+B_{\ell_j+2\omega})\neq\emptyset,\;\;\mbox{then}\;\;|\ell_i-\ell_j|\leq\omega;
\end{equation}

\begin{equation}\label{eq3.5}
\sharp\{j\in\Bbb N:\; (x_i+B_{\ell_i+2\omega})\bigcap(x_j+B_{\ell_j+2\omega})\neq\emptyset\}\leq L,\;\;i\in\Bbb N.
\end{equation}
Here $L$ denotes a nonnegative integer that does not depend on $\Omega_\lambda$. If $E\subset{\Bbb R}^n$ by $\sharp E$ we represent the cardinality of $E$.

Assume now that $\theta\in C^{\infty}({\Bbb R}^n)$ satisfies that $\mbox{supp}\;\theta\subset B_{\omega}$, $0\leq\theta\leq 1$, and $\theta=1$ on $B_0$. We define, for every $j\in\Bbb N$,
$$\theta_j(x)=\theta(A^{-\ell_j}(x-x_j)),\;\;\;x\in{\Bbb R}^n,$$
and, for every $i\in\Bbb N$,
$$\zeta_i(x)=\left\{\begin{array}{ll}
\theta_i(x)/(\sum_{j\in\Bbb N}\theta_j(x)),& x\in\Omega_\lambda, \\
 & \\
 0, & x\in\Omega_\lambda^c.
 \end{array}\right.$$
The sequence $\{\zeta_i\}_{i\in\Bbb N}$ is a smooth partition of unity associated with the covering $\{x_i+B_{\ell_i+\omega}\}_{i\in\Bbb N}$ of $\Omega$.

Let $i,s\in\Bbb N$. By ${\mathcal P}_s$ we denote the linear space of polynomials in ${\Bbb R}^n$ with degree at most $s$. ${\mathcal P}_s$ is endowed with the norm $\|\cdot\|_{i,s}$ defined by
$$\|P\|_{i,s}=\left({1\over{\int\zeta_i}}\int_{{\Bbb R}^n}|P(x)|^2\zeta_i(x)dx\right)^{1/2},\;\;\;P\in{\mathcal P}_s.$$
Thus $({\mathcal P}_s,\|\cdot\|_{i,s})$ is a Hilbert space. We consider on ${\mathcal P}_s$ the functional $T_{f,i,s}$ given by
$$T_{f,i,s}(Q)={1\over{\int\zeta_i}}<f,Q\zeta_i>,\;\;\;Q\in{\mathcal P}_s.$$
$T_{f,i,s}$ is continuous in $({\mathcal P}_s,\|\cdot\|_{i,s})$ and there exists $P_{f,i,s}\in{\mathcal P}_s$ such that
$$T_{f,i,s}(Q)={1\over{\int\zeta_i}}\int_{{\Bbb R}^n}P_{f,i,s}(x)Q(x)\zeta_i(x)dx,\;\;\;Q\in{\mathcal P}_s.$$
To simplify we write $P_i$ to refer to $P_{f,i,s}$. We define $b_i=(f-P_i)\zeta_i$.

We will find values of $s$ and $N$ for which the series $\displaystyle\sum_{i\in\Bbb N} b_i$ converges in $S'({\Bbb R}^n)$ provided that $f\in H^{p(\cdot),q(\cdot)}({\Bbb R}^n,A)$. Then, we define $\displaystyle g=f-\sum_{i\in\Bbb N} b_i$.

The representation $\displaystyle f=g+\sum_{i\in\Bbb N} b_i$ is known as the Calder\'on-Zygmund decomposition of $f$ of degree $s$ and height $\lambda$ associated with $M_N(f)$.

Note firstly that if $f\in H^{p(\cdot),q(\cdot)}({\Bbb R}^n,A)$ and $N\in\Bbb N$, $N\geq N_0$, then $\|\chi_{\{x\in{\Bbb R}^n:\;M_N(f)(x)>\mu\}}\|_{p(\cdot),q(\cdot)}<\infty$ for every $\mu>0$, and by Lemma \ref{medida}, $|\{x\in{\Bbb R}^n:\;M_N(f)(x)>\mu\}|<\infty$, for every $\mu>0$. Here $N_0$ is the one defined in Theorem \ref{Th1.1}.

Our next objective is to prove that $L^1_{loc}({\Bbb R}^n)\cap H^{p(\cdot),q(\cdot)}({\Bbb R}^n,A)$ is a dense subspace of $H^{p(\cdot),q(\cdot)}({\Bbb R}^n,A)$. This property will be useful to deal with the proof that every element of $H^{p(\cdot),q(\cdot)}({\Bbb R}^n,A)$ can be represented as a sum of a special kind of distributions, so called atoms, which will be developed in the next section.

We need to establish some auxiliary results. Firstly we prove the absolute continuity of the norm $\|\cdot\|_{p(.),q(.)}$.

\begin{Prop}\label{prop2.1}
Let $(E_k)_{k\in\Bbb N}$ be a sequence of measurable sets satisfying that $E_k\supset E_{k+1}$, $k\in\Bbb N$, $|E_1|<\infty$, and $|{\cap}_{k\in\Bbb N}E_k|=0$. Assume that $p,q\in{\Bbb P}_0$. If $f\in {\mathcal L}^{p(\cdot),q(\cdot)}({\Bbb R}^n)$, then
$$\|f\chi_{E_k}\|_{p(\cdot),q(\cdot)}\rightarrow\;0,\;\;\;\mbox{as}\;k\rightarrow\infty.$$
\end{Prop}
\begin{proof}
Let $f\in {\mathcal L}^{p(\cdot),q(\cdot)}({\Bbb R}^n)$ and $k\in\Bbb N$. We have that $(f\chi_{E_k})^*\leq f^*$. Then, $f\chi_{E_k}\in{\mathcal L}^{p(\cdot),q(\cdot)}({\Bbb R}^n)$. Moreover, since $|\cap_{k\in\Bbb N} E_k|=\lim_{k\to\infty}|E_k|=0$, for every $t>0$ there exists $k_0\in \mathbb{N}$ such that $(f\chi_{E_k})^*(t)=0$, $k\in \mathbb{N}$, $k\ge k_0$. Hence, for every $t>0$,
$$t^{{1\over{p(\cdot)}}-{1\over{q(\cdot)}}}(f\chi_{E_k})^*(t)\rightarrow 0,\;\;\;\mbox{as}\;k\rightarrow\infty.$$
By using dominated convergence theorem (\cite[Lemma 3.2.8]{DHHR}) jointly with \cite[Lemma 2.3]{CrW} and by taking into account that $q\in \mathbb{P}_0$ and that  $f\in {\mathcal L}^{p(\cdot),q(\cdot)}({\Bbb R}^n)$, we obtain
$$\|f\chi_{E_k}\|_{p(\cdot),q(\cdot)}\rightarrow 0,\;\;\;\mbox{as}\;k\rightarrow\infty.$$
\end{proof}

Note that the last property also holds by more general exponent functions $p$ and $q$.

\begin{Prop}\label{prop3.2}
Assume that $p,q\in{\Bbb P}_0$. There exists $s_0\in \mathbb{N}$, such that, for every $s\in \mathbb{N}$, $s\ge s_0$, and each $N\in \mathbb{N}$, $N>max\{N_0,s\}$, where $N_0$ is defined in Theorem \ref{Th1.1}, the following two properties holds.

(i) Let $f\in H^{p(\cdot),q(\cdot)}({\Bbb R}^n,A)$ and $\lambda>0$. If $f=g+\sum_{i\in\Bbb N} b_{i}$ is the anisotropic Calder\'on-Zygmund decomposition of $f$ associated to $M_Nf$ of height $\lambda$ and degree $s$, then the series $\sum_{i\in\Bbb N} b_i$ converges in $H^{p(\cdot),q(\cdot)}({\Bbb R}^n,A)$.

(ii) Suppose that $f\in H^{p(\cdot),q(\cdot)}({\Bbb R}^n,A)$ and that, for every $j\in\Bbb Z$, $f=g_j+\sum_{i\in\Bbb N} b_{i,j}$ is the anisotropic Calder\'on-Zygmund decomposition of $f$ associated to $M_Nf$ of height $2^j$ and degree $s$. Then, $(g_j)_{j\in\Bbb Z}\subset H^{p(\cdot),q(\cdot)}({\Bbb R}^n,A)$ and $(g_j)_{j\in\Bbb Z}$ converges to $f$, as $j\to +\infty$, in $ H^{p(\cdot),q(\cdot)}({\Bbb R}^n,A)$.
\end{Prop}
\begin{proof}

(i) Let $s,N\in \mathbb{N}$, $N>\max\{N_0,s\}$. The Calder\'on-Zygmund decomposition of $f$ associated to $M_Nf$ of height $\lambda>0$ and degree $s$ is
$f=g+\sum_{i\in\Bbb N} b_i$. We are going to specify $s$ and $N$ in order that the series $\sum_{i\in\Bbb N} b_i$ converges in $H^{p(\cdot),q(\cdot)}({\Bbb R}^n,A)$.

 By using \cite[Lemmas 5.4 and 5.6]{Bow1} we get that there exists $C>0$ so that, for every $i\in \mathbb{N}$,
\begin{align*}
         M_N(b_{i})(x)\leq C\left(M_Nf(x)\chi_{ x_i+B_{\ell_i+2\omega}}(x)+\lambda\sum_{k\in\Bbb N} \lambda_-^{-k(s+1)}\chi_{ x_i+(B_{\ell_i+2\omega+1+k}\backslash B_{\ell_i+2\omega+k})}(x) \right),\;\;x\in{\Bbb R}^n.
        \end{align*}
Let $j,m\in \mathbb{N}$, $j<m$. We infer, for every $x\in \mathbb{R}^n$,
\begin{align*}
          M_N\left( \sum_{i=j}^mb_i\right)(x)& \leq \sum_{i=j}^m M_N(b_{i})(x) \\
         & \leq C\left(M_Nf(x)\sum_{i=j}^m\chi_{x_i+B_{\ell_i+2\omega}}(x)+\lambda\sum_{i=j}^m\sum_{k\in\Bbb N} \lambda_-^{-k(s+1)}\chi_{ x_i+(B_{\ell_i+2\omega+1+k}\backslash B_{\ell_i+2\omega+k})}(x) \right).
        \end{align*}
We also have that, for every $x\in x_i+(B_{\ell_i+2\omega+1+k}\backslash B_{\ell_i+2\omega+k})$, with $i,k\in \mathbb{N}$, $i\le m$,
$$M_{HL}\left(\chi_{ x_i+B_{\ell_i+2\omega}}\right)(x)\geq {1\over{|x_i+B_{\ell_i+2\omega+1+k}|}}\int_{x_i+B_{\ell_i+2\omega+1+k}}\chi_{ x_i+B_{\ell_i+2\omega}}(y)dy=b^{-k-1}.$$

We choose $r>1$ such that $rp,rq\in{\Bbb P}_1$. Then, we take $s\in \Bbb N$ such that $\lambda_-^{-s}b^r\leq 1$ and $N_0<s$. We get, for every $i\in \mathbb{N}$, $i\le m$,
\begin{align*}
        \sum_{k=0}^\infty \lambda_-^{-k(s+1)}\chi_{x_i+(B_{\ell_i+2\omega+1+k}\backslash B_{\ell_i+2\omega+k})}(x) & \leq C\max_{k\in\Bbb N}(\lambda_-^{-s-1}b^r)^k\left(M_{HL}\left(\chi_{ x_i+B_{\ell_i+2\omega}}\right)(x)\right)^r \\
        & \leq C \left(M_{HL}\left(\chi_{ x_i+B_{\ell_i+2\omega}}\right)(x)\right)^r,\;\;\;x\in (x_i+B_{\ell_i+2\omega})^c.
        \end{align*}
Hence we obtain
$$ M_N\left( \sum_{i=j}^m b_{i}\right)(x)\leq C_0\left(M_Nf(x)\sum_{i=j}^m\chi_{x_i+B_{\ell_i+2\omega}}(x)+\lambda\sum_{i=j}^m \left(M_{HL}\left(\chi_{ x_i+B_{\ell_i+2\omega}}\right)(x)\right)^r\right),\;\;\;x\in{\Bbb R}^n.$$

By using \cite[Lemma 2.3]{CrW},  since $\mathcal{L}^{p(\cdot),q(\cdot)}(\mathbb{R}^n)$ is a quasi Banach space we obtain
\begin{align}\label{X1}
       & \left\|M_N\left(\sum_{i=j}^m b_{i}\right)\right\|_{p(\cdot),q(\cdot)}\leq C\left(\left\|M_N(f)\sum_{i=j}^m\chi_{ x_i+B_{\ell_i+2\omega}}\right\|_{p(\cdot),q(\cdot)}\right.\nonumber \\
        & \left.+ \lambda\left\|\sum_{i=j}^m\left(M_{HL}\left(\chi_{ x_i+B_{\ell_i+2\omega}}\right)\right)^r\right\|_{p(\cdot),q(\cdot)}\right) \nonumber\\
       & = C\left(\left\|M_N(f)\sum_{i=j}^m\chi_{x_i+B_{\ell_i+2\omega}}\right\|_{p(\cdot),q(\cdot)}+ \lambda\left\|\left(\sum_{i=j}^m\left(M_{HL}\left(\chi_{ x_i+B_{\ell_i+2\omega}}\right)\right)^r\right)^{1/r}\right\|_{rp(\cdot),rq(\cdot)}^r\right)
        \end{align}
By using Proposition \ref{Prop vectorial} we get
\begin{align*}
        \left\|\left(\sum_{i=j}^m\left(M_{HL}\left(\chi_{ x_i+B_{\ell_i+2\omega}}\right)\right)^r\right)^{1/r}\right\|_{rp(\cdot),rq(\cdot)}^r & \leq C\left\|\left(\sum_{i=j}^m\chi_{ x_i+B_{\ell_i+2\omega}}\right)^{1/r}\right\|_{rp(\cdot),rq(\cdot)}^r \\
        & = C\left\|\sum_{i=j}^m\chi_{ x_i+B_{\ell_i+2\omega}}\right\|_{p(\cdot),q(\cdot)}.
        \end{align*}

>From (\ref{eq3.5}) and (\ref{X1}) it follows that

\begin{align*}
        \left\|M_N(\sum_{i=j}^mb_{i})\right\|_{p(\cdot),q(\cdot)} & \leq  C\left(\left\|M_N(f)\sum_{i=j}^m\chi_{ x_i+B_{\ell_i+2\omega}}\right\|_{p(\cdot),q(\cdot)}
        + \lambda \left\|\sum_{i=j}^m\chi_{x_i+B_{\ell_i+2\omega}}\right\|_{p(\cdot),q(\cdot)}\right) \\
       & \leq C \left\|M_N(f)\sum_{i=j}^m\chi_{x_i+B_{\ell_i+2\omega}}\right\|_{p(\cdot),q(\cdot)}\\
       &  \leq C \left\|M_N(f)\chi_{\cup_{i=j}^\infty(x_i+B_{\ell_i+2\omega})}\right\|_{p(\cdot),q(\cdot)}.
        \end{align*}

We define, for every $k\in \mathbb{N}$, $E_k=\cup_{i\in\Bbb N} (x_i+B_{\ell_i+2\omega})$. By (\ref{eq3.5}) there exists $C>0$ such that $\sum_{i=k}^\infty \chi_{x_i+B_{\ell_i+2\omega}}\le C\chi_{E_k}$, $k\in \mathbb{N}$. By (\ref{eq3.1}) and (\ref{eq3.2}), $\cup_{i\in\Bbb N}(x_i+B_{\ell_i-\omega})\subset \Omega_\lambda$, and then $\sum_{i\in\Bbb N} |x_i+B_{\ell_i-\omega}|=b^{-\omega}\sum_{i\in\Bbb N} b^{\ell_i}\le |\Omega_\lambda|<\infty$, where $\Omega_\lambda=\{x\in \mathbb{R}^n:\,M_N(f)(x)>\lambda\}$. We deduce that
$$
|E_k|\le \sum_{i=k}^\infty |x_i+B_{\ell_i+2\omega}|=b^{2\omega}\sum_{i=k}^\infty b^{\ell_i},\,\,\,k\in \mathbb{N}.
$$
Proposition \ref{prop2.1} implies that
$$
\lim_{k\to \infty}\|M_N(f)\chi_{E_k}\|_{p(\cdot),q(\cdot)}=0.
$$
Hence, the sequence $\{\sum_{i=0}^k b_i\}_{k\in\Bbb N}$ is Cauchy in $H^{p(\cdot),q(\cdot)}({\Bbb R}^n,A)$. Since $H^{p(\cdot),q(\cdot)}({\Bbb R}^n,A)$ is complete (Proposition \ref{prop2.4}), the series $\sum_{i\in\Bbb N} b_i$ converges in $H^{p(\cdot),q(\cdot)}({\Bbb R}^n,A)$.

(ii) In order to prove this property we can proceed as in the proof of (i). Assume that $j\in\Bbb Z$. We define $\Omega_j=\{x\in{\Bbb R}^n:\;M_Nf(x)>2^j\}$. By putting $b_j=\sum_{i\in\Bbb N} b_{i,j}$, since, as we have just proved in (i), the last series converges in $H^{p(\cdot),q(\cdot)}({\Bbb R}^n,A)$ and then in $S'(\mathbb{R}^n)$, we obtain, for a chosen $r>1$ verifying that $rp,rq\in{\Bbb P}_1$,
$$M_N(b_{j})(x)\leq C_0\left(M_Nf(x)\chi_{\Omega_j}(x)+2^j\sum_{i\in\Bbb N}\left(M_{HL}\left(\chi_{ x_i+B_{\ell_i+2\omega}}\right)(x)\right)^r\right),\;\;\;x\in{\Bbb R}^n.$$
It follows that
\begin{align}\label{X111}
       & \|M_N(b_{j})\|_{p(\cdot),q(\cdot)}\nonumber\\
       & \leq C\left(\|M_N(f)\chi_{\Omega_j}\|_{p(\cdot),q(\cdot)}+ 2^j\left\|\left(\sum_{i\in\Bbb N}\left(M_{HL}\left(\chi_{ x_i+B_{\ell_i+2\omega}}\right)\right)^r\right)^{1/r}\right\|_{rp(\cdot),rq(\cdot)}^r\right)
        \end{align}
>From Proposition \ref{Prop vectorial} we get
\begin{align*}
       \left\|\left(\sum_{i\in\Bbb N}\left(M_{HL} \left(\chi_{\{ x_i+B_{\ell_i+2\omega}\}}\right)\right)^r\right)^{1/r}\right\|_{rp(\cdot),rq(\cdot)}^r&\leq C\left\|\left(\sum_{i\in\Bbb N}\chi_{x_i+B_{\ell_i+2\omega}}\right)^{1/r}\right\|_{rp(\cdot),rq(\cdot)}^r \\
        & = C\left\|\sum_{i\in\Bbb N}\chi_{x_i+B_{\ell_i+2\omega}}\right\|_{p(\cdot),q(\cdot)}\\
        & \le C\|\chi_{\Omega_j}\|_{p(\cdot),q(\cdot)}.
       \end{align*}

>From (\ref{X111}) it follows that

\begin{align*}
        \|M_N(b_{j})\|_{p(\cdot),q(\cdot)}  & \leq  C\left(\|M_N(f)\chi_{\Omega_j}\|_{p(\cdot),q(\cdot)}
        + 2^j \|\chi_{\Omega_j}\|_{p(\cdot),q(\cdot)}\right) \\
       & \leq C \|M_N(f)\chi_{\Omega_j}\|_{p(\cdot),q(\cdot)}.
        \end{align*}

Since $f\in  H^{p(\cdot),q(\cdot)}({\Bbb R}^n,A)$, by invoking again \cite[Lemma 2.3]{CrW}, we have that
$$\|M_N(f)\|_{p(\cdot),q(\cdot)}^{1/r}=\|(M_N(f))^{1/r}\|_{rp(\cdot),rq(\cdot)}<\infty.$$
Then, by \cite[Theorem 2.8]{EKS} (see \cite[Definition 2.5 vii)]{EKS}), $M_N(f)(x)<\infty$, a.e. $x\in{\Bbb R}^n$. Hence, $M_N(f)\chi_{\Omega_j}\downarrow 0$, as $j\rightarrow +\infty$, for a.e. $x\in{\Bbb R}^n$. According to Proposition \ref{prop2.1} we conclude that $\|M_N(f)\chi_{\Omega_j}\|_{p(\cdot),q(\cdot)}\rightarrow 0$, as $j\rightarrow +\infty$. Hence, $\|M_N(b_j)\|_{p(\cdot),q(\cdot)}\rightarrow 0$, as $j\rightarrow +\infty$, and  $\|f-g_j\|_{H^{p(\cdot),q(\cdot)}({\Bbb R}^n,A)}\rightarrow 0$, as $j\rightarrow +\infty$.

The proof of this Proposition is now finished.
\end{proof}

By $C_c^\infty(\mathbb{R}^n)$ we denote the space of smooth functions with compact support in $\mathbb{R}^n$. We say that a distribution $h\in S'(\mathbb{R}^n)$ is in $L^1_{loc}(\mathbb{R}^n)$ when there exists a (unique) $H\in L^1_{loc}(\mathbb{R}^n)$ such that
$$
<h,\phi>=\int_{\mathbb{R}^n}H(x)\phi(x)dx,\,\,\,\phi\in C_c^\infty(\mathbb{R}^n).
$$
The space $S'(\mathbb{R}^n)\cap L^1_{loc}(\mathbb{R}^n)$ is also sometimes denoted by $S_r(\mathbb{R}^n)$ and it was studied, for instance, in \cite{Die}, \cite{Szm2} and \cite{Szm1}.

\begin{Prop}\label{prop3.3}
If $f\in S'({\Bbb R}^n)$, $\lambda >0$, $s,N\in \mathbb{N}$, $N\ge 2$ and $s<N$, and $f=g+\sum_{i\in\Bbb N}b_i$ is the anisotropic Calder\'on-Zygmund decomposition of $f$ associated to $M_N(f)$ of height $\lambda$ and degree $s$, then $g\in L_{loc}^1({\Bbb R}^n)$.
\end{Prop}
\begin{proof}
Let $\lambda >0$, $N\in\Bbb N$, $N\geq 2$, $s\in \mathbb{N}$, $s<N$, and $f\in S'({\Bbb R}^n)$ such that $|\Omega_\lambda|<\infty$ where  $\Omega_\lambda=\{x\in{\Bbb R}^n:\;M_N(f)(x)>\lambda\}$. We write $f=g+\sum_{i\in\Bbb N} b_{i}$ the Calder\'on-Zygmund decomposition of $f$ associated to $M_N(f)$ of height $\lambda$ and degree $s$.

According to \cite[Lemma 5.9]{Bow1} we have that
$$M_N(g)(x)\leq C\lambda\sum_{i\in\Bbb N}{\lambda}_-^{-t_i(s+1)}+M_N(f)(x)\chi_{\Omega_\lambda^c}(x),\;\;\;x\in{\Bbb R}^n,$$
where
$$t_i=t_i(x)=\left\{\begin{array}{l}
t,\;\;\;\mbox{if}\; x\in  x_i+(B_{\ell_i+2\omega+t+1}\backslash B_{\ell_i+2\omega+t}),\;\mbox{for some}\;{t\in\Bbb N} ,\\
 \\
 0,\;\;\;\mbox{otherwise}.
 \end{array}\right.$$

As it was shown in the proof of  \cite[Lemma 5.10 (i), p. 34]{Bow1} we get
$$\int_{{\Bbb R}^n}\sum_{i\in\Bbb N}{\lambda}_-^{-t_i(x)(s+1)}dx\leq C|\Omega_\lambda|.$$
Then, since $M_N(f)(x)\leq\lambda$, $x\in\Omega_\lambda^c$, we obtain that $M_N(g)\in L_{loc}^1({\Bbb R}^n)$.

Let $\varphi\in S(\mathbb{R}^n)$. Since for a certain $C>0$ $g*\varphi_k\le CM_N(g)$, $k\in \mathbb{N}$, by proceeding as in the proof of \cite[Theorem 3.9]{Bow1} we can prove that, for every compact subset $F$ of ${\Bbb R}^n$ there exists a sequence $\{k_j\}_{j\in\Bbb N}\subset\Bbb Z$ such that $k_j\rightarrow -\infty$, as  $j\rightarrow \infty$, and $g*\varphi_{k_j}\rightarrow G_F$, as  $j\rightarrow \infty$, in the weak topology of $L^1(F)$ for a certain $G_F\in L^1(F)$. A diagonal argument allows us to get a sequence $\{k_j\}_{j\in\Bbb N}\subset\Bbb Z$ such that $k_j\rightarrow -\infty$, as  $j\rightarrow \infty$, and $g*\varphi_{k_j}\rightarrow G$, in the weak * topology of ${\mathcal M}(K)$ (the space of complex measures supported in $K$) for every compact subset $K$ of ${\Bbb R}^n$, being $G\in L^1_{loc}(\mathbb{R}^n)$. According to \cite[Lemma 3.8]{Bow1}  $g*\varphi_{k_j}\rightarrow g$,  as  $j\rightarrow \infty$ in $S'({\Bbb R}^n)$. If $\phi\in C_c^\infty({\Bbb R}^n)$, we have that

\begin{equation}\label{Z1}
<g,\phi>=\lim_{j\rightarrow \infty}\int_{{\Bbb R}^n}(g*\varphi_{k_j})(x)\phi(x)dx=\int_{{\Bbb R}^n}G(x)\phi(x)dx.
\end{equation}

Since $C_c^\infty(\mathbb{R}^n)$ is a dense subspace of $S(\mathbb{R}^n)$, $g$ is characterized by (\ref{Z1}).
\end{proof}

\begin{Cor}\label{cor1}
Assume that $p,q\in{\Bbb P}_0$. Then, $L_{loc}^1({\Bbb R}^n)\cap H^{p(\cdot),q(\cdot)}({\Bbb R}^n,A)$ is a dense subspace of $H^{p(\cdot),q(\cdot)}({\Bbb R}^n,A)$.
\end{Cor}
\begin{proof}
This property is a consequence of Propositions \ref{prop3.2} and \ref{prop3.3}.
\end{proof}

We finish this section with a convergence property for the good parts of Calder\'on-Zygmund decomposition of distributions in $L_{loc}^1({\Bbb R}^n)\cap H^{p(\cdot),q(\cdot)}({\Bbb R}^n,A)$ which we will use in the proof of atomic decompositions of  the elements of $H^{p(\cdot),q(\cdot)}({\Bbb R}^n,A)$.

\begin{Prop}\label{Propnew}
Assume that $p,q\in{\Bbb P}_0$, and $f\in L_{loc}^1({\Bbb R}^n)\cap H^{p(\cdot),q(\cdot)}({\Bbb R}^n,A)$. For every $j\in \mathbb{N}$, $f=g_j+\sum_{i\in\Bbb N} b_{i,j}$ is the anisotropic Calder\'on-Zygmund decomposition of $f$ associated to $M_N(f)$ of height $2^j$ and degree $s$, with $s,N\in \mathbb{N}$, $s\ge s_0$ and $N> max\{s,N_0\}$, where $N_0$ is as in Theorem \ref{Th1.1} and $s_0$ is as in Proposition \ref{prop3.2}. Then, $g_j\to 0$, as $j\to -\infty$, in $S'(\mathbb{R}^n)$.
\end{Prop}

\begin{proof} Since $f\in L^1_{loc}(\mathbb{R}^n)$, there exists a unique $F\in L^1_{loc}(\mathbb{R}^n)$ such that
$$
<f,\phi>=\int_{\mathbb{R}^n}F(x)\phi(x)dx,\,\,\,\phi\in C_c^\infty(\mathbb{R}^n).
$$
According to Proposition \ref{prop3.3}, for every $j\in \mathbb{Z}$, there exists a unique $G_j\in L^1_{loc}(\mathbb{R}^n)$ for which
\begin{equation}\label{M0}
<g_j,\phi>=\int_{\mathbb{R}^n}G_j(x)\phi(x)dx,\,\,\,\phi\in C_c^\infty(\mathbb{R}^n).
\end{equation}
Let $j\in \mathbb{Z}$ and $\phi\in C_c^\infty(\mathbb{R}^n)$. We are going to see that
$$
\sum_{i\in\Bbb N}\int_{\mathbb{R}^n}|(F(x)-P_{i,j}(x))\zeta_{i,j}(x)||\phi(x)|dx<\infty.
$$
For every $i\in \mathbb{N}$, by \cite[Lemma 5.3]{Bow1} we have that
\begin{align*}
& \int_{\mathbb{R}^n}|(F(x)-P_{i,j}(x))\zeta_{i,j}(x)||\phi(x)|dx\\
&\le C\left(\int_{(x_{i,j}+B_{l_{i,j}+\omega})\cap \rm{supp}(\phi)}|F(x)|dx+2^j|(x_{i,j}+B_{l_{i,j}+\omega})\cap \rm{supp}(\phi)|\right).
\end{align*}
Then,
$$
\sum_{i\in\Bbb N}\int_{\mathbb{R}^n}|(F(x)-P_{i,j}(x))\zeta_{i,j}(x)||\phi(x)|dx\le C\left(\int_{\rm{supp}(\phi)}|F(x)|dx+2^j|\rm{supp}(\phi)|\right).
$$
Hence, from Proposition \ref{prop3.2}, (i), we get
\begin{align*}
\int_{\mathbb{R}^n}\sum_{i\in\Bbb N} (F(x)-P_{i,j}(x))\zeta_{i,j}(x)\phi(x)dx& = \sum_{i\in\Bbb N} \int_{\mathbb{R}^n} (F(x)-P_{i,j}(x))\zeta_{i,j}(x)\phi(x)dx\\
& =\sum_{i\in\Bbb N} <(f-P_{i,j})\zeta_{i,j},\phi>\\
& = <\sum_{i\in\Bbb N} (f-P_{i,j})\zeta_{i,j},\phi>.
\end{align*}
Then, there exists a measurable subset $E\subset \mathbb{R}^n$ such that $|\mathbb{R}^n\setminus E|=0$, and
$$
G_j(x)=F(x)-\sum_{i\in\Bbb N} (F(x)-P_{i,j}(x))\zeta_{i,j}(x),\,\,\,x\in E\,\,\,\rm{and}\,\,\,j\in \mathbb{Z},
$$
for a suitable sense of the convergence of series. Note that we have used a diagonal argument to justify the convergence for every $j\in \mathbb{Z}$.

We can write
$$
G_j(x)=F(x)\chi_{\Omega_j^c}(x)-\sum_{i\in\Bbb N} P_{i,j}(x)\zeta_{i,j}(x),\,\,\,x\in E\,\,\,\rm{and}\,\,\,j\in \mathbb{Z},
$$
where $\Omega_j=\{x\in \mathbb{R}^n:M_N(f)(x)>2^j\}$, $j\in \mathbb{Z}$. Note that the last series is actually a finite sum for every $x\in \mathbb{R}^n$.

Let $j\in \mathbb{Z}$. According to \cite[Lemma 5.3]{Bow1} we obtain
$$
|G_j(x)|\le C2^j,\,\,\,{\rm a.e.}\,\,\,x\in \Omega_j.
$$
On the other hand, $G_j(x)=F(x)$, a.e. $x\in \Omega_j^c$. Also, we have that
$$
|F|\le \sup_{k\in \mathbb{Z},\,\,\varphi\in C_c^\infty(\mathbb{R}^n)\cap S_N}|f*\varphi_k|\le M_N(f).
$$
Then, $|G_j(x)|\le C2^j$, a.e. $x\in \Omega_j^c$. Hence, we conclude that
\begin{equation}\label{M1}
|G_j(x)|\le C2^j,\,\,\,{\rm a.e.}\,\,\,x\in \mathbb{R}^n.
\end{equation}

We consider the functional $T_j$ defined on $S(\mathbb{R}^n)$ by
$$
T_j(\phi)=\int_{\mathbb{R}^n}G_j(x)\phi(x)dx,\,\,\,\phi\in S(\mathbb{R}^n).
$$
>From (\ref{M1}) we deduce that $T_j\in S'(\mathbb{R}^n)$. By (\ref{M0}), $T_j(\phi)=<g_j,\phi>$, $\phi\in C_c^\infty(\mathbb{R}^n)$. Then,
$$
<g_j,\phi>=\int_{\mathbb{R}^n}G_j(x)\phi(x)dx,\,\,\,\phi\in S(\mathbb{R}^n),
$$
and, again from (\ref{M1}) it follows that $g_j\to 0$, as $j\to -\infty$, in $S'(\mathbb{R}^n)$.
\end{proof}

\section{Atomic characterization (Proof of Theorem \ref{Th1.2})} \label{sec:4}
As we mentioned in the introduction, we are going to prove Theorem \ref{Th1.2} in two steps, firstly in the case that $r=\infty$ and then when $r<\infty$.

%
\subsection{Proof of Theorem \ref{Th1.2} when $r=\infty$.}

\noindent{\bf (i)} Suppose that, for every $j\in\Bbb N$, $a_j$ is a $(p(\cdot),q(\cdot),\infty,s)$-atom associated with  $x_j\in{\Bbb R}^n$ and $\ell_j\in\Bbb Z$. Here $s\in \mathbb{N}$ will be fixed later. Assume also that $(\lambda_j)_{j\in \mathbb{N}}\subset (0,\infty)$ and that
\begin{equation}\label{Y0}
\left\|\sum_{j\in\Bbb N}{{\lambda_j}\over{\|\chi_{x_j+B_{\ell_j}}\|_{p(\cdot),q(\cdot)}}}
       \chi_{x_j+B_{\ell_j}}\right\|_{p(\cdot),q(\cdot)}< \infty.
\end{equation}
We are going to show that the series $\sum_{j\in\Bbb N}\lambda_ja_j$ converges in $H^{p(\cdot),q(\cdot)}({\Bbb R}^n,A)$. Let $\ell,m\in\Bbb N$, $\ell<m$. We define
$$f_{\ell,m}=\sum_{j=\ell}^m\lambda_ja_j,$$
and we take $\varphi\in S({\Bbb R}^n)$. We have that

\begin{align}\label{Y1}
       \|M_{\varphi}(f_{\ell,m})\|_{p(\cdot),q(\cdot)}& \leq  \left\|\sum_{j=\ell}^m{\lambda_jM_{\varphi}(a_j)}\right\|_{p(\cdot),q(\cdot)}\nonumber \\
       & \leq C\left(\left\|\sum_{j=\ell}^m\lambda_jM_{\varphi}(a_j)
       \chi_{x_j+B_{\ell_j+\omega}}\right\|_{p(\cdot),q(\cdot)}+\left\|\sum_{j=\ell}^m\lambda_jM_{\varphi}(a_j)
       \chi_{x_j+B_{\ell_j+\omega}^c}\right\|_{p(\cdot),q(\cdot)}\right)\nonumber \\
       & =I_1+I_2.
       \end{align}
       We now estimate $I_i,\;i=1,2$. Firstly we study $I_1$. Let $j\in\Bbb N$. Since $a_j$ is a $(p(\cdot),q(\cdot),\infty,s)$-atom, we can write
$$M_{\varphi}(a_j)(x)\leq \|a_j\|_\infty\|\varphi\|_1\leq C\|\chi_{x_j+B_{\ell_j}}\|_{p(\cdot),q(\cdot)}^{-1},\,\,\,x\in \mathbb{R}^n.$$

By defining $g_j=\chi_{x_j+B_{\ell_j}}(\|\chi_{x_j+B_{\ell_j}}\|_{p(\cdot),q(\cdot)}^{-1}\lambda_j)^\alpha$ it follows that
\begin{align*}
        M_{HL}g_j(x)& \geq (\|\chi_{x_j+B_{\ell_j}}\|_{p(\cdot),q(\cdot)}^{-1}\lambda_j)^{\alpha}{1\over{|B_{\ell_j+\omega}|}}\int_{x_j+B_{\ell_j+\omega}}
       \chi_{x_j+B_{\ell_j}}(y)dy \\
       & =b^{-\omega}(\|\chi_{x_j+B_{\ell_j}}\|_{p(\cdot),q(\cdot)}^{-1}\lambda_j)^{\alpha},\;\;\;x\in x_j+B_{\ell_j+\omega},
       \end{align*}
where $\alpha\in (0,1)$ is such that $p(\cdot)/\alpha,q(\cdot)/\alpha\in{\Bbb P}_1$.
According to Proposition \ref{Prop vectorial} and \cite[Lemma 2.3]{CrW} we have that
\begin{align}\label{Y2}
       I_1 & \leq  C \left\|\sum_{j=\ell}^m\lambda_j\|\chi_{x_j+B_{\ell_j}}\|_{p(\cdot),q(\cdot)}^{-1}
       \chi_{x_j+B_{\ell_j}+\omega}\right\|_{p(\cdot),q(\cdot)} \nonumber\\
       & \leq C\left\|\sum_{j=\ell}^m(M_{HL}g_j)^{1/\alpha}\right\|_{p(\cdot),q(\cdot)}\nonumber\\
        & \leq C\left\|\left(\sum_{j=\ell}^m(M_{HL}g_j)^{1/\alpha}\right)^{\alpha}\right\|_{p(\cdot)/\alpha,q(\cdot)/\alpha}^{1/\alpha}\nonumber \\
       & \leq C\left\|\left(\sum_{j=\ell}^mg_j^{1/\alpha}\right)^{\alpha}\right\|_{p(\cdot)/\alpha,q(\cdot)/\alpha}^{1/\alpha}\nonumber\\
       & = C \left\|\sum_{j=\ell}^m\lambda_j\|\chi_{x_j+B_{\ell_j}}\|_{p(\cdot),q(\cdot)}^{-1}
       \chi_{x_j+B_{\ell_j}}\right\|_{p(\cdot),q(\cdot)}.
        \end{align}

Suppose now that $a$ is a $(p(\cdot),q(\cdot),\infty,s)$-atom associated with $z\in{\Bbb R}^n$ and $k\in\Bbb Z$. Let $m\in\Bbb N$. By proceeding as in \cite[p. 19 and 20]{Bow1} we obtain
\begin{align*}
        M_\varphi(a)(x)& \leq C{1\over{\|\chi_{z+B_{k}}\|_{p(\cdot),q(\cdot)}}}(b\lambda_-^{s+1})^{-m} \\
        & \leq C{1\over{\|\chi_{z+B_{k}}\|_{p(\cdot),q(\cdot)}}}b^{m(\gamma-1)}\lambda_-^{-m(s+1)}\left(
        {1\over{|B_{k+m+\omega+1}|}}\int_{z+B_{k+m+\omega+1}}
       \chi_{z+B_{k}}(y)dy\right)^\gamma \\
       & \leq C{{b^{m(\gamma-1)}\lambda_-^{-m(s+1)}}\over{\|\chi_{z+B_{k}}\|_{p(\cdot),q(\cdot)}}}\left(
        M_{HL}(\chi_{z+B_{k}})(x)\right)^\gamma,\;\;\;x\in z+(B_{k+m+\omega+1}\backslash B_{k+m+\omega}).
       \end{align*}
 Here $\gamma$ is chosen such that $\gamma p(\cdot),\gamma q(\cdot)\in{\Bbb P}_1$. We now take $s\in\Bbb N$, satisfying that  $b^{\gamma-1}\lambda_-^{-(s+1)}\leq 1$. We obtain
$$M_\varphi(a)(x)\leq C{1\over{\|\chi_{z+B_{k}}\|_{p(\cdot),q(\cdot)}}}\left(
        M_{HL}(\chi_{z+B_{k}})(x)\right)^\gamma,
        \,\,\,x\notin z+B_{k+\omega}.$$

By proceeding as above we get
\begin{align}\label{Y3}
       I_2 & \leq  C\left \|\sum_{j=\ell}^m\lambda_j\|\chi_{x_j+B_{\ell_j}}\|_{p(\cdot),q(\cdot)}^{-1}
       \left(M_{HL}(\chi_{x_j+B_{\ell_j}})\right)^\gamma\right\|_{p(\cdot),q(\cdot)}\nonumber \\
       & =C\left \|\left(\sum_{j=\ell}^m\left(\lambda_j^{1/\gamma}\|\chi_{x_j+B_{\ell_j}}\|_{p(\cdot),q(\cdot)}^{-1/\gamma}
       M_{HL}(\chi_{x_j+B_{\ell_j}})\right)^\gamma\right)^{1/\gamma}\right\|_{\gamma p(\cdot),\gamma q(\cdot)}^\gamma \nonumber\\
       & \leq  C \left\|\sum_{j=\ell}^m\lambda_j\|\chi_{x_j+B_{\ell_j}}\|_{p(\cdot),q(\cdot)}^{-1}
       \chi_{x_j+B_{\ell_j}}\right\|_{p(\cdot),q(\cdot)}.
        \end{align}

By combining (\ref{Y1}), (\ref{Y2}) and (\ref{Y3}) we infer that the sequence $(\sum_{j=0}^k\lambda_ja_j)_{k\in\Bbb N}$ is Cauchy in $H^{p(\cdot),q(\cdot)}({\Bbb R}^n,A)$. Since $H^{p(\cdot),q(\cdot)}({\Bbb R}^n,A)$ is complete (Proposition \ref{prop2.4}), the series $\sum_{j\in\Bbb N}\lambda_ja_j$ converges in $H^{p(\cdot),q(\cdot)}({\Bbb R}^n,A)$. Moreover, we get

$$\left\|\sum_{j\in\Bbb N}\lambda_ja_j\right\|_{H^{p(\cdot),q(\cdot)}({\Bbb R}^n,A)}\leq  C \left\|\sum_{j\in\Bbb N}{{\lambda_j}\over{\|\chi_{x_j+B_{\ell_j}}\|_{p(\cdot),q(\cdot)}}}
       \chi_{x_j+B_{\ell_j}}\right\|_{p(\cdot),q(\cdot)}.$$

\noindent{\bf(ii)} Assume that $f\in H^{p(\cdot),q(\cdot)}({\Bbb R}^n,A)\cap L_{loc}^1({\Bbb R}^n)$,  $s\ge s_0$ ($s_0$ was defined in Proposition \ref{prop3.2}) and $N>\max\{N_0,s\}$ ($N_0$ was defined in Theorem \ref{Th1.1}). We recall that $H^{p(\cdot),q(\cdot)}({\Bbb R}^n,A)\cap L_{loc}^1({\Bbb R}^n)$  is a dense subspace of $H^{p(\cdot),q(\cdot)}({\Bbb R}^n,A)$ (Corollary \ref{cor1}). Let $j\in\Bbb Z$. We define $\Omega_j=\{x\in{\Bbb R}^n:\;M_N(f)(x)>2^j\}$. According to \cite[Chapter 1, Section 5]{Bow1} we can write $f=g_j+\sum_{k\in\Bbb N}b_{j,k}$, that is, the Calder\'on-Zygmund decomposition of degree $s$ and height $2^j$ associated with $M_Nf$. The properties of $g_j$ and $b_{j,k}$ will be specified when we need each of them.

As it was proved in Proposition \ref{prop3.2}, (ii), $g_j\rightarrow f$, as $j\rightarrow +\infty$, in both $H^{p(\cdot),q(\cdot)}({\Bbb R}^n,A)$ and $S'({\Bbb R}^n)$, and in Proposition \ref{Propnew} $g_j\rightarrow 0$, as  $j\rightarrow -\infty$, in $S'({\Bbb R}^n)$. We have that
$$f=\sum_{j\in\Bbb Z}(g_{j+1}-g_j),\;\;\;\mbox{in}\;\;S'({\Bbb R}^n).$$
As in \cite[p. 38]{Bow1} we can write, for every $j\in\Bbb Z$,
$$g_{j+1}-g_j=\sum_{i\in\Bbb N} h_{i,j}, \;\;\;\mbox{in}\;\;S'({\Bbb R}^n),$$
where
$$h_{i,j}=(f-P_i^j)\zeta_i^j-\sum_{k\in\Bbb N}((f-P_k^{j+1})\zeta_i^j-P_{i,k}^{j+1})\zeta_k^{j+1},\;\;\;i\in\Bbb N.$$
According to the properties of the polynomials $P$'s and the functions $\zeta$'s it follows that, for every $P\in{\mathcal P}_s$,
$$\int_{{\Bbb R}^n}h_{i,j}(x) P(x)\;dx=0,\;\;\;i,j\in\Bbb N.$$
We also have that, for certain $C_0>0$, $\|h_{i,j}\|_{\infty}\leq C_02^j$ and ${\rm supp}\;h_{i,j}\subset x_{i,j}+B_{\ell_{i,j}+4\omega}$, for every $i,j\in\Bbb N$ (\cite[(6.12) and (6.13), p. 38]{Bow1}). Hence, for every $i,j\in\Bbb N$, the function $a_{i,j}=h_{i,j}2^{-j}C_0^{-1}\|\chi_{x_{i,j}+B_{\ell_{i,j}+4\omega}}\|_{p(\cdot),q(\cdot)}^{-1}$ is a $(p(\cdot),q(\cdot),\infty,s)$-atom. Moreover,
\begin{equation}\label{con1}
f=\sum_{i\in\Bbb N,j\in\Bbb Z}\lambda_{i,j}a_{i,j}\;\;\;\mbox{in}\;\;S'({\Bbb R}^n),
\end{equation}
where $\displaystyle\lambda_{i,j}=2^{j}C_0\|\chi_{x_{i,j}+B_{\ell_{i,j}+4\omega}}\|_{p(\cdot),q(\cdot)}$, for every $i\in\Bbb N,\;j\in\Bbb Z$.

We are going to explain the convergence of the double series in (\ref{con1}).

We now choose $\beta >1$ such that $\beta p,\beta q \in{\Bbb P}_1$. Assume that $\pi=(\pi_1,\pi_2):\mathbb{N}\rightarrow \mathbb{N}\times \mathbb{Z}$ is a bijection. By proceeding as before we get, for every $k\in \mathbb{N}$,
\begin{align*}
        \Big\|\sum_{m=0}^k{{\lambda_{\pi(m)}}\over{\|\chi_{x_{\pi(m)}+B_{\ell_{\pi(m)}+4\omega}}\|_{p(\cdot),q(\cdot)}}} &
       \chi_{x_{\pi(m)}+B_{\ell_{\pi(m)}+4\omega}}\Big\|_{p(\cdot),q(\cdot)}
       \leq C\left\|\sum_{m=0}^k2^{\pi_2(m)}
       \chi_{x_{\pi(m)}+B_{\ell_{\pi(m)}+4\omega}}\right\|_{p(\cdot),q(\cdot)} \\
       & \leq C\left\|\sum_{m=0}^k\left(2^{\pi_2(m)/\beta}
       \chi_{x_{\pi(m)}+B_{\ell_{\pi(m)}+4\omega}}\right)^{\beta}\right\|_{p(\cdot),q(\cdot)} \\
       & \leq C\left\|\sum_{m=0}^k\left(2^{\pi_2(m)/\beta}
       M_{HL}\left(\chi_{x_{\pi(m)}+B_{\ell_{\pi(m)}+2\omega}}\right)\right)^{\beta}\right\|_{p(\cdot),q(\cdot)} \\
       & =  C\left\|\left(\sum_{m=0}^k\left(
       M_{HL}\left(2^{\pi_2(m)/\beta}\chi_{x_{\pi(m)}+B_{\ell_{\pi(m)}+2\omega}}\right)\right)^{\beta}\right)^{1/\beta}
       \right\|_{\beta p(\cdot),\beta q(\cdot)}^{\beta} \\
       & \leq C\left\|\left(\sum_{m=0}^k
       2^{\pi_2(m)}\chi_{x_{\pi(m)}+B_{\ell_{\pi(m)}+2\omega}}\right)^{1/\beta}\right\|_{\beta p(\cdot),\beta q(\cdot)}^{\beta} \\
       & \le C\left\|\sum_{j\in\Bbb Z}
       2^{j}\sum_{i\in\Bbb N}\chi_{x_{i,j}+B_{\ell_{i,j}+2\omega}}\right\|_{p(\cdot),q(\cdot)} \\
       & \leq C\left\|\sum_{j\in\Bbb Z}2^j\chi_{\Omega_j}\right\|_{p(\cdot),q(\cdot)}.
        \end{align*}

Since $f\in H^{p(\cdot),q(\cdot)}({\Bbb R}^n,A)$, by \cite[Theorem 2.8 and Definition 2.5 vii)]{EKS}, $M_N(f)(x)<\infty$, a.e. $x\in{\Bbb R}^n$. Let $x\in{\Bbb R}^n$ such that $M_N(f)(x)<\infty$. There exists $j_0\in\Bbb Z$ such that $2^{j_0}<M_N(f)(x)\leq 2^{j_0+1}$. We have that
$$\sum_{j\in\Bbb Z}2^j\chi_{\Omega_j}(x)=\sum_{j\leq j_0}2^j=2^{j_0+1}\leq 2M_N(f)(x).$$
We conclude that
\begin{align*}
&\left\|\sum_{m=0}^k{{\lambda_{\pi(m)}}\over{\|\chi_{x_{\pi(m)}+B_{\ell_{\pi(m)}+4\omega}}\|_{p(\cdot),q(\cdot)}}}
       \chi_{x_{\pi(m)}+B_{\ell_{\pi(m)}+4\omega}}\right\|_{p(\cdot),q(\cdot)}\\
       &=\left\|\left(\sum_{m=0}^k{{\lambda_{\pi(m)}}\over{\|\chi_{x_{\pi(m)}+B_{\ell_{\pi(m)}+4\omega}}\|_{p(\cdot),q(\cdot)}}}
       \chi_{x_{\pi(m)}+B_{\ell_{\pi(m)}+4\omega}}\right)^{1/\beta}\right\|_{\beta p(\cdot),\beta q(\cdot)}^\beta\leq C\|f\|_{H^{p(\cdot),q(\cdot)}({\Bbb R}^n,A)},
\end{align*}
where $C>0$ does not depend on $(k,\pi)$.

According to \cite[Theorem 2.8 and Definition 2.5, v)]{EKS} we deduce that
$$
\left\|\sum_{m\in\Bbb N}{{\lambda_{\pi(m)}}\over{\|\chi_{x_{\pi(m)}+B_{\ell_{\pi(m)}+4\omega}}\|_{p(\cdot),q(\cdot)}}}
       \chi_{x_{\pi(m)}+B_{\ell_{\pi(m)}+4\omega}}\right\|_{p(\cdot),q(\cdot)}\leq C\|f\|_{H^{p(\cdot),q(\cdot)}({\Bbb R}^n,A)}.
       $$

>From the property we have just established in the part (i) of this proof we deduce that the series $\sum_{m\in\Bbb N} \lambda_{\pi(m)}a_{\pi(m)}$ converges both in $H^{p(\cdot),q(\cdot)}({\Bbb R}^n,A)$ and $S'(\mathbb{R}^n)$. Hence, for every $\phi\in S(\mathbb{R}^n)$, the series $\sum_{m\in\Bbb N} \lambda_{\pi(m)}<a_{\pi(m)},\phi>$ converges in $\mathbb{C}$.

Also we have that if $\Lambda:\mathbb{N}\times \mathbb{Z}\rightarrow \mathbb{N}\times \mathbb{Z}$ is a bijection, then the series $\sum_{m\in\Bbb N} \lambda_{\Lambda\circ\pi(m)}<a_{\Lambda\circ\pi(m)},\phi>$ converges in $\mathbb{C}$, for every $\phi\in S(\mathbb{R}^n)$. In other words, the series $\sum_{m\in\Bbb N} \lambda_{\pi(m)}<a_{\pi(m)},\phi>$ converges unconditionally in $\mathbb{C}$, for every $\phi\in S(\mathbb{R}^n)$. Hence, $\sum_{m\in\Bbb N} \lambda_{\pi(m)}|<a_{\pi(m)},\phi>|<\infty$, for every $\phi\in S(\mathbb{R}^n)$.

Let $\phi\in S(\mathbb{R}^n)$. Since $\sum_{m\in\Bbb N} \lambda_{\pi(m)}|<a_{\pi(m)},\phi>|<\infty$, the double series $\sum_{(i,j)\in \mathbb{N}\times \mathbb{Z}}\lambda_{i,j}<a_{i,j},\phi>$ is summable, that is, $\sup_{m\in \mathbb{N}}\sum_{1\leq i\le m,\,|j|\le m}\lambda_{i,j}|<a_{i,j},\phi>|<\infty$. Then, for every bijection $\pi:\mathbb{N}\rightarrow\mathbb{N}\times \mathbb{Z}$, we have that
$$
<f,\phi>=\sum_{i\in\Bbb N}\left(\sum_{j\in\Bbb Z}\lambda_{i,j}<a_{i,j},\phi>\right)=\sum_{m\in\Bbb N} \lambda_{\pi(m)}<a_{\pi(m)},\phi>.
$$

Suppose now that $f\in H^{p(\cdot),q(\cdot)}({\Bbb R}^n,A)$. Then, there exists a sequence $\{f_j\}_{j\in\Bbb N}$ in $L^1_{loc}({\Bbb R}^n)\cap H^{p(\cdot),q(\cdot)}({\Bbb R}^n,A)$ such that $f_1=0$, $f_j\rightarrow f$, as $j\rightarrow\infty$, in $H^{p(\cdot),q(\cdot)}({\Bbb R}^n,A)$, and $\|f_{j+1}-f_j\|_{H^{p(\cdot),q(\cdot)}({\Bbb R}^n,A)}<2^{-j}\|f\|_{H^{p(\cdot),q(\cdot)}({\Bbb R}^n,A)}$, for every $j\in\Bbb N$. Then, we can write
$$f=\sum_{j\in\Bbb N}(f_{j+1}-f_j),$$
in the sense of convergence in both $H^{p(\cdot),q(\cdot)}({\Bbb R}^n,A)$ and $S'({\Bbb R}^n)$. For every $j\in\Bbb N$, there exist a sequence $\{\lambda_{i,j}\}_{i\in\Bbb N}\subset (0,\infty)$ and a sequence $\{a_{i,j}\}_{i\in\Bbb N}$ of $(p(\cdot),q(\cdot),\infty,s)$-atoms, being for every $i\in\Bbb N$, $a_{i,j}$ associated with $x_{i,j}\in{\Bbb R}^n$ and $\ell_{i,j}\in\Bbb Z$, satisfying that
$$f_{j+1}-f_j=\sum_{i\in\Bbb N} \lambda_{i,j}a_{i,j},\;\;\;\mbox{in}\;S'({\Bbb R}^n),$$
and
$$\left\|\sum_{i\in\Bbb N}{{\lambda_{i,j}}\over{\|\chi_{x_{i,j}+B_{\ell_{i,j}}}\|_{p(\cdot),q(\cdot)}}}
       \chi_{x_{i,j}+B_{\ell_{i,j}}}\right\|_{p(\cdot),q(\cdot)}\leq C2^{-j}\|f\|_{H^{p(\cdot),q(\cdot)}({\Bbb R}^n,A)}.$$
Here $C>0$ does not depend on $f$.

We have that
\begin{align*}
\left\|\sum_{i\in\Bbb N,j\in\Bbb Z}{{\lambda_{i,j}}\over{\|\chi_{x_{i,j}+B_{\ell_{i,j}}}\|_{p(\cdot),q(\cdot)}}}
       \chi_{x_{i,j}+B_{\ell_{i,j}}}\right\|_{p(\cdot),q(\cdot)} & \le \sum_{j\in\Bbb Z}\left\|\sum_{i\in\Bbb N}{{\lambda_{i,j}}\over{\|\chi_{x_{i,j}+B_{\ell_{i,j}}}\|_{p(\cdot),q(\cdot)}}}
       \chi_{x_{i,j}+B_{\ell_{i,j}}}\right\|_{p(\cdot),q(\cdot)}\\
       & \leq C\|f\|_{H^{p(\cdot),q(\cdot)}({\Bbb R}^n,A)}.
\end{align*}

By proceeding as above we can write
$$
f=\sum_{m\in\Bbb N}\lambda_{\pi(m)}a_{\pi(m)},\;\;\;\mbox{in}\;S'({\Bbb R}^n),
$$
for every bijection $\pi:\mathbb{N}\rightarrow\mathbb{N}\times \mathbb{N}$.

Thus the proof of this case is finished.


\subsection{Proof of Theorem \ref{Th1.2} when $r<\infty$.}

In order to prove this property we proceed in a series of steps establishing auxiliary and partial results.

\begin{Prop}\label{prop4.4}
Let $1<r< \infty $ and let $p,q\in \mathbb{P}_0$. There exists $s_0\in \mathbb{N}$ satisfying that if $s\in \mathbb{N}$, $s\ge s_0$, we can find $C>0$ for which, for every $f\in H^{p(\cdot ),q(\cdot)}(\mathbb{R}^n,A)$, there exist, for each $j\in \mathbb{N}$, $\lambda _j>0$ and a $(p(\cdot ),q(\cdot), r,s)$-atom $a_j$ associated with some $x_j\in \mathbb{R}^n$ and $\ell _j \in \mathbb{Z}$, such that
$$
\left\|\sum_{j\in \mathbb{N}}\lambda _j\|\chi_{x_j+B_{\ell _j}}\|_{p(\cdot ),q(\cdot)}^{-1}\chi _{x_j+B_{\ell _j}}\right\|_{p(\cdot ),q(\cdot)}\leq C\|f\|_{H^{p(\cdot ),q(\cdot)}(\mathbb{R}^n,A)},
$$
and $f=\sum_{j\in \mathbb{N}}\lambda _ja_j$ in $S'(\mathbb{R}^n)$.
\end{Prop}

\begin{proof}
Suppose that $a$ is a $(p(\cdot ),q(\cdot),\infty,s)$-atom associated with $x_0\in \mathbb{R}^n$ and $k\in \mathbb{Z}$. We have that
$$
\|a\|_r=\left(\int_{x_0+B_k}|a(x)|^rdx\right)^{1/r}\leq b^{k/r}\|a\|_\infty\leq b^{k/r}\|\chi _{x_0+B_k}\|_{p(\cdot ),q(\cdot)}^{-1}.
$$
Hence, $a$ is a $(p(\cdot ),q(\cdot),r,s)$-atom associated with $x_0\in \mathbb{R}^n$ and $k\in \mathbb{Z}$. Then, this property follows from the previous case $r=\infty$.
\end{proof}

We are going to see that the $(p(\cdot ), q(\cdot), r,s)$-atoms are in $H^{p(\cdot ),q(\cdot)}(\mathbb{R}^n,A)$.
\begin{Prop}\label{prop4.3}
Let $p,q\in \mathbb{P}_0$ such that $p(0)<q(0)$. Assume that $max\{1,q_+\}<r<\infty$. There exists $s_0\in \mathbb{N}$ such that if $a$ is a $(p(\cdot),q(\cdot ),r,s_0)$-atom, then $a\in H^{p(\cdot),q(\cdot )}(\mathbb{R}^n,A)$.
\end{Prop}
\begin{proof}
Let $\varphi \in S(\mathbb{R}^n)$. Assume that $a$ is a $(p(\cdot),q(\cdot ), r,s)$-atom associated with $x_0\in \mathbb{R}^n$ and $\ell _0\in \mathbb{Z}$, where $s\in \mathbb{N}$ will be specified later. We have that
$$
\|M_\varphi (a)\|_{p(\cdot),q(\cdot )}\leq C\left(\|M_\varphi (a)\chi _{x_0+B_{\ell _0+w}}\|_{p(\cdot),q(\cdot )}+\|M_\varphi (a)\chi _{(x_0+B_{\ell _0+w})^{\rm c}}\|_{p(\cdot),q(\cdot )}\right)=I_1+I_2.
$$
It is clear that
$$
(M_\varphi (a)\chi _{x_0+B_{\ell _0+w}})^*(t)=0,\quad t\geq |x_0+B_{\ell _0+w}|=b^{\ell _0+w}.
$$
Then, since $0<p(0)=\lim_{t\rightarrow 0^+}p(t)<q(0)=\lim_{t\rightarrow 0^+}q(t)$, we can write
$$
I_1\leq C\Big\|t^{1/p(t)-1/q(t)}(M_\varphi (a))^*\chi _{(0,b^{\ell _0+w})}\Big\|_{q(\cdot)}\leq C\Big\|(M_\varphi (a))^*\chi _{(0,b^{\ell _0+w})}\Big\|_{q(\cdot)}.
$$
By using \cite[Lemma 2.2]{CrW} and since $r>\max\{1,q_+\}$ we obtain
$$
I_1\leq C\|(M_\varphi (a))^*\|_{L^r(0,\infty )}=C\|M_\varphi (a)\|_{L^r(\mathbb{R}^n)}\leq C\|a\|_{L^r(\mathbb{R}^n)}\leq Cb^{\ell _0/r}\|\chi _{x_0+B_{\ell _0}}\|_{p(\cdot ),q(\cdot)}^{-1}<\infty.
$$

By proceeding as in the proof of the case $r=\infty$ (see \cite[p. 19-21]{Bow1}) we get
$$
M_\varphi (a)(x)\leq \frac{C}{\|\chi _{x_0+B_{\ell _0}}\|_{p(\cdot ),q(\cdot)}}(M_{HL}(\chi _{x_0+B_{\ell _0}})(x))^\gamma ,\quad x\not \in x_0+B_{\ell _0+w},
$$
provided that $s\ge \frac{\gamma-1}{\log_b(\lambda_-)}-1$, where $\gamma >1$ is such that $\gamma p,\gamma q\in \mathbb{P}_1$. Then, Proposition \ref{propM_{HL}} implies that
$$
I_2\leq C\frac{\|(M_{HL}(\chi _{x_0+B_{\ell _0}}))^\gamma \|_{p(\cdot ),q(\cdot)}}{\|\chi _{x_0+B_{\ell _0}}\|_{p(\cdot ),q(\cdot)}}
=C\frac{\|M_{HL}(\chi _{x_0+B_{\ell _0}})\|^\gamma _{\gamma p(\cdot ),\gamma q(\cdot)}}{\|\chi _{x_0+B_{\ell _0}}\|_{p(\cdot ),q(\cdot)}}\leq C.
$$
Thus, we have shown that $a\in H^{p(\cdot ),q(\cdot)}(\mathbb{R}^n,A)$.
\end{proof}

Note that the constant $C$ in the proof of the last proposition depends on the atom $a$. This fact indicates that this next result cannot be a consequence of Proposition \ref{prop4.3}. We need a more involved argument to show the following property.

\begin{Prop}\label{prop4.6}
Let $p,q\in \mathbb{P}_0$ being $p(0)<q(0)$. There exist $s_0\in \mathbb{N}$ and $r_0>1$ such that, for every $r\geq r_0$ we can find $C>0$ satisfying that if, for every $j\in \mathbb{N}$, $\lambda _j>0$ and $a_j$ is a $(p(\cdot ),q(\cdot), r,s_0)$-atom associated with $x_j\in \mathbb{R}^n$ and $\ell _j \in \mathbb{Z}$ such that
$$
\sum_{j\in \mathbb{N}}\lambda _j\|\chi_{x_j+B_{\ell _j}}\|_{p(\cdot ),q(\cdot)}^{-1}\chi _{x_j+B_{\ell _j}}\in \mathcal{L}^{p(\cdot ),q(\cdot)}(\mathbb{R}^n),
$$
then $f=\sum_{j\in \mathbb{N}}\lambda _ja_j\in H^{p(\cdot ),q(\cdot)}(\mathbb{R}^n,A)$, and
$$
\|f\|_{H^{p(\cdot ),q(\cdot)}(\mathbb{R}^n,A)}\leq C\left\|\sum_{j\in \mathbb{N}}\lambda _j\|\chi_{x_j+B_{\ell _j}}\|_{p(\cdot ),q(\cdot)}^{-1}\chi _{x_j+B_{\ell _j}}\right\|_{p(\cdot ),q(\cdot)}.
$$
\end{Prop}

In order to prove this proposition we need to establish some previous properties.

\begin{Lem}\label{lem4.6}
Assume that $(\lambda _k)_{k\in\Bbb N} $ is a sequence in $(0,\infty)$, $(\ell_k)_{k\in\Bbb N}$ is a sequence in $\mathbb{Z}$, $(x_k)_{k\in\Bbb N} $ is a sequence in $\mathbb{R}^n$, $\nu$ is a doubling weight, that is, $\nu dx$ is a doubling measure,  (with respect to the anisotropic balls), $\ell\in \mathbb{N}$,  $\ell\geq 1$, and $0<p<\infty$. Then,
\begin{equation}\label{z0}
\left\|\sum_{k\in\Bbb N} \lambda _k\chi _{x_k+B_{\ell_k+\ell}}\right\|_{L^p(\mathbb{R}^n,\nu)}\leq Cb^{\ell\delta} \left\|\sum_{k\in\Bbb N} \lambda _k\chi _{x_k+B_{\ell_k}}\right\|_{L^p(\mathbb{R}^n,\nu)}.
\end{equation}
Here $C,\delta >0$ depends only on $\nu$.
\end{Lem}
\begin{proof}
Suppose firstly that $p>1$. We follow the ideas in the proof of \cite[Theorem 2, p. 53]{STor}.
We take $0\leq g\in L^{p'}(\mathbb{R}^n,\nu)$, where $p'$ is the exponent conjugated to $p$, that is, $p'=p/(p-1)$. Let $y\in \mathbb{R}^n$ and $k\in \mathbb{Z}$. We define the maximal operator $M_\nu$ by
$$
M_\nu(h)(z)=\sup_{m\in \mathbb{Z},\,y\in z+B_m}\frac{1}{\nu(y+B_m)}\int_{y+B_m}|h(x)|\nu (x)dx,\,\,\,z\in \mathbb{R}^n.
$$
Since $\nu$ is doubling with respect to the anisotropic balls, for a certain $\delta >0$, we have that
\begin{align*}
\int_{y+B_{k+\ell}}g(x)\nu (x)dx& \leq b^{\ell\delta} \frac{\nu(y+B_k)}{\nu(y+B_{k+\ell})}\int_{y+B_{k+\ell}}g(x)\nu (x)dx\\
& \leq b^{\ell\delta} \int_{y+B_k}M_\nu(g)(x)\nu (x)dx,\;\;\;y\in{\Bbb R}^n,\;k\in\Bbb Z.
\end{align*}
We have taken into account that
$$
M_\nu(g)(z)\geq \frac{1}{\nu (y+B_{\ell+k})}\int_{y+B_{\ell+k}}g(x)\nu(x)dx,\quad z\in y+B_k.
$$
Let $m\in \mathbb{N}$. We can write
\begin{align*}
\int_{\mathbb{R}^n}\sum_{k=0}^m\lambda _k\chi _{x_k+B_{\ell+\ell_k}}(x)g(x)\nu (x)dx&=\sum_{k=0}^m\lambda _k\int_{x_k+B_{\ell+\ell_k}}g(x)\nu (x)dx\\
&\leq b^{\ell\delta} \sum_{k=0}^m\lambda_k\int_{x_k+B_{\ell_k}}M_\nu (g)(x)\nu (x)dx.
\end{align*}
Hence, the maximal theorem \cite[Theorem 3, p. 3]{STor} leads to
\begin{align*}
\left|\int_{\mathbb{R}^n}\sum_{k=0}^m\lambda _k\chi_{x_k+B_{\ell+\ell_k}}(x)g(x)\nu (x)dx\right|&\leq b^{\ell\delta}\left\|\sum_{k=0}^m\lambda _k\chi _{x_k+B_{\ell_k}}\right\|_{L^p(\mathbb{R}^n,\nu)}\|M_\nu (g)\|_{L^{p'}(\mathbb{R}^n,\nu)}\\
&\leq Cb^{\ell\delta} \left\|\sum_{k=0}^m\lambda _k\chi _{x_k+B_{\ell_k}}\right\|_{L^p(\mathbb{R}^n,\nu)}\|g\|_{L^{p'}(\mathbb{R}^n,\nu)}.
\end{align*}

We conclude that
$$
\left\|\sum_{k=0}^m\lambda _k\chi _{x_k+B_{\ell+\ell_k}}(x)\right\|_{L^p(\mathbb{R}^n,\nu)}\leq Cb^{\ell\delta} \left\|\sum_{k=0}^m\lambda _k\chi _{x_k+B_{\ell_k}}\right\|_{L^p(\mathbb{R}^n,\nu)}.
$$

By taking $m\rightarrow \infty$, the  monotone convergence theorem allows us to establish \eqref{z0} in this case.

Assume now that $0<p\leq 1$. For every $x_0\in \mathbb{R}^n$ and $k_0\in \mathbb{Z}$, we denote by $\delta _{(x_0,k_0)}$ the Dirac measure in $\mathbb{R}^{n+1}$ supported in $(x_0,k_0)$. Let $m\in \mathbb{N}$. We have that
\begin{align*}
\int_{x\in y+B_{\ell+j}}\sum_{k=0}^m\lambda _k\delta_{(x_k,\ell_k)}(y,j)&=\sum_{k=0}^m\lambda _k\int_{\mathbb{R}^{n+1}}\chi _{\{(y,j):x\in y+B_{\ell+j}\}}(y,j)\delta _{(x_k,\ell_k)}(y,j)\\
&=\sum_{k=0}^m\lambda_k\chi _{\{(y,j):x\in y+B_{\ell+j}\}}(x_k,\ell_k)=\sum_{k=0}^m\lambda _k\chi _{x_k+B_{\ell+\ell_k}}(x),\quad x\in \mathbb{R}^n.
\end{align*}
Also, we can write
$$
\int_{x\in y+B_{j}}\sum_{k=0}^m\lambda _k\delta_{(x_k,\ell_k)}(y,j)=\sum_{k=0}^m\lambda _k\chi _{x_k+B_{\ell_k}}(x),\quad x\in \mathbb{R}^n.
$$

By arguing as in the proof of \cite[Theorem 1, p. 52]{STor} replacing the area Littlewood-Paley functions by our area integrals we can prove that
$$
\left\|\int_{x\in y+B_{\ell+j}}\sum_{k=0}^m\lambda _k\delta_{(x_k,\ell_k)}(y,j)\right\|_{L^p(\mathbb{R}^n,\nu)}\leq Cb^{\ell\delta} \left\|\int_{x\in y+B_{j}}\sum_{k=0}^m\lambda _k\delta_{(x_k,\ell_k)}(y,j)\right\|_{L^p(\mathbb{R}^n,\nu)}.
$$
By letting $m\rightarrow \infty$ and using monotone convergence theorem we conclude \eqref{z0}.

\end{proof}

We now recall definitions of anisotropic $\mathcal{A}_r$-weights and anisotropic weighted Hardy spaces (see \cite{BLYZ} and \cite{STor}).

Let $r\in (1,\infty)$ and $\nu$ be a nonnegative measurable function on $\mathbb{R}^n$. The function $\nu$ is said to be a weight in the anisotropic Muckenhoupt class $\mathcal{A}_r(\mathbb{R}^n,A)$ when
$$
[\nu]_{\mathcal{A}_r(\mathbb{R}^n,A)}=:\sup_{x\in \mathbb{R}^n,\,k\in \mathbb{Z}}\left(\frac{1}{|B_k|}\int_{x+B_k}\nu (y)dy\right)\left(\frac{1}{|B_k|}\int_{x+B_k}(\nu(y))^{-1/(r-1)}dy\right)^{r-1}<\infty.
$$
We say that $\nu$ belongs to the anisotropic Muckenhoupt class $\mathcal{A}_1(\mathbb{R}^n,A)$ when
$$
[\nu]_{\mathcal{A}_1(\mathbb{R}^n,A)}=:\sup_{x\in \mathbb{R}^n,\,k\in \mathbb{Z}}\left(\frac{1}{|B_k|}\int_{x+B_k}\nu (y)dy\right)\sup_{y\in x+B_k}(\nu (y))^{-1}<\infty.
$$
We define $\mathcal{A}_\infty(\mathbb{R}^n,A)=\bigcup_{1\le r<\infty}\mathcal{A}_r(\mathbb{R}^n,A)$.

The weight $\nu$ satisfies the reverse H\"older condition $RH_r(\mathbb{R}^n,A)$ (in short, $\nu\in RH_r(\mathbb{R}^n,A)$) if there exists $C>0$ such that
$$
\left(\frac{1}{|B_k|}\int_{x+B_k}(\nu (y))^rdy\right)^{1/r}\le C\frac{1}{|B_k|}\int_{x+B_k}\nu (y)dy,\,\,\,x\in \mathbb{R}^n\,\,\,\rm{and}\,\,\,k\in \mathbb{Z}.
$$
The classes $\mathcal{A}_r(\mathbb{R}^n,A)$ and
$RH_\alpha(\mathbb{R}^n,A)$ are closely connected. In particular, if
$\nu\in \mathcal{A}_1(\mathbb{R}^n,A)$, there exists $\alpha\in
(1,\infty)$ such that $\nu\in RH_\alpha(\mathbb{R}^n,A)$
(\cite[Theorem 1.3]{Kin1}).

Let $1\leq r<\infty$ and $\nu\in \mathcal{A}_r(\mathbb{R}^n,A)$. For every $N\in \mathbb{N}$, the anisotropic Hardy space $H^r_N(\mathbb{R}^n,\nu,A)$ consists of all those $f\in S'(\mathbb{R}^n)$ such that $M_N(f)\in L^r(\mathbb{R}^n,\nu)$. There exists $N_{r,\nu}\in \mathbb{N}$ satisfying that $H^r_N(\mathbb{R}^n,\nu,A)=H^r_{N_{r,\nu}}(\mathbb{R}^n,\nu,A)$, for every $N\ge N_{r,\nu}$. Moreover, when $N\ge N_{r,\nu}$ the quantities $\|M_N(f)\|_{L^r(\mathbb{R}^n,\nu)}$ and $\|M_{N_{r,\nu}}(f)\|_{L^r(\mathbb{R}^n,\nu)}$ are equivalent, for every $f\in H^r_{N_{r,\nu}}(\mathbb{R}^n,\nu,A)$. We denote $H^r(\mathbb{R}^n,\nu,A)$ to the space $H^r_{N_{r,\nu}}(\mathbb{R}^n,\nu,A)$.

By proceeding as in the proof of \cite[Lemma 5, p. 116]{STor} we can obtain the following property.
\begin{Lem}\label{lem4.7}
Let $p\in (0,\infty )$ and $q>\max\{1,p\}$. Assume that $\nu\in RH_{(q/p)'}(\mathbb{R}^n,A)$. Then, there exists $C>0$ such that if, for every $k\in \mathbb{N}$, the measurable function $a_k$ has its support contained in the ball $x_k+B_{\ell_k}$, where $x_k\in \mathbb{R}^n$, $\ell_k\in \mathbb{Z}$, $\|a_k\|_q\leq \|\chi _{x_k+B_{\ell_k}}\|_q$, and $\lambda _k>0$, we have that
$$
\left\|\sum_{k\in\Bbb N} \lambda _ka_k\right\|_{L^p(\mathbb{R}^n,\nu)}\leq C\left\|\sum_{k\in\Bbb N} \lambda _k\chi _{x_k+B_{\ell_k}}\right\|_{L^p(\mathbb{R}^n,\nu)}.
$$
\end{Lem}

If $1<r\leq \infty$ and $N\in \mathbb{N}$ we say that a function $a\in L^r(\mathbb{R}^n)$ is a $(r,N)$-atom associated with $x_0\in \mathbb{R}^n$ and $j_0\in \mathbb{Z}$, when $a$ satisfies the following properties:
\begin{itemize}
\item[$(i)$] ${\rm supp }\; a\subset x_0+B_{j_0}$,
\item[$(ii)$] $\|a\|_r\leq b^{j_0/r}$,
\item[$(iii)$] $\displaystyle \int_{\mathbb{R}^n}a(x)x^\alpha dx=0$, for all $|\alpha|\leq N$, $\alpha \in \mathbb{N}^n$.
\end{itemize}

Next result is an anisotropic version of the second part of \cite[Theorem 1, p. 112]{STor}.

\begin{Lem}\label{lem4.8}
Let $0<p<\infty$. Assume that $\nu\in RH_{(q/p)'}(\mathbb{R}^n,A)$ where $q>\max \{1,p\}$. There exists $N_1\in \mathbb{N}$ and $C>0$ such that if, for every $k\in \mathbb{N}$, $a_k$ is a $(q,N_1)$-atom associated with $x_k\in \mathbb{R}^n$ and $\ell _k\in \mathbb{Z}$, and $\lambda _k>0$, satisfying that
$$
\Big\|\sum_{k=1}^\infty \lambda _k\chi _{x_k+B_{\ell _k}}\Big\|_{L^p(\mathbb{R}^n,\nu)}<\infty,
$$
then the series $\sum_{k=1}^\infty\lambda _ka_k$ converges in both $S'(\mathbb{R}^n)$ and $H^p(\mathbb{R}^n,\nu,A)$ to an element $f\in H^p(\mathbb{R}^n,\nu,A)$ such that
$$
\|f\|_{H^p(\mathbb{R}^n,\nu,A)}\leq C\Big\|\sum_{k=1}^\infty \lambda _k\chi _{x_k+B_{\ell _k}}\Big\|_{L^p(\mathbb{R}^n,\nu)}.
$$
\end{Lem}

\begin{proof}
Suppose that $a$ is a $(q,N)$-atom associated with $x_0\in \mathbb{R}^n$ and $\ell _0\in \mathbb{Z}$. Here $N\in \mathbb{N}$ will be specified later.

We choose $\varphi \in S(\mathbb{R}^n)$. We now estimate $\|M_\varphi (a)\|_{L^q(\mathbb{R}^n)}$ by considering in a separate way the regions $x_0+B_{\ell _0+w}$ and $(x_0+B_{\ell _0+w})^{\rm c}$.

Since $q>1$ the maximal theorem (\cite[Theorem 3.6]{Bow1}) implies that
$$
\left(\int_{\mathbb{R}^n}\chi _{x_0+B_{\ell _0+w}}(x)|M_\varphi (a)(x)|^qdx\right)^{1/q}\leq \|M_\varphi (a)\|_{L^q(\mathbb{R}^n)}\leq Cb^{\ell _0/q}\leq Cb^{(\ell _0+w)/q}.
$$

Hence, the function $\beta _0=\frac{1}{C}\chi _{x_0+B_{\ell _0+w}}M_\varphi (a)$ is a $(q,-1)$-atom associated with $x_0$ and $\ell _0+w$. The index $-1$ means that no null moment condition need to be satisfied.

By proceeding as in \cite[p. 20]{Bow1} we get, for every $m\in \mathbb{N}$,
$$
M_\varphi (a)(x)\leq C(b\lambda _-^{N+1})^{-m},\quad x\in x_0+(B_{\ell _0+w+m+1}\setminus B_{\ell _0+w+m}).
$$

We define $\rho_m=\chi _{x_0+B_{\ell _0+w+m+1}}$, $m\in \mathbb{N}$. It is clear that $\rho_m$ is a $(q,-1)$-atom associated with $x_0$ and $\ell _0+w+m+1$, for every $m\in \mathbb{N}$, and that
$$
\chi _{(x_0+B_{\ell _0+w})^{\rm c}} M_\varphi (a)\leq C\sum_{m\in\Bbb N} (b\lambda _-^{N+1})^{-m}\rho_m.
$$
Hence, we obtain
\begin{equation}\label{z1}
M_\varphi (a)\leq C\left(\beta _0+\sum_{m\in\Bbb N} (b\lambda _-^{N+1})^{-m}\rho_m\right).
\end{equation}
Here $C>0$ does not depend on $a$.

Suppose that  $k\in \mathbb{N}$ and, for every $j\in \mathbb{N}$, $j\leq k$, $\lambda _j>0$ and $a_j$ is a $(q,N)$-atom associated with $x_j\in \mathbb{R}^n$ and $\ell _j\in \mathbb{Z}$. According to \eqref{z1} we get
$$
M_\varphi \left(\sum_{j=0}^k\lambda _ja_j\right)\leq C\left(\sum_{j=0}^k\lambda _j(\beta _{0,j}+\sum_{m=0}^\infty (b\lambda _-^{N+1})^{-m}\rho_{m,j})\right),
$$
where $\beta _{0,j}$ and $\rho_{m,j}$, $j=1,...,k$, and $m\in \mathbb{N}$ have the obvious meaning and are $(q,-1)$-atoms. By using Lemmas \ref{lem4.6} and \ref{lem4.7}, and by taking $p_1=\min\{1,p\}$ we have that
\begin{align*}
\Big\|\sum_{j=0}^k\lambda _ja_j\big\|_{H^p(\mathbb{R}^n,\nu,A)}^{p_1}&\leq C\Big\|\sum_{j=0}^k\lambda _j(\beta _{0,j}+\sum_{m\in\Bbb N} (b\lambda _-^{N+1})^{-m}\rho_{m,j})\Big\|_{L^p(\mathbb{R}^n,\nu)}^{p_1}\\
&\leq C\left(\sum_{m\in\Bbb N} (b\lambda _-^{N+1})^{-mp_1}\Big\|\sum_{j=0}^k\lambda _j\rho_{m,j}\Big\|_{L^p(\mathbb{R}^n,\nu)}^{p_1}+\Big\|\sum_{j=0}^k\lambda _j\beta _{0,j}\Big\|_{L^p(\mathbb{R}^n,\nu)}^{p_1}\right)\\
&\leq C\left(\sum_{m\in\Bbb N} (b\lambda _-^{N+1})^{-mp_1}\Big\|\sum_{j=0}^k\lambda _j\chi _{x_j+B_{\ell _j+w+m+1}}\Big\|_{L^p(\mathbb{R}^n,\nu)}^{p_1}+\Big\|\sum_{j=0}^k\lambda _j\chi _{x_j+B_{\ell _j}}\Big\|_{L^p(\mathbb{R}^n,\nu)}^{p_1}\right)\\
&\leq C\left(\sum_{m\in\Bbb N}(b\lambda _-^{N+1})^{-mp_1}b^{\delta mp_1}+1\right)\Big\|\sum_{j=0}^k\lambda _j\chi _{x_j+B_{\ell _j+w}}\Big\|_{L^p(\mathbb{R}^n,\nu)}^{p_1},
\end{align*}
for a certain $\delta >0$. Hence, if $(\delta -1)\ln b/\ln (\lambda _-)<N+1$, we conclude that
$$
\Big\|\sum_{j=0}^k\lambda _ja_j\Big\|_{H^p(\mathbb{R}^n,\nu,A)}\leq C\Big\|\sum_{j=0}^k\lambda _j\chi _{x_j+B_{\ell _j}}\Big\|_{L^p(\nu)}.
$$
Standard arguments allow us to finish the proof of this property.
\end{proof}

>From Lemma \ref{lem4.8} we can deduce the following.

\begin{Lem}\label{lem4.9}
Assume that $p,q\in \mathbb{P}_0$, $p_0\in (0,\infty )$, $q_0>\max\{1,p_0\}$ and $\nu\in \mathcal{A}_1(\mathbb{R}^n,A)\bigcap RH_{(q_0/p_0)'}(\mathbb{R}^n,A)$. Suppose that, for every $k\in \mathbb{N}$, $\lambda _k>0$ and $a_k$ is a $(p(\cdot), q(\cdot ),q_0,N_1)$-atom associated with $x_k\in \mathbb{R}^n$ and $\ell _k\in \mathbb{Z}$, satisfying that
$$
\Big\|\sum_{k\in\Bbb N} \lambda _k\|\chi _{x_k+B_{\ell _k}}\|_{p(\cdot),q(\cdot )}^{-1}\chi _{x_k+B_{\ell _k}}\Big\|_{L^{p_0}(\mathbb{R}^n,\nu)}<\infty.
$$
Here $N_1$ is the one defined in Lemma \ref{lem4.8}.

Then, the series $f=\sum_{k\in\Bbb N} \lambda _ka_k$ converges in $H^{p_0}(\mathbb{R}^n,\nu,A)$ and
$$
\|f\|_{H^{p_0}(\mathbb{R}^n,\nu,A)}\leq C\Big\|\sum_{k\in\Bbb N} \lambda _k\|\chi _{x_k+B_{\ell _k}}\|_{p(\cdot),q(\cdot )}^{-1}\chi _{x_k+B_{\ell _k}}\Big\|_{L^{p_0}(\mathbb{R}^n,\nu)}.
$$
Here $C$ does not depend on $\{\lambda _k\}_{k\in\Bbb N} $ and $\{a_k\}_{k\in\Bbb N} $.
\end{Lem}
\begin{proof}
It is sufficient to note that, for every $k\in \mathbb{N}$, $a_k\|\chi _{x_k+B_{\ell _k}}\|_{p(\cdot),q(\cdot)}$ is a $(q_0,N_1)$-atom and $\nu$ is doubling with respect to anisotropic balls.
\end{proof}


\begin{proof}[Proof of Proposition \ref{prop4.6}]
We choose $\alpha>1$ such that $\alpha p,\alpha q\in \mathbb{P}_1$. Then, $(\alpha p)',(\alpha q)'\in \mathbb{P}_1$. We recall that the dual space $(\mathcal{L}^{\alpha p(\cdot ),\alpha q(\cdot )}(\mathbb{R}^n))^*$ of $\mathcal{L}^{\alpha p(\cdot ),\alpha q(\cdot )}(\mathbb{R}^n)$ is $\mathcal{L}^{(\alpha p(\cdot ))', (\alpha q(\cdot ))'}(\mathbb{R}^n)$ and the maximal operator $M_{HL}$ is bounded from $\mathcal{L}^{(\alpha p(\cdot ))', (\alpha q(\cdot ))'}(\mathbb{R}^n)$ into itself (Proposition \ref{propM_{HL}}).

In the sequel our argument is as in \cite{CrW} supported in the Rubio de Francia iteration algorithm. Given a function $h$ we define $M_{HL}^0(h)=|h|$ and, for every $i\in\mathbb{N}$, $i\ge 1$, $M_{HL}^i(h)=M_{HL}\circ M_{HL}^{i-1}(h)$. We consider
$$
R(h)=\sum_{i=0}^\infty \frac{M_{HL}^i(h)}{2^i\|M_{HL}\|_{(\alpha p(\cdot))',(\alpha q(\cdot))'}^i}.
$$

We have that

$(i)$ $|h|\leq R(h)$;

$(ii)$ $R$ is bounded from $\mathcal{L}^{(\alpha p(\cdot))',(\alpha q(\cdot))'}(\mathbb{R}^n)$ into itself and $\|R(h)\|_{(\alpha p(\cdot))',(\alpha q(\cdot))'}\leq 2\|h\|_{(\alpha p(\cdot))',(\alpha q(\cdot))'}$;

$(iii)$ $R(h)\in \mathcal{A}_1(\mathbb{R}^n,A)$ and $[R(h)]_{\mathcal{A}_1(\mathbb{R}^n,A)}\leq 2\|M_{HL}\|_{(\alpha p(\cdot))',(\alpha q(\cdot))'}$. Hence, there exists $\beta _0>1$ such that $R(h)\in RH_{\beta _0}(\mathbb{R}^n,A)$.

We choose $r>\max\{1,q_+\}$ such that $R(h)\in RH_{(r\alpha)'}(\mathbb{R}^n,A)$. It is sufficient to take $r>\max\{1,q_+,\frac{\beta_0}{\alpha(\beta_0-1)}\}$.

Suppose that $k\in \mathbb{N}$ and, for every $j \in \mathbb{N}$, $j\leq k$, $\lambda _j>0$ and $a_j$ is a $(p(\cdot ),q(\cdot ), r,N_1)$-atom associated with $x_j\in \mathbb{R}^n$ and $\ell _j\in \mathbb{Z}$. Here $N_1$ is the one defined in Lemma \ref{lem4.8}. We define $f_k=\sum_{j=0}^k\lambda _ja_j$. According to Proposition \ref{prop4.3}, $f_k\in H^{p(\cdot ),q(\cdot )}(\mathbb{R}^n,A)$.

Since $R(h)\in \mathcal{A}_1(\mathbb{R}^n,A)\bigcap RH_{(r\alpha)'}(\mathbb{R}^n,A)$, by Lemma \ref{lem4.9}, $f_k\in H^{1/\alpha}(\mathbb{R}^n,R(h),A)$ and
\begin{equation}\label{z3}
\|f_k\|_{H^{1/\alpha}(\mathbb{R}^n,R(h),A)}\leq C\Big\|\sum_{j=0}^k\lambda _j\|\chi _{x_j+B_{\ell _j}}\|_{p(\cdot ),q(\cdot )}^{-1}\chi _{x_j+B_{\ell _j}}\Big\|_{L^{1/\alpha}(\mathbb{R}^n,R(h))}.
\end{equation}

Let $\varphi \in S(\mathbb{R}^n)$. By using \cite[Lemma 2.3]{CrW} and \cite[Lemma 2.7]{EKS} we can write
$$
\|M_\varphi (f_k)\|_{p(\cdot),q(\cdot)}^{1/\alpha}=\|(M_\varphi (f_k))^{1/\alpha}\|_{\alpha p(\cdot ),\alpha q(\cdot )}\leq C\sup_{h}\int_{\mathbb{R}^n}(M_\varphi (f_k)(x))^{1/\alpha}h(x)dx,
$$
where the supremum is taken over all the functions $0\le h \in \mathcal{L}^{(\alpha p(\cdot ))', (\alpha q(\cdot ))'}(\mathbb{R}^n)$ such that $\|h\|_{(\alpha p(\cdot ))', (\alpha q(\cdot ))'}\leq 1$.

By the above properties $(i)$, $(ii)$ and $(iii)$ and \eqref{z3}, for every $0\le h\in \mathcal{L}^{(\alpha p(\cdot ))', (\alpha q(\cdot ))'}(\mathbb{R}^n)$ such that $\|h\|_{(\alpha p(\cdot ))', (\alpha q(\cdot ))'}\leq 1$, we get
\begin{align*}
\int_{\mathbb{R}^n}(M_\varphi (f_k)(x))^{1/\alpha}h(x)dx&\leq \int_{\mathbb{R}^n}(M_\varphi (f_k)(x))^{1/\alpha}R(h)(x)dx\\
&\leq C\int_{\mathbb{R}^n}\left(\sum_{j=0}^k\lambda _j\|\chi _{x_j+B_{\ell _j}}\|_{p(\cdot ),q(\cdot )}^{-1}\chi _{x_j+B_{\ell _j}}(x)\right)^{1/\alpha}R(h)(x)dx\\
&\leq C\left\|\left(\sum_{j=0}^k\lambda _j\|\chi _{x_j+B_{\ell _j}}\|_{p(\cdot ),q(\cdot )}^{-1}\chi _{x_j+B_{\ell _j}}\right)^{1/\alpha}\right\|_{\alpha p(\cdot ),\alpha q(\cdot )}\|R(h)\|_{(\alpha p(\cdot ))',(\alpha q(\cdot ))'}\\
&\leq C\left\|\sum_{j=0}^k\lambda _j\|\chi _{x_j+B_{\ell _j}}\|_{p(\cdot ),q(\cdot )}^{-1}\chi _{x_j+B_{\ell _j}}\right\|_{p(\cdot ),q(\cdot )}^{1/\alpha}\|h\|_{(\alpha p(\cdot ))',(\alpha q(\cdot ))'}\\
&\leq C\left\|\sum_{j=0}^k\lambda _j\|\chi _{x_j+B_{\ell _j}}\|_{p(\cdot ),q(\cdot )}^{-1}\chi _{x_j+B_{\ell _j}}\right\|_{p(\cdot ),q(\cdot )}^{1/\alpha}.
\end{align*}
Hence, we obtain
$$
\|f_k\|_{H^{p(\cdot ),q(\cdot )}(\mathbb{R}^n,A)}\leq C\left\|\sum_{j=0}^k\lambda _j\|\chi _{x_j+B_{\ell _j}}\|_{p(\cdot ),q(\cdot )}^{-1}\chi _{x_j+B_{\ell _j}}\right\|_{p(\cdot ),q(\cdot )}.
$$
We finish the proof by using standard arguments.
\end{proof}

\section{Finite atomic decomposition (proof of Theorem \ref{Th1.3})}

The proof of this result follows the ideas developed in  \cite{BLYZ}  and \cite{MSV}. Here we only show those points where variable exponent Lorentz space norm appears.

\begin{enumerate}
\item[(i)] Assume that $r_0<r<\infty$ and $s\in\Bbb N$, $s\geq s_0$, being $r_0$ and $s_0$ the parameters appearing  in  Theorem \ref{Th1.2} (i). By using this result we get that $H^{p(\cdot),q(\cdot),r,s}_{fin}({\Bbb R}^n,A)\subset H^{p(\cdot),q(\cdot)}({\Bbb R}^n,A)$ and, for every $f\in H^{p(\cdot),q(\cdot),r,s}_{fin}({\Bbb R}^n,A)$,
    $$\|f\|_{H^{p(\cdot),q(\cdot)}({\Bbb R}^n,A)}\leq C \|f\|_{H^{p(\cdot),q(\cdot),r,s}_{fin}({\Bbb R}^n,A)}.$$
    We now prove that, there exists $C>0$ such that $\|f\|_{H^{p(\cdot),q(\cdot),r,s}_{fin}({\Bbb R}^n,A)}\leq C$, provided that $f\in H^{p(\cdot),q(\cdot),r,s}_{fin}({\Bbb R}^n,A)$ and $\|f\|_{H^{p(\cdot),q(\cdot)}({\Bbb R}^n,A)}=1$.

Let $f\in H^{p(\cdot),q(\cdot),r,s}_{fin}({\Bbb R}^n,A)$ such that $\|f\|_{H^{p(\cdot),q(\cdot)}({\Bbb R}^n,A)}=1$. We have that $f\in L^r({\Bbb R}^n)$ and $\mbox{supp}\; f\subset B_{m_0}$ for some $m_0\in\Bbb Z$. For every $j\in\Bbb Z$ we define the set $\Omega_j=\{x\in{\Bbb R}^n:\;M_N(f)(x)>2^j\}$, where $N\in\Bbb N$, $N>\max\{N_0,s\}$ (here $N_0$ is as in Theorem \ref{Th1.1}). According to the proof of  Theorem \ref{Th1.2} and \cite[p. 3088]{BLYZ},  for every $i\in\Bbb N$ and $j\in\Bbb Z$ there exist $\lambda_{i,j}>0$ and a $(p(\cdot),q(\cdot),\infty,s)$-atom $a_{i,j}$ satisfying the following properties:
\begin{enumerate}
\item[(a)] $f=\sum_{i,j}\lambda_{i,j}a_{i,j}$, where the series converges unconditionally in $S'({\Bbb R}^n)$.
\item[(b)] $|\lambda_{i,j}a_{i,j}|\leq C 2^j$, $i\in\Bbb N$ and $j\in\Bbb Z$; 

for certain sequences $\{x_{i,j}\}_{i\in\Bbb N, j\in\Bbb Z}\subset{\Bbb R}^n$ and $\{\ell_{i,j}\}_{i\in\Bbb N, j\in\Bbb Z}\subset\Bbb Z$,
\item[(c)] $\mbox{supp}\;(a_{i,j})\subset x_{i,j}+B_{\ell_{i,j}+4\omega}$;
\item[(d)] $\Omega_j=\cup_{i\in\Bbb N}(x_{i,j}+B_{\ell_{i,j}+4\omega})$;
\item[(e)] $(x_{i,j}+B_{\ell_{i,j}-2\omega})\cap(x_{k,j}+B_{\ell_{k,j}-2\omega})=\emptyset$, $j\in \mathbb{Z}$, $i,k\in \mathbb{N}$, $i\not=k$;
\item[(f)] $\left\|\sum_{i\in\Bbb N,j\in\Bbb Z}\lambda_{i,j}\|\chi_{x_{i,j}+B_{\ell_{i,j}+4\omega}}\|^{-1}_{p(\cdot),q(\cdot)}\chi_{x_{i,j}+B_{\ell_{i,j}+4\omega}}\right\|_{p(\cdot),q(\cdot)}\leq C\|f\|_{H^{p(\cdot ),q(\cdot)}(\mathbb{R}^n,A)}=C.$
\end{enumerate}
The constants $C$ in (b) and (f) do not depend on $f$.

By using (b), (c), (d) and (e), we obtain
\begin{align*}
         \int_{{\Bbb R}^n}\sum_{j\in\Bbb Z}\sum_{i\in\Bbb N}|\lambda_{i,j}a_{i,j}(x)|dx &\leq C \sum_{j\in\Bbb Z}\sum_{i\in\Bbb N}2^j|x_{i,j}+B_{\ell_{i,j}+4\omega}| \\
       & \leq C\sum_{j\in\Bbb Z}2^j\sum_{i\in\Bbb N}|x_{i,j}+B_{\ell_{i,j}-2\omega}|=C\sum_{j\in\Bbb Z}2^j\Big|\bigcup_{i\in\Bbb N}(x_{i,j}+B_{\ell_{i,j}+4\omega})\Big| \\
       & \leq C\sum_{j\in\Bbb Z}2^j|\Omega_j|=C \int_{{\Bbb R}^n}\sum_{j\in\Bbb Z}2^j\chi_{\Omega_j}(x)dx\leq C\int_{{\Bbb R}^n}M_N(f)(x)dx.
        \end{align*}
Note that $M_N(f)\in L^1({\Bbb R}^n)$ because $f$ is a multiple of a $(1,r,s)$-atom. Let $\pi=(\pi_1,\pi_2):\mathbb{N}\rightarrow \mathbb{N}\times \mathbb{Z}$ be a bijection. We have that $\int_{{\Bbb R}^n}\sum_{m\in\Bbb N}|\lambda_{\pi(m)}a_{\pi(m)}(x)|dx<\infty$. Then, there exist a monotone function $\mu:\Bbb N\rightarrow\Bbb N$ and a subset $E\subset{\Bbb R}^n$ such that $\sum_{m\in\Bbb N}|\lambda_{\pi(\mu(m))}a_{\pi(\mu(m))}(x)|<\infty$, for every $x\in E$ and $|{\Bbb R}^n\backslash E|=0$. Hence, $\sum_{m\in\Bbb N}|\lambda_{\pi(m)}a_{\pi(m)}(x)|<\infty$, for every $x\in E$. Since the last series has positive terms, we conclude that the series $\sum_{m\in\Bbb N}\lambda_{\pi(m)}a_{\pi(m)}(x)$ is unconditionally convergent, for every $x\in E$, and $\sum_{m\in\Bbb N}\lambda_{\pi(m)}a_{\pi(m)}(x)=\sum_{j\in\Bbb Z}(\sum_{i\in\Bbb N}\lambda_{i,j}a_{i,j}(x))$, $x\in E$. Moreover, the arguments in the proof of Theorem \ref{Th1.2}, (ii) (see also \cite[pp. 3088 and 3089]{BLYZ}) lead us to
$$f(x)=\sum_{j\in\Bbb Z}\Big(\sum_{i\in\Bbb N}\lambda_{i,j}a_{i,j}(x)\Big),\;\;\;x\in E.$$

We have that
\begin{equation}\label{A1}
M_N(f)(x)\leq C_1\|\chi_{B_{m_0}}\|_{p(\cdot),q(\cdot)}^{-1},\;\;\;x\in(B_{m_0+4\omega})^{c}.
\end{equation}
Indeed, let $x\in(B_{m_0+4\omega})^{c}$. It was proved in \cite[pp. 3092 and 3093]{BLYZ} that
$$M_N(f)(x)\leq C \inf_{u\in B_{m_0}}M_N(f)(u).$$
Then, we obtain
\begin{align*}
        M_N(f)(x)\leq & {C\over{\|\chi_{B_{m_0}}\|_{p(\cdot),q(\cdot)}}}\Big\|\inf_{u\in B_{m_0}}[M_N(f)(u)]\chi_{B_{m_0}}\Big\|_{p(\cdot),q(\cdot)} \\
       & \leq {C\over{\|\chi_{B_{m_0}}\|_{p(\cdot),q(\cdot)}}}\|M_N(f)\chi_{B_{m_0}}\|_{p(\cdot),q(\cdot)} \\
       & \leq {C\over{\|\chi_{B_{m_0}}\|_{p(\cdot),q(\cdot)}}}\|M_N(f)\|_{p(\cdot),q(\cdot)} \\
       & \leq C_1\|\chi_{B_{m_0}}\|_{p(\cdot),q(\cdot)}^{-1}.
        \end{align*}
Thus, (\ref{A1}) is established.

We now choose $j_0\in\Bbb Z$ such that $2^{j_0}< C_1\|\chi_{B_{m_0+4\omega}}\|_{p(\cdot),q(\cdot)}^{-1}\leq 2^{j_0+1}$, where $C_1$ is the constant appearing in (\ref{A1}). We have that
$$\Omega_j\subset B_{m_0+4\omega},\;\;\;j>j_0.$$
By following the ideas developed in \cite{MSV} (see also \cite{BLYZ}) we define
$$h=\sum_{j\leq j_0}\sum_{i\in \mathbb{N}}\lambda_{i,j}a_{i,j}\;\;\;\mbox{and}\;\;\;{\frak l}=\sum_{j> j_0}\sum_{i\in \mathbb{N}}\lambda_{i,j}a_{i,j}.$$
Note that the series converges unconditionally in $S'({\Bbb R}^n)$ and almost everywhere. We have that $\cup_{j>j_0}\Omega_j\subset B_{m_0+4\omega}$. Then, $\mbox{supp}\;{\frak l}\subset B_{m_0+4\omega}$. Since $\mbox{supp}\;f\subset B_{m_0+4\omega}$, also that $\mbox{supp}\;h\subset B_{m_0+4\omega}$. As above we can see that
$$\int_{{\Bbb R}^n}\sum_{j> j_0}\sum_{i\in \mathbb{N}}|\lambda_{i,j}a_{i,j}(x)x^{\alpha}|dx=\int_{B_{m_0+4\omega}}\sum_{j> j_0}\sum_{i\in \mathbb{N}}|\lambda_{i,j}a_{i,j}(x)x^{\alpha}|dx<\infty.$$
Then $\int_{{\Bbb R}^n}\frak{l}(x)x^{\alpha}dx=0$, for every $\alpha\in{\Bbb N}^n$ and $|\alpha|\leq s$. Since  $\int_{{\Bbb R}^n}f(x)x^{\alpha}dx=0$,  for every $\alpha\in{\Bbb N}^n$ and $|\alpha|\leq s$, we have also that $\int_{{\Bbb R}^n}h(x)x^{\alpha}dx=0$,  for every $\alpha\in{\Bbb N}^n$ and $|\alpha|\leq s$.

Moreover, by using (\ref{eq3.5})  we get
$$|h(x)|\leq C\sum_{j\leq j_0}2^j\leq C2^{j_0}\leq C_2\|\chi_{B_{m_0+4\omega}}\|_{p(\cdot),q(\cdot)}^{-1}.$$
Here $C_2$ does not depend on $f$. Hence, $h/C_2$ is a $(p(\cdot),q(\cdot),\infty,s)$-atom associated to the ball $B_{m_0+4\omega}$.

As in \cite[Step 4, p. 3094]{BLYZ} we can see that, if $F_J=\{(i,j):\;i\in\Bbb N,\;j\in\Bbb Z,\;j>j_0\;\mbox{and}\;i+|j|\leq J\}$, and ${\frak l}_J=\sum_{(i,j)\in F_J}\lambda_{i,j}a_{i,j}$, for every $J\in\Bbb N$ such that $J>|j_0|$, then $\lim_{J\rightarrow +\infty}{\frak l}_J={\frak l}$, in $L^r({\Bbb R}^n)$. Moreover, we can find $J$ large enough such that ${\frak l}-{\frak l}_{J}$ is a $(p(\cdot),q(\cdot),r,s)$-atom associated to the ball $B_{m_0+4\omega}$. We have that $f=h+{\frak l}_{J}+({\frak l}-{\frak l}_{J})$ and

\begin{align*}
    \|f\|_{H^{p(\cdot),q(\cdot),r,s}({\Bbb R}^n,A)}&\leq \Big\|C_2{{\chi_{B_{m_0+4\omega}}}\over{\|\chi_{B_{m_0+4\omega}}\|_{p(\cdot),q(\cdot)}}}\\
&\quad +\sum_{(i,j)\in F_J}\lambda_{i,j}{{\chi_{x_{i,j}+B_{\ell_{i,j}+4\omega}}}\over{\|\chi_{x_{i,j}+B_{\ell_{i,j}+4\omega}}\|_{p(\cdot),q(\cdot)}}} +  {{\chi_{B_{m_0+4\omega}}}\over{\|\chi_{B_{m_0+4\omega}}\|_{p(\cdot),q(\cdot)}}}\Big\|_{p(\cdot),q(\cdot)} \\
& \leq C\Big(C_2+1+\Big\| \sum_{(i,j)\in F_J}\lambda_{i,j}{{\chi_{x_{i,j}+B_{\ell_{i,j}+4\omega}}}\over{\|\chi_{x_{i,j}+B_{\ell_{i,j}+4\omega}}\|_{p(\cdot),q(\cdot)}}} \Big\|_{p(\cdot),q(\cdot)}\Big) \leq C. \\
\end{align*}
Thus, (i) is established.

\item[(ii)] This assertion can be proved by using Theorem \ref{Th1.2} and by proceeding as in \cite[Steps 5 and 6, pp. 3094 and 3095]{BLYZ} (see also \cite[pp. 2926 and 2927]{MSV}).

\end{enumerate}

\section{Applications (proof of Theorem \ref{Th1.4})}

In this section we present a proof of Theorem \ref{Th1.4}. We are inspired in some ideas due to Cruz-Uribe and Wang \cite{CrW} but their arguments have to be modified to adapt them to anisotropic setting and variable exponent Lorentz spaces.

First of all we formulate the type of operator which we are working  with. We consider an operator $T:S({\Bbb R}^n)\rightarrow S'({\Bbb R}^n)$ that commutes with translations. It is well-known that this commuting property is equivalent to both the fact that $T$ commutes with convolutions and that there exists $L\in  S'({\Bbb R}^n)$ such that
$$T(\phi)=L*\phi,\;\;\;\phi\in S({\Bbb R}^n).$$
Assume that
\begin{enumerate}
\item[(i)] The Fourier transform $\hat L$ of $L$ is in $L^\infty({\Bbb R}^n)$.

\noindent This property is equivalent to that the operator $T$ can be extended to $L^2({\Bbb R}^n)$ as a bounded operator from  $L^2({\Bbb R}^n)$  into itself.

We say that $T$ is associated to a measurable function $K:{\Bbb R}^n\backslash\{0\}\rightarrow\Bbb R$ when, for every $\phi\in L^\infty_{c}({\Bbb R}^n)$, the space of $L^\infty({\Bbb R}^n)$-functions with compact support,
\begin{equation}\label{Ap1}
T(\phi)(x)=\int_{{\Bbb R}^n}K(x-y)\phi(y)dy,\;\;\;\;x\notin \mbox{supp }\phi.
\end{equation}
We assume that $K$ satisfies the following properties: there exists $C_K>0$ such that
\item[(ii)] $|K(x)|\leq {{C_K}\over{\rho(x)}}$, $x\in {\Bbb R}^n\backslash\{0\}$,
\item[(iii)] for some $\gamma>0$,
$$|K(x-y)-K(x)|\leq C_K{{{\rho(y)}^\gamma}\over{{\rho(x)}^{\gamma+1}}},\;\;\;\;b^{2\omega}\rho(y)\leq\rho(x).$$
\end{enumerate}
An operator $T$ satisfying the above properties is usually called a Calder\'on-Zygmund singular integral in our anisotropic context. These operators and other ones related with them have been studied, for instance, in \cite{Bow1}, \cite{LYY} and \cite{Wang}. In \cite{Wang} some sufficient conditions are given in order that a measurable function $K:{\Bbb R}^n\backslash\{0\}\rightarrow\Bbb R$ defines by (\ref{Ap1}) a principal value integral tempered distribution having a Fourier transform in $L^\infty({\Bbb R}^n)$.

If $T$ is an anisotropic Calder\'on-Zygmund singular integral, $T$ can be extended from $L^2({\Bbb R}^n)\cap L^p({\Bbb R}^n,\nu)$ to $L^p({\Bbb R}^n,\nu)$ as a bounded operator from $L^p({\Bbb R}^n,\nu)$ into itself, for every $1<p<\infty$ and $\nu\in{\mathcal A}_p({\Bbb R}^n,A)$, and as a bounded operator from  $L^1({\Bbb R}^n,\nu)$ into $L^{1,\infty}({\Bbb R}^n,\nu)$, for every $\nu\in{\mathcal A}_1({\Bbb R}^n,A)$. Also,  anisotropic Calder\'on-Zygmund singular integrals satisfy the following Kolmogorov type inequality.

\begin{Prop}\label{Prop6.1}
Let $T$ be an anisotropic Calder\'on-Zygmund singular integral. If $\nu\in{\mathcal A}_1({\Bbb R}^n,A)$ and $0<r<1$, there exists $C>0$ such that, for every $x_0\in{\Bbb R}^n$ and $\ell\in\Bbb Z$.
$$\int_{x_0+B_{\ell}}|Tf(x)|^r\nu(x)dx\leq C\nu(x_0+B_{\ell})^{1-r}\Big(\int_{{\Bbb R}^n}|f(x)|\nu(x)dx\Big)^r,\;\;\;\;f\in L^1({\Bbb R}^n,\nu).$$
Here $C=C([\nu]_{\mathcal{A}_1(\mathbb{R}^n,A)},r)$.
\end{Prop}

\begin{proof}
This property can be proved by taking into account that the operator $T$ is bounded from $L^1({\Bbb R}^n,\nu)$ into $L^{1,\infty}({\Bbb R}^n,\nu)$, provided that $\nu\in{\mathcal A}_1({\Bbb R}^n,A)$. Indeed, let $\nu\in{\mathcal A}_1({\Bbb R}^n,A)$ and $0<r<1$. For every $f\in L^1({\Bbb R}^n,\nu)$, we have that
\begin{align*}
       &  \int_{a+B_{\ell}}|Tf(x)|^r\nu(x)dx=r\int_0^\infty\lambda^{r-1}\nu(\{x\in a+B_{\ell}:\;|Tf(x)|>\lambda\})d\lambda \\
       & \leq C\int_0^\infty\lambda^{r-1}\min\Big\{\nu(a+B_{\ell}),{{\|f\|_{L^1({\Bbb R}^n,\nu)}}\over{\lambda}}\Big\}d\lambda \\
       & \leq C\Big(\nu(a+B_{\ell})\int_0^{{\|f\|_{L^1({\Bbb R}^n,\nu)}}\over{\nu(a+B_{\ell})}}\lambda^{r-1}d\lambda+\|f\|_{L^1({\Bbb R}^n,\nu)}\int_{{\|f\|_{L^1({\Bbb R}^n,\nu)}}\over{\nu(a+B_{\ell})}}^\infty\lambda^{r-2}d\lambda\Big) \\
       & \leq C\nu(a+B_{\ell})^{1-r}\Big(\int_{{\Bbb R}^n}|f(x)|\nu(x)dx\Big)^r.
        \end{align*}
\end{proof}

In order to study Calder\'on-Zygmund singular integrals in Hardy spaces it is usual to require on the kernel $K$ more restrictive regularity conditions than the above ones (ii) and (iii).

As in \cite[p. 61]{Bow1} (see also \cite{Hu}) we say that the anisotropic Calder\'on-Zygmund singular integral $T$ associated with the kernel $K$ is of order $m$ when $K\in C^m({\Bbb R}^n\backslash\{0\})$ and there exists $C_{K,m}>0$ such that, for every $x,y\in{\Bbb R}^n$, $x\neq y$,
\begin{equation}\label{Ap2}
|(\partial_y^\alpha \tilde K)(x,A^{-k}y)|\leq{{C_{K,m}}\over{\rho(x-y)}}=C_{K,m}b^{-k},\quad \alpha\in{\Bbb N}^n,\;|\alpha|\leq m,
\end{equation}
where $k$ is the unique integer such that $x-y\in B_{k+1}\backslash B_k$. Here $\tilde K$ is defined by
$$\tilde K(x,y)=K(x-A^{k}y),\;\;\;x,y\in{\Bbb R}^n,\;x-y\in B_{k+1}\backslash B_k.$$
As it can be seen in \cite[p. 61]{Bow1} this property reduces to the usual condition in the isotropic setting.

In order to prove Theorem \ref{Th1.4} we need to consider weighted finite atomic anisotropic Hardy spaces as follows.

Let $p,q\in {\Bbb P}_0$, $r>1$, $s\in\Bbb N$, $p_0\in(0,1)$ and $\nu\in {\mathcal A}_1({\Bbb R}^n,A)$. The space $H^{p(\cdot),q(\cdot),r,s}_{p_0,\nu,fin}({\Bbb R}^n,A)$ consists of all finite sums of multiple of $(p(\cdot),q(\cdot),r,s)$-atoms and it is endowed with the norm $\|\cdot\|_{H^{p(\cdot),q(\cdot),r,s}_{p_0,\nu,fin}({\Bbb R}^n,A)}$ defined as follows: for every $f\in H^{p(\cdot),q(\cdot),r,s}_{p_0,\nu,fin}({\Bbb R}^n,A)$,
$$\|f\|_{H^{p(\cdot),q(\cdot),r,s}_{p_0,\nu,fin}({\Bbb R}^n,A)}=\inf\Big\{\Big\|\sum_{j=1}^k\lambda_j^{p_0}\|\chi_{x_j+B_{\ell_j}}\|^{-p_0}_{p(\cdot),q(\cdot)}\chi_{x_j+B_{\ell_j}}\Big\|_{L^1(\mathbb{R}^n,\nu)}^{1/p_0}:\;f=\sum_{j=1}^k\lambda_ja_j\Big\},$$
where the infimum, as usual, is taken over all the possible finite decompositions. Note that according to Proposition \ref{prop4.3}, if $\max\{1,q_+\}<r<\infty$ and $s\geq s_0$, being $s_0$ the same as in Proposition \ref{prop4.3}, then $H^{p(\cdot),q(\cdot),r,s}_{p_0,\nu,fin}({\Bbb R}^n,A)=H^{p(\cdot),q(\cdot),r,s}_{fin}({\Bbb R}^n,A)$ as sets.

The following property will be useful in the sequel.
\begin{Lem}\label{Lem6.1}
Let $p,q\in {\Bbb P}_0$,  $\max\{1,q_+\}<r<\infty$ and $s\in\Bbb N$. There exists $s_0\in\Bbb N$ such that if $s\geq s_0$, $p_0<\min\{p_-,q_-\}$ and $\nu\in{\mathcal A}_1({\Bbb R}^n,A)\cap \mathcal{L}^{(p(\cdot)/p_0)',(q(\cdot)/p_0)'}({\Bbb R}^n)$ we can find $C>0$ such that, for every $f\in H^{p(\cdot),q(\cdot),r,s}_{p_0,\nu,fin}({\Bbb R}^n,A)$,
$$\|f\|_{H^{p(\cdot),q(\cdot),r,s}_{p_0,\nu,fin}({\Bbb R}^n,A)}\leq C\|f\|_{H^{p_0}({\Bbb R}^n,\nu,A)}.$$
\end{Lem}
\begin{proof}
The proof of this property follows the same ideas that the ones in the proof of Theorem \ref{Th1.3}. Let $s_0$ as in Proposition \ref{prop4.3} and $f\in H^{p(\cdot),q(\cdot),r,s}_{p_0,\nu,fin}({\Bbb R}^n,A)$, with $s\geq s_0$. Then, $f\in H^{p(\cdot),q(\cdot),r,s}_{fin}({\Bbb R}^n,A)$ and there exists $m_0\in\Bbb Z$ such that $\mbox{supp} f\subset B_{m_0}$. Also, $f\in L^r({\Bbb R}^n)$ and, as we proved in (\ref{A1}), $M_N(f)(x)\leq C_1\|\chi_{B_{m_0}}\|_{p(\cdot),q(\cdot)}^{-1}$, when $x\in(B_{m_0+4\omega})^{c}$.

Assume that $\|f\|_{H^{p_0}({\Bbb R}^n,\nu,A)}=1$. Our objective is to see that $\|f\|_{H^{p(\cdot),q(\cdot),r,s}_{p_0,\nu,fin}({\Bbb R}^n,A)}\leq C$, for some $C>0$ that does not depend on $f$.

A careful reading of the proof of \cite[Lemma 5.4]{BLYZ} allows us to see that there exist a sequence $\{x_{i,k}\}_{i\in\Bbb N,k\in\Bbb Z}\subset \mathbb{R}^n$, a sequence $\{\ell _{i,k}\}_{i\in\Bbb N,k\in\Bbb Z}\subset \mathbb{Z}$ and a bounded sequence $\{b_{i,k}\}_{i\in\Bbb N,k\in\Bbb Z}$ such that
\begin{enumerate}
\item $f=\sum_{k\in\Bbb Z}(\sum_{i\in\Bbb N}2^kb_{i,k})$,

\noindent where the convergence is unconditional in $S'({\Bbb R}^n)$ and almost everywhere of ${\Bbb R}^n$;
\item for a certain $s_1\in\Bbb N$, $\int_{{\Bbb R}^n}b_{i,k}(x)x^{\alpha}dx=0$, $\alpha\in{\Bbb N}^n$ and $|\alpha|\leq s_1$;
\item $\mbox{supp}(b_{i,k})\subset x_{i,k}+B_{\ell_{i,k}+4\omega}$, $k\in\Bbb Z$ and $i\in\Bbb N$;
\item $\Omega_k:=\{x\in{\Bbb R}^n:\;M_N(f)(x)>2^k\}=\cup_{i\in\Bbb N}(x_{i,k}+B_{\ell_{i,k}+4\omega})$, $k\in\Bbb Z$;
\item there exists $L\in\Bbb N$ for which $\sharp\{j\in\Bbb N:\;(x_{i,k}+B_{\ell_{i,k}+2\omega})\cap(x_{j,k}+B_{\ell_{j,k}+2\omega})\neq \emptyset\}\leq L$, $i\in\Bbb N$ and $k\in\Bbb Z$.
\end{enumerate}
We define, for every $k\in\Bbb Z$ and $i\in\Bbb N$, $\lambda_{i,k}=2^k\|\chi_{x_{i,k}+B_{\ell_{i,k}}}\|_{p(\cdot),q(\cdot)}$ and $a_{i,k}=b_{i,k}\|\chi_{x_{i,k}+B_{\ell_{i,k}}}\|_{p(\cdot),q(\cdot)}^{-1}$. There exists $C_0>0$ such that $C_0a_{i,k}$ is a $(p(\cdot),q(\cdot),\infty,s_1)$-atom, for every $k\in\Bbb Z$ and $i\in\Bbb N$.

We have that
$$\sum_{i\in \mathbb{N}}\lambda_{i,k}^{p_0}{{\chi_{x_{i,k}+B_{\ell_{i,k}}}(x)}\over{\|\chi_{x_{i,k}+B_{\ell_{i,k}}}\|^{p_0}_{p(\cdot),q(\cdot)}}}\leq C2^{kp_0}\chi_{\Omega_k}(x),\;\;\;k\in\Bbb Z\;\mbox{and}\;x\in{\Bbb R}^n.$$
Then,
\begin{align*}
\Big\|\sum_{k\in\Bbb Z}\sum_{i\in \mathbb{N}}\lambda_{i,k}^{p_0} & {{\chi_{x_{i,k}+B_{\ell_{i,k}}}}\over{\|\chi_{x_{i,k}+B_{\ell_{i,k}}}\|^{p_0}_{p(\cdot),q(\cdot)}}}\Big\|_{L^1(\mathbb{R}^n,\nu)}\leq C\Big\|\sum_{k\in\Bbb Z}2^{kp_0}\chi_{\Omega_k}\Big\|_{L^1(\mathbb{R}^n,\nu)} \\
& \leq C\|M_N(f)^{p_0}\|_{L^1(\mathbb{R}^n,\nu)}=C\|M_N(f)\|_{L^{p_0}(\mathbb{R}^n,\nu)}^{p_0}=C\|f\|_{H^{p_0}({\Bbb R}^n,\nu,A)}^{p_0}.
 \end{align*}
 We now choose $k_0\in\Bbb Z$ such that $2^{k_0}\leq \|\chi_{B_{m_0}}\|^{-1}_{p(\cdot),q(\cdot)}$. We define, as in the proof of Theorem \ref{Th1.3},
 $$h=\sum_{k\leq k_0}\sum_{i\in \mathbb{N}}\lambda_{i,k}a_{i,k}\;\;\;\mbox{and}\;\;\;{\frak l}=\sum_{k> k_0}\sum_{i\in \mathbb{N}}\lambda_{i,k}a_{i,k},$$
 where the convergence of the two series is unconditional in $S'({\Bbb R}^n)$ and almost everywhere of ${\Bbb R}^n$. We have that
 \begin{enumerate}
 \item[(a)] There exists $C_1>0$ that does not depend on $f$  such that $h/C_1$ is a $(p(\cdot),q(\cdot),\infty,s_1)$-atom;
 \item[(b)] By defining, for every $J\in\Bbb N$, $F_J$ and ${\frak l}_J$ as in the proof of Theorem \ref{Th1.3}, there exists $J_1\in\Bbb N$ such that ${\frak l}-{\frak l}_{J_1}$ is a $(p(\cdot),q(\cdot),r,s_1)$-atom;
 \item[(c)] $f=C_1{h\over{C_1}}+ ({\frak l}-{\frak l}_{J_1})+{\frak l}_{J_1}$.
 \end{enumerate}
 Then, for $s\geq \max\{s_0,s_1\}$ we can write
 \begin{align*}
         \|f\|_{H^{p(\cdot),q(\cdot),r,s}_{p_0,\nu,fin}({\Bbb R}^n,A)}&\\
&\hspace{-2cm} \leq \Big\|C_1^{p_0}{{\chi_{B_{m_0}}}\over{\|\chi_{B_{m_0}}\|_{p(\cdot),q(\cdot)}^{p_0}}}+{{\chi_{B_{m_0}}}\over{\|\chi_{B_{m_0}}\|_{p(\cdot),q(\cdot)}^{p_0}}}+\sum_{(i,k)\in F_{J_1}} \lambda_{i,k}^{p_0}{{\chi_{x_{i,k}+B_{\ell_{i,k}+4\omega}}}\over{\|\chi_{x_{i,k}+B_{\ell_{i,k}+4\omega}}}\|_{p(\cdot),q(\cdot)}^{p_0}}\Big\|_{L^1(\mathbb{R}^n,\nu)} \\
&\hspace{-2cm} \leq C\Big(1+{{\nu(B_{m_0})}\over{\|\chi_{B_{m_0}}\|_{p(\cdot),q(\cdot)}^{p_0}}}\Big)=C\Big(1+{{\nu(B_{m_0})}\over{\|\chi_{B_{m_0}}\|_{p(\cdot)/p_0,q(\cdot)/p_0}}}\Big)
  \leq C\Big(1+\|\nu\|_{(p(\cdot)/p_0)',(q(\cdot)/p_0)'}\Big).
        \end{align*}
 The last inequality follows because ${\mathcal L}^{(p(\cdot)/p_0)',(q(\cdot)/p_0)'}=( {\mathcal L}^{p(\cdot)/p_0,q(\cdot)/p_0})'$ (see \cite{EKS}), since $p_0<\min\{p_-,q_-\}$, and then
 $$\nu(B_{m_0})=\int_{{\Bbb R}^n}\chi_{B_{m_0}}(x)\nu(x)dx\leq \|\chi_{B_{m_0}}\|_{p(\cdot)/p_0,q(\cdot)/p_0}\|\nu\|_{(p(\cdot)/p_0)',(q(\cdot)/p_0)'}.$$
 Hence, $\|f\|_{H^{p(\cdot),q(\cdot),r,s}_{p_0,\nu,fin}({\Bbb R}^n,A)}\leq C$, where $C$ does not depend on $f$.
\end{proof}

We now prove a general boundedness result for sublinear operators.

\begin{Prop}\label{Prop6.2}
Assume that $p,q\in {\Bbb P}_0$, $p(0)<q(0)$, $0<p_0<\min\{p_-,q_-,1\}$, $\max\{1,q_+,\}<r$ and $s\in\Bbb N$. There exists $s_0\in\Bbb N$ and $r_0>1$ such that if $s\geq s_0$, $r>r_0$ and $T$ is a sublinear operator defined on $\mbox{span}\{a:\;a\;\mbox{is a }\;(p(\cdot),q(\cdot),r,s)\mbox{-atom}\}$, then
\begin{enumerate}
\item[(i)] $T$ has a (unique) extension on $H^{p(\cdot),q(\cdot)}({\Bbb R}^n,A)$ as a bounded operator from $H^{p(\cdot),q(\cdot)}({\Bbb R}^n,A)$ into ${\mathcal L}^{p(\cdot),q(\cdot)}({\Bbb R}^n)$, provided that for each $\nu\in {\mathcal A}_1({\Bbb R}^n,A)\cap RH_{(r/p_0)'}({\Bbb R}^n,A)$ there exists $C=C([\nu]_{{\mathcal A}_1({\Bbb R}^n,A)},[\nu]_{RH_{(r/p_0)'}({\Bbb R}^n,A)})>0$ such that
    $$\|Ta\|_{L^{p_0}(\mathbb{R}^n, \nu)}\leq C {{\nu(x_0+B_{\ell_0})^{1/p_0}}\over{\|\chi_{x_0+B_{\ell_0}}\|_{p(\cdot),q(\cdot)}}},$$
    for every $(p(\cdot),q(\cdot),r/p_0,s)$-atom $a$ associated with $x_0\in{\Bbb R}^n$ and $\ell_0\in\Bbb Z$.
\item[(ii)] $T$ has a (unique) extension on $H^{p(\cdot),q(\cdot)}({\Bbb R}^n,A)$ as a bounded operator from $H^{p(\cdot),q(\cdot)}({\Bbb R}^n,A)$ into itself, provided that for each $\nu\in {\mathcal A}_1({\Bbb R}^n,A)\cap RH_{(r/p_0)'}({\Bbb R}^n,A)$ there exists $C=C([\nu]_{{\mathcal A}_1({\Bbb R}^n,A)},[\nu]_{RH_{(r/p_0)'}({\Bbb R}^n,A)})>0$ such that
    $$\|Ta\|_{H^{p_0}({\Bbb R}^n,\nu,A)}\leq C {{\nu(x_0+B_{\ell_0})^{1/p_0}}\over{\|\chi_{x_0+B_{\ell_0}}\|_{p(\cdot),q(\cdot)}}},$$
    for every $(p(\cdot),q(\cdot),r/p_0,s)$-atom $a$ associated with $x_0\in{\Bbb R}^n$ and $\ell_0\in\Bbb Z$.
\end{enumerate}
\end{Prop}

\begin{proof}
\begin{enumerate}
\item[(i)] Suppose that for every $\nu\in {\mathcal A}_1({\Bbb R}^n,A)\cap RH_{(r/p_0)'}({\Bbb R}^n,A)$ there exists $C>0$ such that, for every $(p(\cdot),q(\cdot),r/p_0,s)$-atom $a$ associated with $x_0\in{\Bbb R}^n$ and $\ell_0\in\Bbb Z$,
\begin{equation}\label{eq6.1}
\|Ta\|_{L^{p_0}({\Bbb R}^n,\nu)}\leq C {{\nu(x_0+B_{\ell_0})^{1/p_0}}\over{\|\chi_{x_0+B_{\ell_0}}\|_{p(\cdot),q(\cdot)}}}.
\end{equation}
Here $C$ can depend on $[\nu]_{{\mathcal A}_1({\Bbb R}^n,A)}$ and $[\nu]_{RH_{(r/p_0)'}({\Bbb R}^n,A)}$.

The set $H^{p(\cdot),q(\cdot),r/p_0,s}_{fin}({\Bbb R}^n,A)$ is dense in $H^{p(\cdot),q(\cdot)}({\Bbb R}^n,A)$ (see Theorem \ref{Th1.3}). Hence, in order to see that there exists an extension $\widetilde{T}$ of $T$ to $H^{p(\cdot),q(\cdot)}({\Bbb R}^n,A)$ as a bounded operator from $H^{p(\cdot),q(\cdot)}({\Bbb R}^n,A)$ into ${\mathcal L}^{p(\cdot),q(\cdot)}({\Bbb R}^n)$, it is sufficient to prove that, there exists $C>0$ such that
$$\|T(f)\|_{p(\cdot),q(\cdot)}\leq C\|f\|_{H^{p(\cdot),q(\cdot)}({\Bbb R}^n,A)}, \;\;\;f\in H^{p(\cdot),q(\cdot),r/p_0,s}_{fin}({\Bbb R}^n,A).$$
Let $f\in H^{p(\cdot),q(\cdot),r/p_0,s}_{fin}({\Bbb R}^n,A)$. As in the proof of Proposition \ref{prop4.6} the Rubio de Francia's iteration algorithm allows us to write,
$$\|T(f)\|_{p(\cdot),q(\cdot)}^{p_0}=\|(Tf)^{p_0}\|_{p(\cdot)/p_0,q(\cdot)/p_0}\leq \sup \int_{{\Bbb R}^n}|Tf(x)|^{p_0}Rh(x)dx,$$
where the supremum is taken over all the functions $h\in {\mathcal L}^{(p(\cdot)/p_0)',(q(\cdot)/p_0)'}({\Bbb R}^n)$ such that $\|h\|_{(p(\cdot)/p_0)',(q(\cdot)/p_0)'}\leq 1$. Also, there exists $r_1>1$ such that if $r>r_1$ we can find $C>0$ such that, for every $h\in {\mathcal L}^{(p(\cdot)/p_0)',(q(\cdot)/p_0)'}({\Bbb R}^n)$, $Rh\in {\mathcal A}_1({\Bbb R}^n,A)\cap RH_{(r/p_0)'}({\Bbb R}^n,A)$ and $[Rh]_{{\mathcal A}_1({\Bbb R}^n,A)}+ [Rh]_{RH_{(r/p_0)'}({\Bbb R}^n,A)}\leq C$.

Let $h\in {\mathcal L}^{(p(\cdot)/p_0)',(q(\cdot)/p_0)'}({\Bbb R}^n)$ such that $\|h\|_{(p(\cdot)/p_0)',(q(\cdot)/p_0)'}\leq 1$. We are going to estimate $\|T(f)\|_{L^{p_0}({\Bbb R}^n,Rh)}$. As it was mentioned above 
$$
H^{p(\cdot),q(\cdot),r/p_0,s}_{fin}({\Bbb R}^n,A)=H^{p(\cdot),q(\cdot),r/p_0,s}_{p_0, Rh,fin}({\Bbb R}^n,A).
$$
 We write $f=\sum_{j=1}^k\lambda_j a_j$, where  for every $j\in\Bbb N$, $j\leq k$, $\lambda_j>0$ and $a_j$ is a $(p(\cdot),q(\cdot),r/p_0,s)$-atom associated with $x_j\in{\Bbb R}^n$ and $\ell_j\in\Bbb Z$. Since $0<p_0<1$ and $T$ is sublinear, from (\ref{eq6.1}) we deduce that

\begin{align*}
       \|T(f)\|_{L^{p_0}(\mathbb{R}^n, Rh)}^{p_0} & = \int_{{\Bbb R}^n}|T(f)(x)|^{p_0}Rh(x)dx\leq \sum_{j=1}^k\lambda_j^{p_0}\int_{{\Bbb R}^n}|Ta_j(x)|^{p_0}Rh(x)dx \\
       & \leq C\sum_{j=1}^k\lambda_j^{p_0}{{Rh(x_j+B_{\ell_j})}\over{\|\chi_{x_j+B_{\ell_j}}\|_{p(\cdot),q(\cdot)}^{p_0}}}=C\Big\|\sum_{j=1}^k\lambda_j^{p_0}{{\chi_{x_j+B_{\ell_j}}}\over{\|\chi_{x_j+B_{\ell_j}}\|
       _{p(\cdot),q(\cdot)}^{p_0}}}\Big\|_{L^1(\mathbb{R}^n,Rh)}.
       \end{align*}
As it was established in the proof of Proposition \ref{prop4.6}, $Rh\in {\mathcal A}_1({\Bbb R}^n,A)\cap {\mathcal L}^{(p(\cdot)/p_0)',(q(\cdot)/p_0)' }({\Bbb R}^n)$. According to Lemma \ref{Lem6.1}, the arbitrariness of the representation of $f$ leads to
$$ \|T(f)\|_{L^{p_0}(\mathbb{R}^n,Rh)}\leq C\|f\|_{H^{p_0}({\Bbb R}^n,Rh,A)}.$$
Since $R$ is bounded from ${\mathcal L}^{(p(\cdot)/p_0)',(q(\cdot)/p_0)' }({\Bbb R}^n)$ into itself, we can write
\begin{align*}
     &  \|T(f)\|_{L^{p_0}(\mathbb{R}^n,Rh)}\leq C \|f\|_{H^{p_0}(\mathbb{R}^n,Rh,A)}\leq C \int_{{\Bbb R}^n}(M_N(f)(x))^{p_0}Rh(x)dx \\
       & \leq C\|(M_N(f))^{p_0}\|_{p(\cdot)/p_0,q(\cdot)/p_0}\|Rh\|_{(p(\cdot)/p_0)',(q(\cdot)/p_0)'} \\
  & \leq C\|(M_N(f))^{p_0}\|_{p(\cdot)/p_0,q(\cdot)/p_0}=C\|M_N(f)\|_{p(\cdot),q(\cdot)}^{p_0}=C\|f\|_{H^{p(\cdot),q(\cdot)}({\Bbb R}^n,A)}^{p_0},
       \end{align*}
provided that  $h\in {\mathcal L}^{(p(\cdot)/p_0)',(q(\cdot)/p_0)'}({\Bbb R}^n)$ and $\|h\|_{(p(\cdot)/p_0)',(q(\cdot)/p_0)'}\leq 1$.

We conclude that
$$
\|T(f)\|_{p(\cdot),q(\cdot)}\leq C\|f\|_{H^{p(\cdot),q(\cdot)}({\Bbb R}^n,A)},
$$
and the proof of (i) is finished.

 \item[(ii)] In order to prove the property (ii), we proceed in a similar way as in the proof of (i). Assume tha $\varphi\in S({\Bbb R}^n)$ such that $\int\varphi dx\neq 0$. Let $f\in H^{p(\cdot),q(\cdot),r/p_0,s}_{fin}({\Bbb R}^n,A)$, with $s\geq s_0$, and $s_0$ as before. We have that

 \begin{align*}
       \|T(f)\|_{H^{p(\cdot),q(\cdot)}({\Bbb R}^n,A)}^{p_0} & \leq C \|M_\varphi^0(Tf)\|^{p_0}_{p(\cdot),q(\cdot)}=C \|(M_\varphi^0(Tf))^{p_0}\|_{p(\cdot)/p_0,q(\cdot)/p_0} \\ & \leq C\sup\int_{{\Bbb R}^n}(M_\varphi^0(Tf)(x))^{p_0}Rh(x)dx\leq C\sup \|T(f)\|_{H^{p_0}({\Bbb R}^n,Rh,A)}^{p_0},
       \end{align*}
       where the supremum is taken over all the functions $h\in {\mathcal L}^{(p(\cdot)/p_0)',(q(\cdot)/p_0)'}({\Bbb R}^n)$ such that $\|h\|_{(p(\cdot)/p_0)',(q(\cdot)/p_0)'}\leq 1$.

 We now finish the proof as the one of (i) provided that, for every $\nu\in {\mathcal A}_1({\Bbb R}^n,A)\cap RH_{(r/p_0)'}({\Bbb R}^n,A)$ there exists $C>0$ such that
 $$\|Ta\|_{H^{p_0}({\Bbb R}^n,\nu,A)}\leq C {{\nu(x_j+B_{\ell_j})^{1/p_0}}\over{\|\chi_{x_j+B_{\ell_j}}\|_{p(\cdot),q(\cdot)}}},$$
 for every $(p(\cdot),q(\cdot),r/p_0,s)$-atom $a$ associated with $x_j\in{\Bbb R}^n$ and $\ell_j\in\Bbb Z$. Here the constant $C$ can depend on $[\nu]_{{\mathcal A}_1({\Bbb R}^n,A)}$ and $[\nu]_{RH_{(r/p_0)'}({\Bbb R}^n,A)}$.

\end{enumerate}
\end{proof}

We now prove Theorem \ref{Th1.4}  by applying  the criterions established in Proposition \ref{Prop6.2}.

\begin{proof}[Proof of Theorem \ref{Th1.4}, $(i)$.]
Assume that $a$ is a $(p(\cdot ),q(\cdot ),r/p_0, s)$-atom associated with $x_0\in \mathbb{R}^n$ and $\ell _0\in \mathbb{Z}$, and $\nu \in \mathcal{A}_1(\mathbb{R}^n,A)\bigcap RH_{(r/p_0)'}(\mathbb{R}^n, A)$. Here $p_0, r$ and $s$ are as in Proposition \ref{Prop6.2}. We can write
$$
\|Ta\|_{L^{p_0}(\mathbb{R}^n,\nu )}^{p_0}=\int_{x_0+B_{\ell _0+w}}|T(a)(x)|^{p_0}\nu (x)dx+\int_{(x_0+B_{\ell _0+w})^c}|T(a)(x)|^{p_0}\nu (x)dx=I_1+I_2.
$$
According to Proposition \ref{Prop6.1} there exits $C>0$ such that
\begin{align*}
I_1&\leq C\nu (x_0+B_{\ell _0+w})^{1-p_0}\left(\int_{\mathbb{R}^n}|a(x)|\nu (x)dx\right)^{p_0}\\
&\leq C\nu (x_0+B_{\ell _0})^{1-p_0}|B_{\ell _0}|^{p_0}\left[\left(\frac{1}{|B_{\ell _0}|}\int_{x_0+B_{\ell _0}}|a(x)|^{r/p_0}dx\right)^{p_0/r}\left(\frac{1}{|B_{\ell _0}|}\int_{x_0+B_{\ell _0}}\nu(x)^{(r/p_0)'}dx\right)^{1/(r/p_0)'}\right]^{p_0}.
\end{align*}

We have used that $\nu$ is a doubling measure.

By taking into account that $a$ is a $p(\cdot), q(\cdot ), r/p_0,s)$-atom and that $\nu\in RH_{(r/p_0)'}(\mathbb{R}^n,A)$ we obtain
\begin{align*}
I_1&\leq C\nu (x_0+B_{\ell _0+w})^{1-p_0}|B_{\ell_0}|^{p_0}\left( \frac{1}{\|\chi _{x_0+B_{\ell _0}}\|_{p(\cdot ),q(\cdot )}|B_{\ell _0}|}\int_{x_0+B_{\ell _0}}\nu (x)dx\right)^{p_0}\\
&\leq C\frac{\nu (x_0+B_{\ell _0})}{\|\chi _{x_0+B_{\ell _0}}\|_{p(\cdot ),q(\cdot )}^{p_0}}.
\end{align*}
Note that $C=C([\nu ]_{A_1(\mathbb{R}^n,A)},[\nu ]_{RH_{(r/p_0)'}(\mathbb{R}^n,A)})$.

Since $a$ is a $(p(\cdot ),q(\cdot ), r/p_0, s)$-atom associated with $x_0\in \mathbb{R}^n$ and $\ell _0\in \mathbb{Z}$, by using the condition (\ref{Ap2}) and by proceeding as in \cite[pp. 64 and 65]{Bow1} we deduce that, for every $x\in (x_0+B_{\ell _0+w+\ell +1})\setminus (x_0+B_{\ell_0+w+\ell })$, with $\ell \in \mathbb{N}$,
\begin{align*}
|Ta(x)|&\leq Cb^{-\ell _0-\ell}\sup_{z\in B_{-\ell }}|z|^m\int_{x_0+B_{\ell _0}}|a(y)|dy\\
&\leq Cb^{-\ell _0-\ell }(\lambda _-^{-\ell })^m|B_{\ell _0}|^{1/(r/p _0)'}\|a\|_{r/p_0}\\
&\leq Cb^{-\ell _0}b^{-\ell (\delta +1)}b^{\ell _0/(r/p_0)'}\|a\|_{r/p_0}\\
&\leq Cb^{-\ell _0p_0/r}(\rho (x,x_0)b^{-\ell _0-w})^{-(\delta +1)}\|a\|_{r/p_0}\\
&\leq Cb^{-\ell _0p_0/r}(\rho (x-x_0)b^{-\ell _0-w})^{-(\delta +1)}\|a\|_{r/p_0},
\end{align*}
where $\delta =m\ln \lambda_-/\ln b$.

Then,
\begin{align*}
|Ta(x)|&\leq C\frac{b^{\ell _0(\delta +1)}}{\rho (x-x_0)^{\delta +1}}b^{-\ell _0p_0/r}\frac{|B_{\ell _0}|^{p_0/r}}{\|\chi _{x_0+B_{\ell _0}}\|_{p(\cdot ),q(\cdot )}}\\
&=C\frac{|B_{\ell _0}|^{\delta +1}}{\rho (x-x_0)^{\delta +1}\|\chi _{x_0+B_{\ell _0}}\|_{p(\cdot ),q(\cdot )}},\quad x\not\in x_0+B_{\ell _0+w}.
\end{align*}
Thus,
$$
I_2\leq C\frac{|B_{\ell _0}|^{p_0(\delta +1)}}{\|\chi _{x_0+B_{\ell _0}}\|_{p(\cdot ),q(\cdot )}^{p_0}}\int_{(x_0+B_{\ell _0+w})^c}\frac{\nu (x)}{\rho (x-x_0)^{p_0(\delta +1)}}dx.
$$
Since $(x_0+B_{\ell _0+w})^c=\bigcup_{i=0}^\infty (x_0+B_{\ell _0+w+i+1})\setminus (x_0+B_{\ell _0+w+i})$ and $b^{\ell _0+w+i}\leq \rho (x-x_0)\leq b^{\ell _0+w+i+1}$,
for every $x\in (x_0+B_{\ell _0+w+i+1})\setminus (x_0+B_{\ell _0+w+i})$, $i\in \mathbb{N}$, we have that
\begin{align*}
\int_{(x_0+B_{\ell _0+w})^c}\frac{\nu (x)}{\rho (x-x_0)^{p_0(\delta +1)}}dx&=\sum_{i=0}^\infty \int_{(x_0+B_{\ell _0+w+i+1})\setminus (x_0+B_{\ell _0+w+i})}\frac{\nu (x)}{\rho (x-x_0)^{p_0(\delta +1)}}dx\\
&\hspace{-2cm}\leq \sum_{i=0}^\infty b^{-(\ell _0+w+i)p_0(\delta +1)}\int_{x_0+B_{\ell _0+w+i+1}}\nu (x)dx\\
&\hspace{-2cm}\leq [\nu ]_{A_1(\mathbb{R}^n,A)}\sum_{i=0}^\infty b^{-(\ell _0+w+i)p_0(\delta +1)}|B_{\ell _0+w+i+1}|\mbox{essinf}_{x\in x_0+B_{\ell _0+w+i+1}} \nu (x)\\
&\hspace{-2cm}\leq b[\nu ]_{A_1(\mathbb{R}^n,A)}\sum_{i=0}^\infty b^{-(\ell _0+w+i)(p_0(\delta +1)-1)}\mbox{essinf}_{x\in x_0+B_{\ell _0}} \nu (x)\\
&\hspace{-2cm}\leq b[\nu ]_{A_1(\mathbb{R}^n,A)}\frac{1}{|B_{\ell _0}|}\int_{x_0+B_{\ell _0}}\nu (z)dz\sum_{i=0}^\infty b^{-i(p_0(\delta +1)-1)}b^{-(\ell _0+w)(p_0(\delta +1)-1)}\\
&\hspace{-2cm}=b[\nu ]_{A_1(\mathbb{R}^n,A)}\frac{1}{b^{\ell _0}}\frac{b^{-(\ell _0+w)(p_0(\delta +1)-1)}}{1-b^{-p_0(\delta +1)+1}}\nu (x_0+B_{\ell _0}).
\end{align*}
Note that $p_0>1/(\delta +1)$.

We get
$$
I_2\leq C[\nu]_{A_1(\mathbb{R}^n,A)}\frac{\nu (x_0+B_{\ell _0})}{\|\chi _{x_0+B_{\ell _0}}\|_{p(\cdot ),q(\cdot )}^{p_0}},
$$
where $C$ does not depend on $\nu$.

Hence, for a certain $C=C([\nu]_{A_1(\mathbb{R}^n,A)},[\nu]_{RH_{(1/p_0)'}(\mathbb{R}^n,A)})$,
$$
\|Ta\|_{L^{p_0}(\mathbb{R}^n,\nu )}^{p_0}\leq C\frac{\nu (x_0+B_{\ell _0})}{\|\chi _{x_0+B_{\ell _0}}\|_{p(\cdot ),q(\cdot )}^{p_0}}.
$$

The proof finishes applying Proposition \ref{Prop6.2} $(i)$.

\end{proof}

Before proving Theorem \ref{Th1.4} $(ii)$, we establish the following auxiliar result.

\begin{Lem}\label{Lem6.2}
Let $\phi \in \mathcal{S}(\mathbb{R}^n)$ such that $\supp \phi \subset B_0$ and $\int \phi (x)dx\not =0$. Assume that $L\in \mathcal{S}'(\mathbb{R}^n)$ and that $T_L$ is a Calder\'on-Zygmund singular integral of order $m$. Then, for every $\ell \in \mathbb{Z}$, the operator $S_{(\ell)} =T_{\phi _\ell }\circ T_L$ is a Calder\'on-Zygmund singular integral of order $m$. Moreover, if $S_{(\ell)}$ is associated with the kernel $K_\ell$, there exists $C>0$ such that
$$
\sup_{\ell \in \mathbb{N}}\left\{\|\widehat{S_{(\ell)}}\|_\infty ,C_{K_\ell }, C_{K_{\ell },m}\right\}\leq C.
$$
\end{Lem}
\begin{proof}
Let $\ell \in \mathbb{Z}$. For every $\psi \in  \mathcal{S}(\mathbb{R}^n)$ we have that
$$
S_{(\ell )}(\psi )=T_{\phi _\ell }(T_L(\psi ))=\phi _\ell *(L*\psi )=(L*\phi _\ell )*\psi .
$$
Hence, $S_{(\ell )}=T_{L*\phi _\ell }$. Since $|\widehat{\phi _\ell}|\leq \|\phi \|_1$, the interchange formula leads to
$\|\widehat{S_{(\ell )}}\|_\infty=\|\widehat{L}\widehat{\phi _\ell }\|_\infty\leq \|\widehat{L}\|_\infty \|\phi \|_1$. According to \cite[p. 248]{Sch}, $L*\phi _\ell$ is a multiplier for $\mathcal{S}(\mathbb{R}^n)$ and, for every $\psi \in \mathcal{S}(\mathbb{R}^n)$,
$$
S_{(\ell )}(\psi )(x)=\int_{\mathbb{R}^n}(L*\phi _\ell )(x-y)\psi (y)dy,\quad x\in \mathbb{R}^n.
$$
Note that this integral is absolutely convergent for every $x\in \mathbb{R}^n$. Then, $S_{(\ell)}$ is associated with the kernel $L*\phi _\ell$ which is in $C^\infty (\mathbb{R}^n)$. We define, for every $k\in \mathbb{Z}$, $L_k\in \mathcal{S}'(\mathbb{R}^n)$ as follows:
$$
<L_k,\psi >=<L,\psi (A^k\cdot )>,\quad \psi \in \mathcal{S}(\mathbb{R}^n).
$$
It is not hard to see that, for every $k\in \mathbb{Z}$, $(L*\phi _\ell )_k=L_k*\phi _{\ell +k}$. Then, $(L*\phi _\ell )_{-\ell}=L_{-\ell }*\phi$.

Suppose that $T_L$ is associated with the kernel $K$, that is, for every $\psi \in \mathcal{S}(\mathbb{R}^n)$,
$$
(L*\psi )(x)=\int_{\mathbb{R}^n}K(x-y)\psi (y)dy, \quad x\not \in \mbox{supp}\; \psi ,
$$
and $K$ satisfies (ii) and (iii) after (\ref{Ap1}).

Let $k\in \mathbb{Z}$ and $\psi \in \mathcal{S}(\mathbb{R}^n)$. We have that
\begin{align*}
(L_k*\psi )(x)&=<L_k(y),\psi (x-y)>=<L(y),\psi (x-A^ky)>=<L(y),\psi (A^k(A^{-k}x-y))>\\
&=(L*\psi (A^k\cdot ))(A^{-k}
x)=\int_{\mathbb{R}^n}K(A^{-k}x-y)\psi (A^ky)dy,\quad A^{-k}x\not \in \mbox{supp }\psi (A^k\cdot).
\end{align*}
Then,
$$
(L_k*\psi )(x)=b^{-k}\int_{\mathbb{R}^n}K(A^{-k}(x-y))\psi (y)dy,\quad x\not \in \mbox{supp }\psi .
$$

We are going to see that, there exists $C>0$ that does not depend on $\ell$ such that:

$(i)$ $|(L_{-\ell }*\phi )(x)|\leq C/\rho (x)$, $x\in \mathbb{R}^n\setminus\{0\}$,

\noindent and, if $\delta=\min\{\gamma , \ln \lambda _-/\ln b\}$,

$(ii)$ $|(L_{-\ell }*\phi )(x-y)-(L_{-\ell }*\phi )(x)|\leq C\rho (y)^\delta/(\rho (x))^{\delta +1}$, when $b^{2w}\rho (y)\leq \rho (x)$.

Firstly we prove $(i)$. We have that $L_{-\ell}*\phi \in L^2(\mathbb{R}^n)$ and $\widehat{L_{-\ell }*\phi }=\widehat{L_{-\ell }}\widehat{\phi }\in L^1(\mathbb{R}^n)$. Then, we can write
$$
(L_{-\ell}*\phi )(x)=\int_{\mathbb{R}^n}e^{-2\pi ix\cdot y}\widehat{L_{-\ell}}(y)\widehat{\phi}(y)dy,\quad x\in \mathbb{R}^n.
$$
Note that the two sides in the last equalities define smooth functions in $\mathbb{R}^n$. Since $\|\widehat{L_{-\ell}}\|_\infty =\|\widehat{L}\|_\infty$, we deduce that
$$
|(L_{-\ell }*\phi )(x)|\leq \|\widehat{L_{-\ell}}\|_\infty\int_{\mathbb{R}^n}|\widehat{\phi}(y)|dy,\quad x\in \mathbb{R}^n.
$$

We obtain
$$
|(L_{-\ell }*\phi )(x)|\leq \frac{b^{1+w}\|\widehat{L_{-\ell}}\|_\infty}{\rho (x)}\int_{\mathbb{R}^n}|\widehat{\phi}(y)|dy,\quad x\in B_{1+w}\setminus\{0\}.
$$

On the other hand, since $\rho (x+y)\leq b^w(\rho (x)+\rho (y))$, $x,y\in \mathbb{R}^n$, we have that $\rho (x-y)\geq b^{-w}\rho (x)-\rho (y)$, $x,y\in \mathbb{R}^n$. Then, if $x\not \in B_{1+w}$ and $y\in B_0$, it follows that $\rho (x-y)\geq b^{-w}\rho (x)-b^{-w-1}\rho (x)=b^{-w}(1-b^{-1})\rho (x)$. We can write
\begin{align*}
|(L_{-\ell }*\phi )(x)|&\leq b^\ell \int_{B_0}|K(A^\ell (x-y))||\phi (y)|dy\leq C_Kb^\ell \int_{B_0}\frac{|\phi (y)|}{\rho (A^\ell (x-y))}dy\\
&\leq C_K\int_{B_0}\frac{|\phi (y)|}{\rho (x-y)}dy\leq \frac{C_Kb^w}{(1-b^{-1})\rho (x)}\int_{B_0}|\phi (y)|dy,\quad x\not \in B_{w+1}.
\end{align*}

We conclude that $|(L_{-\ell }*\phi )(x)|\leq C/\rho (x)$, $x\in \mathbb{R}^n\setminus \{0\}$, where $C>0$ is independent of $\ell$, and $(i)$ is proved.

We now establish $(ii)$. We can write
$$
(L_{-\ell }*\phi )(x-y)-(L_{-\ell }*\phi )(x)=\int_{\mathbb{R}^n}\Big(e^{-2\pi i(x-y)\cdot z}-e^{-2\pi ix\cdot z}\Big)\widehat{L_{-\ell }}(z)\widehat{\phi }(z)dz,\quad x,y\in \mathbb{R} ^n.
$$
Mean value theorem leads to
$$
|(L_{-\ell }*\phi )(x-y)-(L_{-\ell }*\phi )(x)|\leq C|y|\|\widehat{L}\|_\infty \int_{\mathbb{R}^n}|z||\widehat{\phi }(z)|dz,\quad x,y\in \mathbb{R} ^n.
$$
According to \cite[(3.3) p. 11]{Bow1} , $|y|\leq  C\rho (y)^{\ln \lambda _-/\ln b}$, when $\rho(y)\leq 1$. Also, by  \cite[(3.2) p. 11]{Bow1},  we get
$$
|y|\leq C\rho (y)^{\ln \lambda _+/\ln b}\leq C\rho (y)^{\ln \lambda _-/\ln b},\quad 1\leq \rho (y)\leq b^{4w}.
$$
Hence,
\begin{align*}
|(L_{-\ell }*\phi )(x-y)-(L_{-\ell }*\phi )(x)|&\leq C\rho (y)^{\ln \lambda _-/\ln b}\\
&\leq C\frac{\rho (y)^{\ln \lambda _-/\ln b}}{\rho (x)^{\ln \lambda _-/\ln b+1}}, \quad b^{2w}\rho (y)\leq \rho (x)\leq b^{4w}.
\end{align*}

Assume that $\rho (x)\geq b^{4w}$ and $b^{2w}\rho (y)\leq \rho (x)$. It is clear that $x\not \in \mbox{ supp}\;\phi$. Also, we have that $\rho (x-y)\geq b^{-w}\rho (x)-\rho (y)\geq b^{-w}\rho (x)-b^{-2w}\rho (x)\geq b^{3w}-b^{2w}\geq b$. Then, $x-y\not \in \rm{supp }\;\phi $. We can write
$$
(L_{-\ell }*\phi )(x-y)-(L_{-\ell }*\phi )(x)=\int_{\mathbb{R}^n}(K_{-\ell }(x-y-z)-K_{-\ell }(x-z))\phi (z)dz,
$$
where $K_{-\ell}(z)=b^\ell K(A^\ell z)$, $z\in \mathbb{R}^n$.

Suppose that $\rho (y)\leq b^{-6w}\rho (x)$ and $z\in \rm{supp}\;\phi$. Since $\rho (z)\leq b^{-4w}\rho (x)$, we have that $\rho (x-z)\geq b^{-w}\rho (x)-\rho (z)\geq b^{-w}\rho (x)-b^{-4w}\rho (x)\geq b^{6w}(b^{-w}-b^{-4w})\rho (y)=b^{2w}(b^{3w}-1)\rho (y)\geq b^{2w}\rho (y)$. We obtain
\begin{align*}
|(L_{-\ell }*\phi )(x-y)-(L_{-\ell }*\phi )(x)|&\leq \int_{B_0}|K_{-\ell }(x-y-z)-K_{-\ell }(x-z)||\phi (z)|dz\\
&\leq C\int_{B_0}|\phi (z)|\frac{\rho (y)^\gamma}{\rho (x-z)^{\gamma +1}}dz\leq C\frac{\rho (y)^\gamma }{\rho (x)^{\gamma +1}}\int_{\mathbb{R}^n}|\phi (z)|dz.
\end{align*}

Suppose now that $b^{2w}\rho (y)\leq \rho (x)\leq b^{6w}\rho (y)$. It follows that $\rho (x-y)\geq b^{-w}\rho (x)-\rho (y)\geq (b^{-w}-b^{-2w})\rho (x)$. From $(i)$ we deduce
$$
|(L_{-\ell }*\phi )(x-y)-(L_{-\ell }*\phi )(x)|\leq C\Big(\frac{1}{\rho (x-y)}+\frac{1}{\rho (x)}\Big)\leq \frac{C}{\rho (x)}\leq C\frac{\rho (y)^\gamma }{\rho (x)^{\gamma +1}}.
$$

We conclude that, if $\delta=\min\{\gamma , \ln \lambda _-/\ln b\}$,
$$
|(L_{-\ell }*\phi )(x-y)-(L_{-\ell }*\phi )(x)|\leq C\frac{\rho (y)^\delta }{\rho (x)^{\delta +1}},\quad b^{2w}\rho (y)\leq \rho (x),
$$
where $C>0$ does not depend on $\ell$, and $(ii)$ is proved.

Since $L*\phi _\ell=(L_{-\ell }*\phi)_\ell$, from $(i)$ and $(ii)$ we infer that

$(i')$ $|(L*\phi _\ell)(x)|\leq C/\rho (x)$, $x\in \mathbb{R}^n\setminus\{0\}$,

\noindent and, if $\delta=\min\{\gamma , \ln \lambda _-/\ln b\}$,

$(ii')$ $|(L*\phi _\ell)(x-y)-(L*\phi _\ell )(x)|\leq C\rho (y)^\delta/(\rho (x))^{\delta +1}$, when $b^{2w}\rho (y)\leq \rho (x)$,

\noindent and $C>0$ does not depend on $\ell$.

We are going to prove the $m$-regularity property for the kernel
$$
H_\ell (x,y)=(L*\phi_\ell )(x-y),\quad x,y\in \mathbb{R}^n.
$$

We have to show that, if $\alpha \in \mathbb{N}^n$, $|\alpha|\leq m$, and $x,y\in \mathbb{R}^n$, $x-y\in B_{k+1}\setminus B_k$, with $k\in \mathbb{Z}$, then
$$
|(\partial _y^\alpha \widetilde{H_\ell})(x,A^{-k}y)|\leq \frac{C}{\rho (x-y)}=\frac{C}{b^k},
$$
where $\widetilde{H_\ell}(x,y)=H_\ell (x, A^ky)$ and $C>0$ is independent of $\ell$. In order to prove this, we proceed as in  \cite[pp. 66 and 67]{Bow1}.

We have that
$$
H_\ell (x,y)=\int_{\mathbb{R}^n}\mathbb{K}(x-z ,y)\phi _\ell (z)dz,\quad x-y\in B_\ell ,
$$
where $\mathbb{K}(x,y)=K(x-y)$, $x,y\in \mathbb{R}^n\setminus\{0\}$.

Suppose that $x_0,y_0\in \mathbb{R}^n$ and $x_0-y_0\in B_{j+2w+1}\setminus B_{j+2w}$, where $j\in \mathbb{N}$, $j\geq \ell$. By \cite[(2.11), p. 68]{Bow1} it follows that $x_0-y_0-z\not\in B_{j+w}$ and $x_0-y_0-z\in B_{j+3w+1}$, for every $z\in B_\ell$. By using the regularity of $K$ we deduce  (see \cite[(9.29), p. 66]{Bow1}), for every $\alpha \in \mathbb{N}^n$, $|\alpha|\leq m$,
$$
|(\partial _y^\alpha[\mathbb{K}(\cdot ,A^{j+2w}\cdot)](x_0-z,A^{-j-2w}y_0)|\leq Cb^{-j-2w},\quad z\in B_\ell .
$$

Differentiating under the integral sign we get
$$
|(\partial _y^\alpha  \widetilde{H_\ell})(x_0,A^{-j-2w}y_0)|\leq Cb^{-j-2w},\quad \alpha \in \mathbb{N}^n,\;|\alpha|\leq m,
$$
where $C>0$ does not depend on $\{\ell ,j\}$.

Assume that $x_0,y_0\in \mathbb{R}^n$ and $x_0-y_0\in B_{j+1}\setminus B_j$, being $j<\ell +2w$. Let $\alpha \in \mathbb{N}^n$, $|\alpha |\leq m$. We can write
\begin{align*}
\widetilde{H_\ell}(x,y)&=\int_{\mathbb{R}^n}e^{-2\pi iz\cdot(x-A^jy)}\widehat{L}(z)\widehat{\phi_\ell}(z)dz\\
&=\int_{\mathbb{R}^n}e^{-2\pi iz\cdot(x-A^jy)}\widehat{L}(z)\widehat{\phi}((A^*)^\ell z)dz,\quad x,y\in \mathbb{R}^n,
\end{align*}
where $A^*$ denotes the adjoint matrix of $A$.

After making a change of variables we get
\begin{align*}
\widetilde{H_\ell}(x,y)&=b^{-\ell}\int_{\mathbb{R}^n}e^{-2\pi i(A^*)^{-\ell}z\cdot(x-A^jy)}\widehat{L}((A^*)^{-\ell}z)\widehat{\phi}(z)dz\\
&=b^{-\ell}\int_{\mathbb{R}^n}e^{-2\pi iz\cdot(A^{-\ell}x-A^{-\ell +j}y)}\widehat{L}((A^*)^{-\ell}z)\widehat{\phi}(z)dz,\quad x,y\in \mathbb{R}^n.
\end{align*}

Then, differentiating under the integral sign, we obtain
$$
|\partial _y^\alpha  \widetilde{H_\ell}(x,y)|\leq Cb^{-\ell}\int_{\mathbb{R}^n}|z|^\alpha |\widehat{\phi}(z)|dz,\quad x,y\in \mathbb{R}^n,\;x-y\in B_{j+1}\setminus B_j,
$$
because $j-\ell<2w$. Here $C>0$ does not depend on $\{\ell ,j\}$.

Hence,
$$
|(\partial _y^\alpha  \widetilde{H_\ell})(x_0,A^{-j}y_0)|\leq Cb^{-\ell}=Cb^{-\ell +j}b^{-j}\leq Cb^{2w}\rho (x_0-y_0)^{-1}.
$$
We conclude that there exists $C>0$ such that, for every $\alpha \in \mathbb{N}^n$, $|\alpha |\leq m$, and $x,y\in \mathbb{R}^n$, $x-y\in B_{k+1}\setminus B_k$, $k\in \mathbb{Z}$,
$$
|(\partial _y^\alpha  \widetilde{H_\ell})(x,A^{-k}y)|\leq \frac{C}{\rho (x-y)}.
$$
Thus, the proof of the property is finished.
\end{proof}

\begin{proof}[Proof of Theorem \ref{Th1.4}, $(ii)$]

Consider $r$, $p_0$ and $s$ as in Proposition \ref{Prop6.2}. Assume that $a$ is a $(p(\cdot ),q(\cdot ), r/p_0,s)$-atom associated with $x_0\in \mathbb{R}^n$ and $\ell _0\in \mathbb{Z}$, and $\nu \in \mathcal{A}_1(\mathbb{R}^n,A)\cap RH_{(r/p_0)'}(\mathbb{R}^n,A)$. We take $\varphi \in \mathcal{S}(\mathbb{R}^n)$ such that $\int \varphi (x)dx\not=0$ and ${\rm supp}\;\varphi \subset B_0$.

We can write
\begin{align*}
\|Ta\|_{H^{p_0}(\mathbb{R}^n,\nu ,A)}^{p_0}&\leq C\left(\int_{x_0+B_{\ell _0+w}}(M_\varphi ^0(Ta)(x))^{p_0}\nu (x)dx+\int_{(x_0+B_{\ell _0+w})^c}(M_\varphi ^0(Ta)(x))^{p_0}\nu (x)dx\right)\\
&=J_1+J_2.
\end{align*}
The Hardy-Littlewood maximal function satisfies Kolmogorov inequality (see \cite[p. 91]{Gr}). Then, since $M_\varphi ^0(Ta)\leq CM_{HL}(Ta)$, we get
$$
J_1\leq C\nu (x_0+B_{\ell _0+w})^{1-p_0}\left(\int_{\mathbb{R}^n}|T(a)(x)|\nu (x)dx\right)^{p_0}.
$$
Here $C=C([\nu ]_{\mathcal{A}_1(\mathbb{R}^n,A)})>0$.

By splitting the last integral in the same way we obtain
$$
\int_{\mathbb{R}^n}|Ta(x)|\nu (x)dx=\int_{x_0+B_{\ell _0+w}}|Ta(x)|\nu (x)dx+\int_{(x_0+B_{\ell _0+w})^c}|Ta(x)|\nu (x)dx.
$$
Since $T$ is a bounded operator in $L^{r/p_0}(\mathbb{R}^n)$ (see Proposition \ref{Prop6.2}) and $\nu\in RH_{(r/p_0)'}(\mathbb{R}^n,A)$, it follows that
\begin{align*}
\int_{x_0+B_{\ell _0+w}}|Ta(x)|\nu (x)dx&\leq \left(\int_{x_0+B_{\ell _0+w}}|Ta(x)|^{r/p_0}dx\right)^{p_0/r}\left(\int_{x_0+B_{\ell _0+w}}\nu (x)^{(r/p_0)'}dx\right)^{1/(r/p_0)'}\\
&\leq C\|a\|_{r/p_0}|B_{\ell _0+w}|^{1/(r/p_0)'-1}\nu (x_0+B_{\ell _0+w})\\
&\leq C\frac{\nu (x_0+B_{\ell _0})}{\|\chi _{x_0+B_{\ell _0}}\|_{p(\cdot ),q(\cdot )}}.
\end{align*}
Here $C=C([\nu ]_{\mathcal{A}_1(\mathbb{R}^n,A)}, [\nu]_{RH_{(r/p_0)'}(\mathbb{R}^n,A)})$. We have used that $\nu$ defines a doubling measure.

By proceeding as in the estimation of $I_2$ in the proof of Theorem \ref{Th1.4}  (i) we get
\begin{align*}
\int_{(x_0+B_{\ell _0+w})^c}|Ta(x)|\nu (x)dx&\leq C\frac{|B_{\ell _0}|^{\delta +1}}{\|\chi _{x_0+B_{\ell _0}}\|_{p(\cdot ),q(\cdot )}}\int_{(x_0+B_{\ell _0+w})^c}\frac{\nu (x)}{\rho (x-x_0)^{\delta +1}}dx\\
&\leq C\frac{|B_{\ell _0}|^{\delta +1}}{\|\chi _{x_0+B_{\ell _0}}\|_{p(\cdot ),q(\cdot )}}\nu (x_0+B_{\ell _0})|B_{\ell _0}|^{-(\delta +1)}\\
&= C\frac{\nu (x_0+B_{\ell _0})}{\|\chi _{x_0+B_{\ell _0}}\|_{p(\cdot ),q(\cdot )}},
\end{align*}
where $C=C([\nu ]_{\mathcal{A}_1(\mathbb{R}^n,A)})>0$ and $\delta =m\ln \lambda _-/\ln b$.

We conclude that
$$
J_1\leq C\frac{\nu (x_0+B_{\ell _0})}{\|\chi _{x_0+B_{\ell _0}}\|_{p(\cdot ),q(\cdot )}^{p_0}}.
$$
Here $C=C([\nu ]_{\mathcal{A}_1(\mathbb{R}^n,A)}, [\nu]_{RH_{(r/p_0)'}(\mathbb{R}^n,A)})>0$.

According to Lemma \ref{Lem6.2}, for every $k\in \mathbb{Z}$, the convolution operator $S_k$ defined by
$$
S_{(k)}(\psi )=\varphi _k*(T\psi),\quad \psi \in \mathcal{S}(\mathbb{R}^n),
$$
is a Calder\'on-Zygmund singular integral of order $m$ and this property is uniformly in $k\in \mathbb{Z}$, that is, the characteristic constant is not depending on $k$.

If $k\in \mathbb{Z}$, by proceeding as in the proof of Theorem \ref{Th1.4} $(i)$ we get
$$
|S_{(k)}(a)|\leq C\frac{|B_{\ell _0}|^{\delta +1}}{\rho (x-x_0)^{\delta +1}}\frac{1}{\|\chi _{x_0+B_{\ell _0}}\|_{p(\cdot ),q(\cdot )}},\quad x\not \in x_0+B_{\ell _0+w},
$$
where $C>0$ does not depend on $k$.

Then,
$$
|M_\varphi ^0(Ta)(x)|\leq C\frac{|B_{\ell _0}|^{\delta +1}}{\rho (x-x_0)^{\delta +1}}\frac{1}{\|\chi _{x_0+B_{\ell _0}}\|_{p(\cdot ),q(\cdot )}},\quad x\not \in x_0+B_{\ell _0+w}.
$$
We conclude that
$$
J_2\leq C([\nu ]_{A_1(\mathbb{R}^n,A)})\frac{\nu (x_0+B_{\ell _0})}{\|\chi _{x_0+B_{\ell _0}}\|_{p(\cdot ),q(\cdot )}^{p_0}}.
$$
The proof of this theorem can be finished putting together the above estimates.
\end{proof}

\begin{Rem} In order to prove the boundedness of an operator $T$ defined on Hardy type spaces (or finite atomic Hardy type spaces) and which takes values in a Banach (or quasi-Banach) space it is usual to put the condition: $T$ is uniformly bounded on atoms. As it can be seen (for instance, in \cite[pp. 3096 and 3097]{BLYZ}, the last condition implies the boundedness of $T$, roughly speaking, proceeding as follows: if $f=\sum_{j=1}^k\lambda _ja_j$, then
\begin{equation}\label{R1}
\|Tf\|_X\leq \sum_{j=1}^k|\lambda _j|\|Ta_j\|_X\leq C\sum_{j=1}^k|\lambda _j|\leq C\|f\|_{H^p}.
\end{equation}
In our case, for the anisotropic Hardy-Lorentz spaces with variable exponents, we do not know if the last inequality in (\ref{R1}) holds. In Theorem \ref{Th1.2}  we establish our atomic quasinorm. The condition in Proposition \ref{Prop6.2} is adapted to the quasinorms on the anisotropic Hardy-Lorentz spaces with variable exponents and they replace the uniform boundedness on atoms condition.
\end{Rem}

\begin{Rem} As it is well known, Lorentz and Hardy-Lorentz spaces appear related with interpolation. C. Fefferman, Rivi\`ere and Sagher (\cite{FRS}) proved that, if $0<p_0<1$, then $(H^{p_0}(\mathbb{R}^n), L^\infty (\mathbb{R}^n))_{\eta ,q}=H^{p,q}(\mathbb{R}^n)$, where $1/p=(1-\eta)/p_0$, $0<\eta <1$ and $0<q\leq \infty$. Recently, Liu, Yang and Yuan (\cite[Lemma 6.3]{LYY})) established an anisotropic version of this result. By using a reiteration argument in \cite[Theorem 6.1]{LYY} the interpolation spaces between anisotropic Hardy spaces are described. Kempka and Vyb\'{\i}ral (\cite[Theorem 8]{KV}) proved that $(L^{p(\cdot )}(\mathbb{R}^n),L^\infty (\mathbb{R}^n))_{\theta ,q}=L_{\tilde{p}(\cdot),q}$, where $0<\theta <1$, $0<q\leq \infty$ and $1/\tilde{p}(\cdot)=(1-\theta)/p(\cdot)$. It is clear that a similar property can not be hoped for the Lorentz space $\mathcal{L}^{p(\cdot ),q(\cdot)}(\mathbb{R}^n)$, since in the definition of $L^{p(\cdot )}(\mathbb{R}^n)$, $p$ is a measurable function defined in $\mathbb{R}^n$ while in the definition of the Lorentz space $\mathcal{L}^{p(\cdot ), q(\cdot )}(\mathbb{R}^n)$, $p$ and $q$ are measurable functions defined in $(0,\infty)$. Then, the arguments used in \cite{LYY} to study interpolation in anisotropic Hardy spaces do not work in our variable exponent setting. New arguments will must be developed in order to describe interpolation spaces between our anisotropic Hardy-Lorentz spaces with variable exponents.
\end{Rem}




\end{document}